\UseRawInputEncoding
\documentclass[12pt]{article}

\usepackage{amssymb, amsfonts, amsmath, authblk}
\usepackage{amsthm}
\usepackage{tikz-cd} 
\usepackage{graphicx}
\usepackage{hyperref}
\usepackage{todonotes}
\usepackage{gensymb}
\newcommand{\C}{\mathbb{C}}
\newcommand{\R}{\mathbb{R}}
\newcommand{\Z}{\mathbb{Z}}
\newcommand{\D}{\mathbb{D}}
\newcommand{\N}{\mathbb{N}}

\newcommand{\M}{\mathcal{M}}

\newcommand{\cD}{\mathcal{D}}
\newcommand{\cO}{\mathcal{O}}

\newcommand{\cA}{\mathcal{A}}
\newcommand{\ind}{\mathrm{ind}}
\newcommand{\coker}{\mathrm{coker}}

\newcommand{\sign}{\mathrm{sign}}

\newcommand{\Aut}{\mathrm{Aut}}
\newcommand{\im}{\mathrm{Im}}
\newcommand{\ii}{\mathrm{i}}
\newcommand{\jj}{\mathrm{j}}

\newcommand{\id}{\mathrm{id}}
\newcommand{\area}{\mathrm{area}}

\newcommand\inner[2]{\langle #1, #2 \rangle}
\newcommand\innerb[2]{\bigg\langle #1, #2 \bigg\rangle}
\newcommand{\delbar}{\bar{\partial}}%delbar operator$
\newcommand{\defeq}{\mathrel{\mathop:}=}%define equal to %
\theoremstyle{plain}
\newtheorem{thm}{Theorem}[section]
\newtheorem{lem}[thm]{Lemma}
\newtheorem{cor}[thm]{Corollary}
\newtheorem{clm}[thm]{Claim}
\newtheorem{prop}[thm]{Proposition}

\theoremstyle{definition}
\newtheorem{defn}[thm]{Definition}

\theoremstyle{remark}
\newtheorem{rem}{Remark}

\newcommand{\wt}{\widetilde}

\newcommand{\ol}{\overline}

\newcommand{\p}{\partial}

\newcommand{\eps}{\varepsilon}

\newcommand{\vol}{\mathrm{vol}}

\newtheoremstyle{TheoremNum}
        {\topsep}{\topsep}              %%% space between body and thm
        {\itshape}                      %%% Thm body font
        {}                              %%% Indent amount (empty = no indent)
        {\bfseries}                     %%% Thm head font
        {.}                             %%% Punctuation after thm head
        { }                             %%% Space after thm head
        {\thmname{#1}\thmnote{ \bfseries #3}}%%% Thm head spec
    \theoremstyle{TheoremNum}

\newtheoremstyle{DefinitionNum}
        {\topsep}{\topsep}              %%% space between body and thm
        {}                      %%% Thm body font
        {}                              %%% Indent amount (empty = no indent)
        {\bfseries}                     %%% Thm head font
        {.}                             %%% Punctuation after thm head
        { }                             %%% Space after thm head
        {\thmname{#1}\thmnote{ \bfseries #3}}%%% Thm head spec
    \theoremstyle{DefinitionNum}

\begin{document}
\author{Ipsita Datta\thanks{School of Mathematics, Institute for Advanced Study, Princeton, NJ 08540, USA.
ipsi@ias.edu}}

\title{Lagrangian Cobordisms between Enriched Knot Diagrams}

\date{}
\maketitle 

\abstract
In this paper, we present new obstructions to the existence of Lagrangian cobordisms in $\R^4$ that depend only on the enriched knot diagrams of the boundary knots or links, using holomorphic curve techniques. We define enriched knot diagrams for generic smooth links. The existence of Lagrangian cobordisms gives a well-defined transitive relation on equivalence classes of enriched knot diagrams that is a strict partial order when restricted to exact enriched knot diagrams
To establish obstructions we study $1$-dimensional moduli spaces of holomorphic disks with corners that have boundary on Lagrangian tangles - an appropriate immersed Lagrangian closely related to embedded Lagrangian cobordisms. We adapt existing techniques to prove compactness and transversality, and compute dimensions of these moduli spaces. We produce obstructions as a consequence of characterizing all boundary points of such moduli spaces. We use these obstructions to recover and extend results about ``growing" and ``shrinking" Lagrangian slices.
We hope that this investigation will open up new directions in studying Lagrangian surfaces in $\R^4$.

\tableofcontents

\section*{Acknowledgements}
\addcontentsline{toc}{section}{Acknowledgements}
I thank my dissertation advisor Yasha Eliashberg for introducing me to the subject, suggesting these questions and many hours of discussions. 
I thank Laurent C\^{o}t\'{e}, Fran\c{c}ois-Simon Fauteux-Chapleau, Joj Helfer, and Yuan Yao for useful discussions. Thanks to Josh Sabloff for reading details. Thanks to the anonymous referee for their numerous comments and suggestions that led to the present exposition.

I thank Department of Mathematics at Stanford University, Maryam Mirzakhani Graduate Fellowship, the Institute for Advanced Study, US National Science Foundation Grants DMS-1807270 and DMS-1926686, for their support.

\section{Introduction}\label{sec_intro}
Lagrangians are interesting and important objects in symplectic geometry. It is known that they display both flexible and rigid properties. 
A generic Lagrangian surface in $\R^4$ intersects parallel hyperplanes, $\R^3_a\defeq \{y_2 = a\}$ and $\R^3_b$, transversely. So, $L \cap \R^4_{[a,b]} \defeq \{a \leq y_2 \leq b\}$ gives a cobordism from $\p_- L = L \cap \R^3_a$ to $\p_+ L = L \cap \R^3_b$ that is Lagrangian.
\begin{figure}[h!]
\begin{center}
\includegraphics[width = 5in]{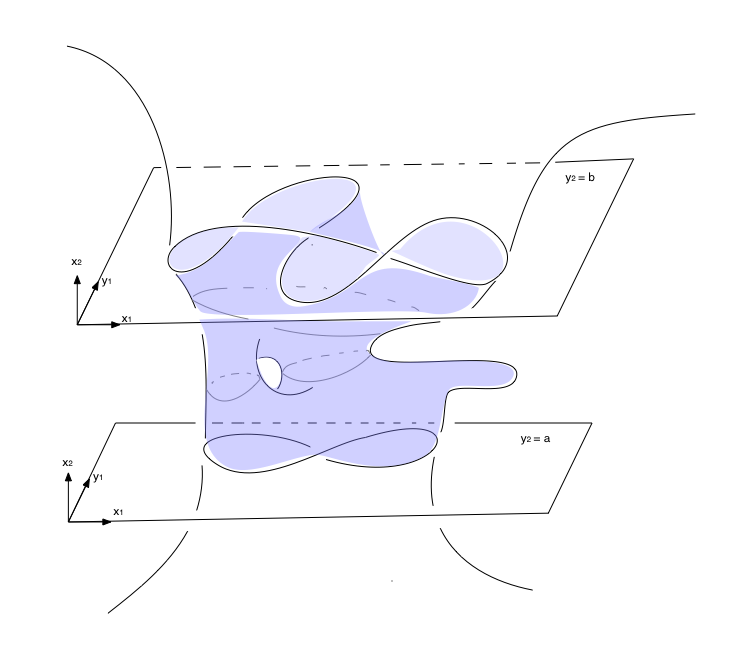}
 \end{center}
 \caption[caption]
    {\tabular[t]{@{}l@{}}Lagrangian cobordism is a natural way of viewing Lagrangian surfaces\\ by restricting our attention to $\R^4_{[a,b]}$.\endtabular}
  \label{fig:cut-lagrangian}
\end{figure}

Lagrangian cobordism can refer to one of many similar objects. Lagrangian cobordisms were first studied by Arnold \cite{arn1}.
For a usual cobordism, $\beta$, between manifolds $B_1$ and $B_0$, that is, $\p \beta = B_1 - B_0$,  Arnold referred to Lagrangian submanifolds $\lambda$ in $T^* \beta$ with Lagrange border $\underline \p \lambda = \pi(\p L \cap \p T^* \beta) = L_1 - L_0$, where $\pi : \p T^* \beta \to T^* \p \beta$ is the standard projection, as a \textbf{Lagrange cobordism over} {\boldmath$\beta$} between Lagrangian manifolds $L_1 \subset T^* B_1$ and $L_0$ in $T^* B_0$. Turns out
Lagrange cobordisms in $\beta = B \times [0,1]$ give an equivalence relation on the set of immersed Lagrangians in $B$.  

Eliashberg (\cite{eliash_cobordisme}) showed that immersed Lagrangian cobordisms obey an $h$-principle, that is, their existence is assured upto vanishing of algebraic topological invariants. 
This is referred to as a ``flexible" phenomenon in symplectic geometry.

Embedded Lagrangians are more ``rigid." For example, embedded Lagrangians in $\R^2$ are only circles, and it can be shown that all embedded Lagrangian cobordisms between circles in $\R^2$ must be between circles of equal area. 
Biran and Cornea show that monotone Lagrangians cobordisms in $M \times \R^2$, for symplectic manifold $M$, between embedded Lagrangians in $M$ preserve Floer homology and similar invariants. 

What we are considering is somewhere in between. Like in \cite{ST}, we consider cobordisms that are embedded in $\R^4$ but we allow the Lagrangian projection of the boundaries to $\R^2$ to be immersed.
Sabloff and Traynor used generating functions to define capacities,
$
c_{\pm}^{L,a} : H^* (L_a) \to (-\infty, 0], C_{\pm}^{L,a} : H^* (L_a) \to [0, \infty),
$
 for unknotted, planar, flat-at-infinity Lagrangian cobordisms in $\R^{2n}$ \cite{ST}.
Slices of such Lagrangians in $\R^4$ (upto compactly supported area preserving diffeomorphisms of $\R^2$) have a partial order defined by the existence of a Lagrangian cobordism in $\R^4$.

In this paper we consider relatively exact Lagrangian cobordisms in $\R^4$. Relative exactness is a generalization of exactness - A Lagrangian cobordism is said to be {\bf relatively exact} if the image of the function $\omega: H_2(\R^4, L; \Z) \to \R$ is contained in the set of $\Z$-multiples of $a(\p_- L) = |\int_{\p_- L} x_1 dy_1|$. In fact, relative exactness follows from topological considerations of the cobordism (see Remark \ref{rem_automatic relative exactness}). This means that, in some cases, obstructing relatively exact Lagrangian cobordisms, obstructs the existence of any Lagrangian cobordisms with those boundaries. We find obstructions to the existence of relatively exact Lagrangian cobordisms between knots/links based on the combinatorial data of enriched knot diagrams. Intuitively, the enriched knot diagram for a knot, $K$, keeps track of the topological of the knot like a standard knot diagram, and additionally records the geometric data of areas of different regions in $\R^2$ defined by $\pi_1(K)$ for the standard projection $\pi_1: \R^3 \to \R^2$ forgetting the last coordinate (see Definition \ref{defn_diagram}).
\begin{defn}\label{defn_undercut}
Given two diagram classes, $[\cD_1, \sigma_1, \cA_1]$ and $[\cD_2, \sigma_2, \cA_2]$, we say $[\cD_1, \sigma_1, \cA_1]$ \textbf{undercuts} $[\cD_2, \sigma_2, \cA_2]$ if there exists a relatively exact Lagrangian cobordism (Definition \ref{defn_lag cobordism}), $L$, such that $K_{\cD_1} \prec_L K_{\cD_2}$, for links $K_{\cD_j}$ having diagram $[\cD_j, \sigma_j, \cA_j]$, $j = 1,2$. We denote this by $[\cD_1, \sigma_1, \cA_1] \prec [\cD_2, \sigma_2, \cA_2]$ (or $[\cD_1] \prec [\cD_2]$).
\end{defn}

\begin{lem}\label{lem_partial order}
The undercut relation is a transitive relation on equivalence classes of enriched knot diagrams, that is,
\begin{enumerate}
\item the relation $\prec$ is well-defined on equivalence classes of diagrams;
\item undercutting is transitive: if $[\cD_3] \prec [\cD_2]$ and $[\cD_2] \prec [\cD_1]$, then $[\cD_3] \prec [\cD_1]$.
\end{enumerate}
Additionally, the undercut relation gives a strict partial order when restricted to exact diagrams (see Definition \ref{defn_exact diagram}), that is,
\begin{enumerate}
\item undercutting is non-reflexive: for all exact diagrams $[\cD]$, $[\cD] \nprec [\cD]$;
\item undercutting is anti-symmetric: for $[\cD_1] \neq [\cD_2]$, $[\cD_1] \prec [\cD_2]$ implies $[\cD_2] \nprec [\cD_1]$;
\item not all enriched knot diagrams are related by the undercut relation.
\end{enumerate}
\end{lem}

Given an enriched knot diagram, $\cD$, 
and two disks, $A$ and $B$, bound by $\cD$, (Definition \ref{defn_bound disk}), we say $A$ and $B$ \textbf{share a corner} $q \in X(\cD)$ if $q$ is a corner for both the disks $A$ and $B$ and both $A$ and $B$ have the same sign at $q$.
\begin{defn}\label{defn_big little disk}
A disk, $A$, bound by a diagram pair $((\cD_1,\sigma_1, \cA_1),( \cD_2, \sigma_2, \cA_2))$ is called \textbf{big} if:
\begin{enumerate}
\item $A$ has all convex corners;
\item $A$ is bound by $\cD_1$ and has all negative corners, or $A$ is bound by $\cD_2$ and has all positive corners;
\item $A$ is an aligned disk as in Definition \ref{defn_aligned disk};
\item either $a(D_1) = a(D_2) = 0$ or $\area(A) \leq a(D_1)$.
\end{enumerate} 

Given a diagram pair $((\cD_1,\sigma_1, \cA_1),( \cD_2, \sigma_2, \cA_2))$ and a big disk, $A$, bound by it, we define a {\bf little disk relative to ${\bf A}$} to be a disk, $B$, bound by the pair $((\cD_1,\sigma_1, \cA_1),( \cD_2, \sigma_2, \cA_2))$, distinct from $A$, such that one of the following conditions hold:
\begin{enumerate} 
\item[(a)] $B$ and $A$ are both bound by $\mathcal{D}_2$ and $B$ has sign equal to $+1$ only at those corner points it shares with $A$. The rest are negative corners. If $B$ has no negative corners, $B$ must share all of $A$'s corners;
\item[(b)] $B$ and $A$ are both bound by $\mathcal{D}_1$ and $B$ has sign equal to $-1$ only at those corner points it shares with $A$. The rest are positive corners. If $B$ has no positive corners, $B$ must share all of $A$'s corners;
\item[(c)] $A$ is bound by $\cD_2$; $B$ is bound by $\cD_1$ and has all positive corners , with at least one corner when  $A$ has non-zero number of corners; 
\item[(d)] $A$ is bound by $\cD_1$; $B$ is bound by $\cD_2$ and has all negative corners, with at least one corner when $A$ has non-zero number of corners.
\end{enumerate}
\end{defn}
We show that the existence of a big disk implies that of a little disk whenever an enriched knot diagram, $\cD_1$, undercuts another enriched knot diagram, $\cD_2$. 
\begin{thm}\label{thm_main application}
If $[\cD_1,\sigma_1, \cA_1]\prec [ \cD_2, \sigma_2, \cA_2]$ and there exists a big disk $A$ bound by the pair, then there must exist a little disk $B$ relative to $A$ bound by the diagram pair, such that
\begin{align*}
\mathrm{area}(A) \geq \mathrm{area} (B).
\end{align*}
Equality of area can hold only when $[A] = [B] \in \Pi_2(\R^4, L)$ for any possible Lagrangian tangle $L$ with $\cD_{\p_+ L} = \cD_2$ and $\cD_{\p_- L} = \cD_1$ and they share all corners. 
\end{thm}
Note that even though the condition on equality a priori feels like it depends on more than just the diagrams, as we can determine the topology of the tangle using the diagrams from Equation (\ref{eqn_euler char from writhe}), this condition actually depends only on the diagrams.

\begin{figure}[h!] 
\begin{center}
  \includegraphics[width = 5in]{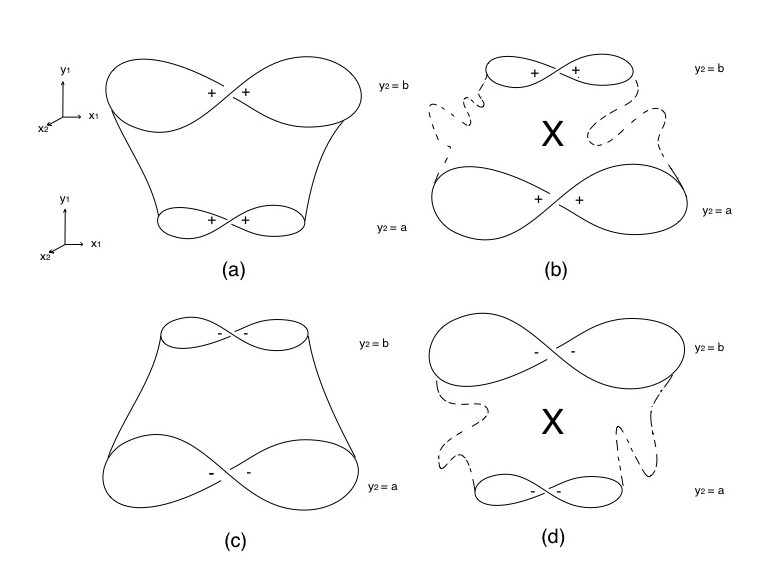}
   \end{center}
 \caption[caption]{\tabular[t]{@{}l@{}}Parts (a) and (c) show for $r < R$ possible Lagrangian cobordisms \\ $8_+(r) \prec 8_+(R)$ and $8_-(R) \prec 8_-(r)$, respectively. \\There are no Lagrangian cobordisms when the orders are reversed\\ as in (b) and (d). This figure appeared in \cite{ST}. \endtabular}
 \label{fig:8knots}
\end{figure}
The motivating question or observation that propelled this research is as follows. Consider the following $8$-shaped curves in $\R^3$ : $8_{\pm}(\pi r^2) = \{(x_1, y_1, x_2)| x_1^2 + x_2^2 = r^2, y_1 = \pm 2 x_1x_2\}$. It is possible to construct Lagrangian cobordisms $L \subset \R^4_{[a,b]}$ such that $\p_+ L = 8_+ (A)$ and $\p_- L = 8_+ (B)$ for $A>B$, but it was shown in \cite{ST} that this is impossible for $A<B$. Similarly, it is possible to construct $L \subset \R^4_{[a,b]}$ such that $\p_+ L = 8_- (A)$ and $\p_- L = 8_- (B)$ for $A<B$, but impossible for $A>B$. 

We are able to recover these results using our methods. Additionally, we are able to show similar results for knots like trefoils with positive or negative crossings, $T_\pm$, which are beyond the scope of the techniques in \cite{ST}. These are presented as corollaries of Theorem \ref{thm_main application} in Section \ref{sec_applications}. Many
similar results can be obtained by applying Theorem \ref{thm_main application} and we only present select few.

 To establish the obstruction, Theorem \ref{thm_main application}, we study moduli spaces on holomorphic disks with corners that have boundary on ``Lagrangian tangles" which are immersed Lagrangian cobordisms that have boundary on copies of $\R^2$ of the form $\R^2_a \defeq \{x_2 = 0, y_2 = a\}$ such that the Lagrangian is embedded away from its boundary. 
\begin{defn}\label{defn_lag tangle}
A \textbf{Lagrangian tangle}, $L$, is a compact, connected, oriented, immersed Lagrangian submanifold (with boundary) such that:
\begin{itemize}
\item $L \subset \R^4_{[a,b]} \defeq \{a \leq y_2 \leq b \} \subset \R^4$ for some $a < b \in \R$;
\item $L$ intersects $\R^3_a \cup \R^3_b$ transversely, and $\p L = L \cap (\R^3_a \cup \R^3_b)$;
\item $\p L$ is \textbf{flat}, that is, $\p L \subset \R^2_a \cup \R^2_b$;
\item the only self-intersections of $L$ are transverse double points on $\p L$;
\item $L$ intersects $\R^3_t$ transversely for $b-\epsilon < t \leq b$ and for $a \leq t < a+ \epsilon$, for some small $\epsilon > 0$;
\item $L \cap \R^4_{[a+\epsilon, b-\epsilon]}$ is relatively exact as defined in Definition \ref{defn_rel exact} for all $\epsilon > 0$.
\end{itemize}
To show that one can go between (embedded) relatively exact Lagrangian cobordisms and Lagrangian tangles as required, we include constructions of Lagrangian collars. Some of the constructions rely on Lagrangian movies perspective which was utilized very effectively in \cite{Sauvaget2004} and \cite{lin} previously.
\end{defn}

Theorem \ref{thm_main application} is proved in Section \ref{sec_applications} using Theorem \ref{thm_properties of type 2 disks}. Theorem \ref{thm_properties of type 2 disks} is our main technical result where we completely classify boundary points of $1$-dimensional moduli spaces of holomorphic disks with corners that have boundary on a Lagrangian tangle. The proof of Theorem \ref{thm_properties of type 2 disks} spans the entirety of Section \ref{sec_boundary points}.
\begin{figure}[h!]
 \begin{center}
  \includegraphics[width = 6in]{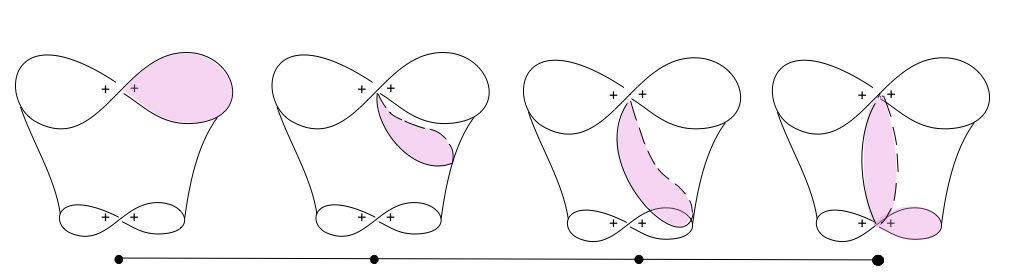}
  \end{center}
 \caption[caption]{\tabular[t]{@{}l@{}}Depiction of a holomorphic disk with corners that has boundary on a  \\Lagrangian tangle moving in a one dimensional moduli space.\endtabular}
  \label{fig:proof summary}
\end{figure}

In \cite{ST}, Sabloff and Traynor study the Morse theory of generating functions of unknotted and planar-at-infinity Lagrangians. An embedded submanifold, $L \subset \R^4 = T^* \R^2$ is ``planar" if it is diffeomorphic to the zero section, $L_0 = \{y_1= y_2 = 0\}$. A planar Lagrangian is ``flat-at-infinity" if it agrees with $L_0$ outside of a compact subset of $\R^4$. These conditions mean that the Lagrangian cobordism has genus $0$ and the writhe (with respect to the blackboard framing) of the knots that can appear as Lagrangian slices has to be $\pm 1$. Relaxing the ``global" conditions of unknottedness and planarity implies we are able to prove results for a larger range of knots. In particular, we allow writhes to be any integer, and do not require the knots to be capped or filled by a Lagrangian disk.
\begin{figure}[h!]
\begin{center}
  \includegraphics[width = 6in]{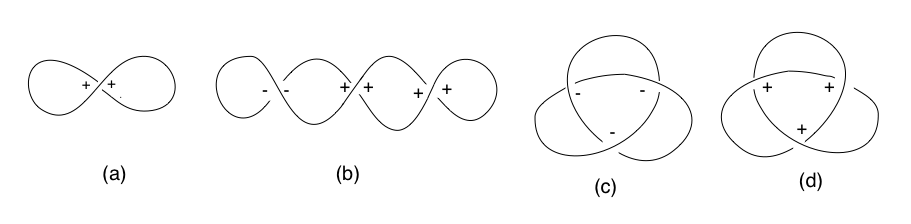}
   \end{center}
 \caption{(a) $\mathrm{wr} = 1$, (b) $\mathrm{wr} = 1$, (c) $\mathrm{wr} = -3$, (d) $\mathrm{wr} = 3$.}
  \label{fig:writhe diagrams}
\end{figure}

A related, well-studied, setup is to consider Lagrangians in the symplectization of a contact manifold. For example, the contact manifold $\R^3$ with contact form $\alpha = dz - y dx$ has symplectization $\R^4_+ = \R_+ \times \R^3$ with the symplectic form $d(e^t\alpha)$. Here, one studies Lagrangians that are cylindrical over Legendrians at $t = \infty$ and $0$. With the change of coordinates $x_1 = x, y_1 = e^t y, x_2 = e^t, y_2 = z$, we can view the symplectization as $\R^4$ with the standard symplectic form. In these new coordinates the Lagrangians would be conical over Legendrians at $\pm \infty$. This situation is a special case of the setup that we consider in this paper. 
An enriched knot diagram, $\cD$, has a Legendrian representative if and only if
\begin{itemize}
\item the total area, $\int_\cD x_1d_1$, bound by the diagram is zero, and
\item at each crossing, the sign of a corner is the same as the sign of the area bound by the arc that starts and ends at that crossing and contains that corner. (Here we consider the parametrization on the arc that is anti-clockwise along the boundary of the considered corner.)
\end{itemize}
\begin{figure}[h!]
\begin{center}
  \includegraphics[width = 5in]{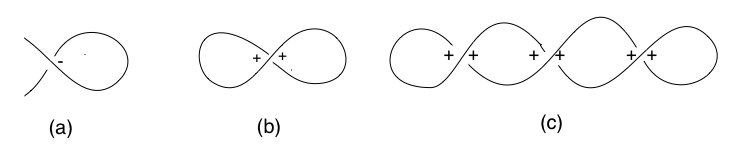}
   \end{center}
 \caption[caption]{\tabular[t]{@{}l@{}}(a) This loop cannot appear in an enriched knot diagram that has\\ \quad\quad  a Legendrian representative.\\ (b),(c) Diagrams with Legendrian representatives. \endtabular}
  \label{fig:Legendrian diagrams}
\end{figure}
This means that the $8_+$ enriched knot diagrams have Legendrian representatives whereas the $8_-$ enriched knot diagrams don't. We remark here that the signs look opposite to more common conventions (for example, those of Legendrian contact homology in \cite{}) because we slice by $y_2 = z=$ constant, whereas, in the above described coordinates, the more common convention is to slice by $t =$ constant which is equivalent to $x_2 =$ constant.

\section{Preliminaries}\label{sec_prelim}
\subsection{Setup}
Consider the standard $\R^4$ with coordinates $x_1, y_1, x_2, y_2$, with the standard symplectic form, 
$
\omega = dx_1 \wedge dy_1 + dx_2 \wedge dy_2.
$
The complex structure obtained by identifying $\R^4$ with $\C \oplus \C$, given by 
$
\ii(\p_{x_j}) = \p_{ y_j}$ for $j = 1,2,
$
is compatible with $\omega$. The pair $(\omega, \ii)$ generate the standard Riemannian metric on $\R^4$,
$
g = dx_1 \otimes dx_1 + dy_1 \otimes dy_1 + dx_2 \otimes dx_2 + dy_2 \otimes dy_2 = \omega( \cdot, \ii \cdot).
$
Let us fix notation for some special subsets of $\R^4$:
\begin{itemize}
\item $\R^4_{I} \defeq \{(x_1, y_1, x_2, y_2) \in \R^4 | \, y_2 \in I\}$ is a restricted part of $\R^4$ for any interval (open or closed) $I \subset \R$;
\item $\R^3_a \defeq \{(x_1, y_1, x_2, y_2)| \, y_2 = a\}$ is a hyperplane for any $a \in \R$;
\item $\R^2_a \defeq \{(x_1, y_1, x_2, y_2)| \, y_2 = a, x_2 = 0\}$ is a complex plane contained in $\R^3_a$ for any $a \in \R$.
\end{itemize} 
A \textbf{ Lagrangian submanifold}, $L$, of $\R^4$ is a half-dimensional submanifold with or without boundary, on which the symplectic form vanishes, that is, $\omega|_{TL} \equiv 0$. In case of $\R^4$, hal-dimensional means a real $2$ dimensional manifold, so $L$ is a surface. We assume that all our manifolds are smooth unless mentioned otherwise. For any Lagrangian surface $L \subset \R^4$ we denote by $L_a$ the \textbf{slice} of $L$, 
\begin{align*}
L_a \defeq L \cap \R^3_a
\end{align*}
whenever $L$ intersects $\R^3_a$ transversely.

A \textbf{knot} is a smooth embedding, $S^1 \hookrightarrow \R^3$. A \textbf{link} is a smooth embedding,  $\sqcup_{j=1}^k S^1 \hookrightarrow \R^3$, of finite number of copies of $S^1$.
Denote by $\pi_1$ the projection $\pi_1 : \R^3 \to \R^2$, $(x_1, y_1, x_2) \mapsto (x_1, y_1)$. The projection $\pi_1$ is sometimes referred to as the Lagrangian projection of $\R^3$.
For a generic oriented link $K$ in $\R^3$, $\pi_1(K)$ is an immersed curve with finitely many transverse double points, called \textbf{crossings}. We assume that all links, $K$, in this paper have such $\pi_1(K)$.

\begin{defn}\label{defn_lag cobordism}
A \textbf{Lagrangian cobordism}, $L$, from a link $K_1$ to a link $K_2$  is an embedded, orientable, connected,  compact, Lagrangian surface $L \subset \R^4_{[a,b]}$ for some real numbers $a<b$, such that 
\begin{itemize}
\item $L$ intersects $\R^3_a \defeq \{y_2 = a\}$ and $\R^3_b \defeq \{y_2 = b\}$ transversely, and
\item $L \cap \R^3_b = \iota_b(K_2)$, $L \cap \R^3_a = \iota_a(K_1)$ for the canonical inclusion $\iota_c : \R^3 \to \R^4$ given by $\iota_c(x_1, y_1, x_2) = (x_1, y_1, x_2, c)$ for $c \in \R$.
\end{itemize}
If such a Lagrangian cobordism exists, we say that $K_1$ {\bf undercuts} $K_2$, and denote it by by $K_1 \prec K_2$.
We write $K_1 \prec_L K_2$ if we want to specify that the specific Lagrangian $L$ gives this cobordism.

A knot $K$ is said to be \textbf{capped} by a Lagrangian if $K \prec \emptyset$, and \textbf{filled} by a Lagrangian if $\emptyset \prec K$.
\end{defn}
For a Lagrangian cobordism $L \subset \R^4_{[a,b]}$, let us denote $$\p_+ L \defeq L \cap \R^3_b \text{ and } \p_-L \defeq L \cap \R^3_a.$$ As $L$ is oriented, $L$ induces an orientation on its boundaries $\p_\pm L$. So, any Lagrangian cobordism is a Lagrangian cobordism from $\p_- L$ to $\p_+ L$.

Define \textbf{area of a knot} $K$ to be
\begin{align*}
a(K) \defeq \left|\int_{\pi_1(K)} x_1 dy_1 \right|.
\end{align*}
Note that as we take absolute value, in the case $K$ is a knot, the specified orientation does not matter. If $K$ had multiple connected components, the orientation would be important.
By Stoke's theorem, a Lagrangian cobordism from knot $K_1$ to $K_2$ cannot exists unless $a(K_1) = a(K_2)$. 

\begin{defn}\label{defn_rel exact}
Recall that the symplectic form $\omega \in H^2(\R^4, L)$ defines a map $\omega: \Pi_2(\R^4, L) \to \R$.
We call a Lagrangian cobordism \textbf{relatively exact} if the image $\omega(\Pi_2(\R^4, L))$ is $a \Z$ where $ a = a (\p_+ L ) = a(\p_- L)$. 
\end{defn}
\begin{rem}\label{rem_automatic relative exactness}
In some cases, we can conclude that a Lagrangian cobordism is relatively exact purely from topological constraints. Suppose $L$ is a Lagrangian cobordism from a knot $K_1$ to $K_2$ such that the writhes (with respect to the blackboard framing in $x_1, y_1, x_2$ coordinates) satisfy $\mathrm{wr}(K_1) = \mathrm{wr}(K_2)$. As $L$ is Lagrangian, one can show that the Euler characteristic satisfies
\begin{align*}
\chi(L) = \mathrm{wr}(K_2) - \mathrm{wr}(K_1),
\end{align*}
and hence $\chi(L) = 0$. As we know that $L$ is topologically a surface with at least two boundary components, zero Euler characteristic implies
\begin{align*}
0 = \chi(L) = 2 - 2g - b = -2g,
\end{align*}
that is, genus $g  = 0$ and $b=2$. This forces $L$ to be topologically a cylinder and $\Pi_2(\R^4, L)$ is generated by a class $[u]$ for $u$ a disk that has boundary equal to one of the boundary components of $L$. Thus, $L$ is automatically relatively exact.
\end{rem}

To study relatively exact Lagrangian cobordisms we consider related immersed Lagrangians called Lagrangian tangles (Definition \ref{defn_lag tangle}). An {\bf immersed Lagrangian} in $(\R^4, \omega)$ is the image, $L = \iota(\wt L)$, of an immersion
\begin{align*}
\iota: \wt L \to \R^4,
\end{align*}
where $\wt L$ is a $2$-manifold (possibly with boundary), such that the pull-back $\iota^* \omega \equiv 0$ vanishes. Given such an immersed Lagrangian, $L$, its {\bf transverse double points} are $q \in L$ such that there exists $p_1 \neq p_2$ in $\wt L$ with
 $\iota(p_1) = \iota(p_2) = q$, and $d\iota_{p_1} (T_{p_1} L) \oplus d \iota_{p_2} (T_{p_2} L) = T_q \R^4$. Denote the set of transverse double points of an immersed Lagrangian, $L$, by $$\Delta(L) \defeq \{q \in L |\, \#\iota^{-1} \{q\} = 2\}.$$

A Lagrangian tangle (Definition \ref{defn_lag tangle}), $L$, has two leaves near any point $q \in \Delta(L)$, that is, in a small neighbourhood of $q$, $L \setminus \{q\}$ has two connected path components. Note that, here we are using the assumption that the only singularities are transverse double points. Near each double point, for any fixed $y_2$-value, one leaf has higher $x_2$-values than the other. We refer to the former as the \textbf{higher leaf}, $L^h$, and the latter as the \textbf{lower leaf}, $L^l$, near $q$.

Lagrangian cobordisms and Lagrangian tangles are related by the addition or removal of appropriate Lagrangian collars, which we describe in the next section. 

\subsection{Lagrangian Collars}\label{sec_lag collars}
In this section we describe how we can add small Lagrangian collars to the ends of a given Lagrangian to get desired boundary conditions. We first restate some lemmas from \cite{ELST} that show that Lagrangians with equal slices can be glued smoothly. We remove the requirement that the Lagrangians are planar unknotted and include proofs that are slight alterations of those in \cite{ELST} so that we do not use the planar unknottedness hypothesis.
\begin{lem}\label{lem_standard nbhd of slice}\cite[Lemma 5.1]{ELST}
Let $L, L' \subset \R^4$ be two Lagrangians that are transverse to and agree on $\R^3_a$, for some $a \in \R$. Let
\begin{align*}
S \defeq L_a = L'_a .
\end{align*}
Then, for all $\epsilon > 0$, there exist neighborhoods $V  \subset U$ of $S$ in $L$ and a symplectic isotopy $\phi_t$ of $\R^4$ such that $\phi_t|_L$ is the identity on $S$ and on the complement of $U$, $\phi_1(U) \subset \R^4_{(a-\epsilon, a+\epsilon)}$, and $\phi_1(V)\subset L'$.
\end{lem}
\begin{proof}
We present the argument when $S$ has only one connected component but, as the argument is local in nature, it extends to the case when $S$ has many connected components by repeating this argument for each component.

Let $L_0$ be the $(x_1, x_2)$-plane, that is, the zero-section of $\R^4 = T^* \R^2$.
By the Lagrangian neighbourhood theorem, there exists a small neighbourhood $A \subset L$ of $S$ and a symplectomorphism $\psi$ of a tubular neighbourhood $ U \subset \R^4$ of $A$ taking $A$ to $S^1 \times I \subset L_0$ for some open interval $I \subset \R$.  Let $\gamma$ be the image of $S$ under $\psi$. 

The rest of the proof is identical to that of \cite[Lemma 5.1]{ELST}.
\end{proof}
This allows us to smoothly glue two Lagrangians that have a common slice or have matching boundaries.
\begin{lem}\label{lem_gluing lagrangians smoothly}
Let $L \subset \R^4_{[a,b]}$ and $L' \subset \R^4_{[b,c]}$ be two Lagrangians that are transverse to and agree on $\R^3_b$,
\begin{align*}
S = \p_+ L = L_b = L'_b= \p_- L'.
\end{align*}
Then, for all $\epsilon > 0$, there exists a 
Lagrangian $L'' \subset \R^4_{[a,c]}$ such that
\begin{enumerate}
\item $L'' \cap \R^4_{(-\infty, b - \epsilon)} = L \cap \R^4_{(-\infty, b - \epsilon)}$, and
\item  $L′′  \cap \R^4_{(b + \epsilon, \infty)} = L′  \cap \R^4_{(b + \epsilon, \infty)}$.
\end{enumerate}
\end{lem}
In \cite{ELST}, they assumed that the common slice was connected to get the planarity assumption. As we do not care about planarity of the Lagrangians, we can drop this assumption.
\begin{proof}
We first extend both $L$ and $L'$ arbitrarily such that the previous boundaries $\p_+ L$ and $\p_- L'$ are interior slices and the Lagrangians are still embedded. Then we use the previous lemma to make Lagrangians $\wt L$ and $\wt L'$ such that
\begin{align*}
\wt L \cap \R^4_{(-\infty,b-\epsilon)} &= L \cap \R^4_{(-\infty,b-\epsilon)} ,\\
\wt L' \cap \R^4_{(b + \epsilon, \infty)} &= L \cap \R^4_{(b + \epsilon, \infty)},\\
\wt L \cap \R^4_{(b-\frac{\epsilon}{2} ,b+\frac{\epsilon}{2})} &= \wt L' \cap \R^4_{(b-\frac{\epsilon}{2} ,b+\frac{\epsilon}{2})}.
\end{align*}
Then we may identify $\wt L$ and $\wt L'$ to get
\begin{align*}
L'' = (\wt L \cap \R^4_{(-\infty,b]} ) \sqcup_{\wt L_b = \wt L'_b} (\wt L' \cap \R^4_{[b, \infty)} )
\end{align*}
that is smooth.
\end{proof}

\begin{figure}[h!]
\begin{center}
  \includegraphics[width = 5in, height=3in]{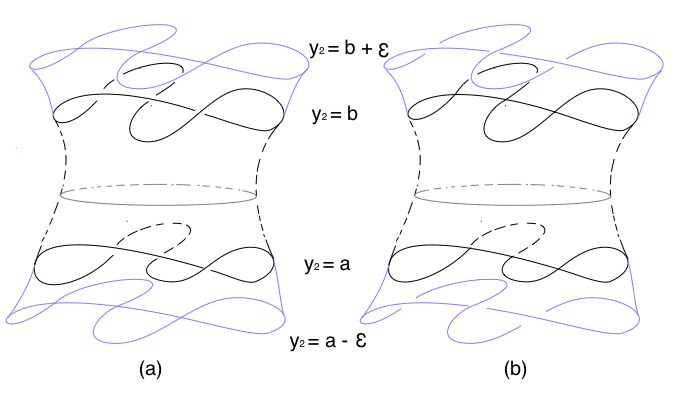}
  \end{center}
 \caption[caption]{\tabular[t]{@{}l@{}}Two types of Lagrangian collar attaching.\\ (a) Lagrangian tangle from relatively exact Lagrangian cobordism.\\ (b) Extension of Lagrangian tangle to get immersed Lagrangian with\\  transverse double points in the interior. \endtabular}
  \label{fig:lag collars}
\end{figure}
\begin{lem}\label{lem_collar_1}
Suppose we have a Lagrangian cobordism $L \subset \R^4_{[a,b]}$ (Definition \ref{defn_lag cobordism}). Given $\epsilon > 0$,
there exists a Lagrangian tangle  $L' \subset \R^4_{[a-\epsilon , b+ \epsilon]}$ (Definition \ref{defn_lag tangle}), 
such that $L = L' \cap \R^4_{[a, b]}$.
\end{lem}  

\begin{proof}
We attach small immersed Lagrangian collars, that is, images of immersions of $\sqcup^k S^1 \times [0, \eps]$ into $\R^4$ , at both $\partial_+ L$ and $\partial_- L$ to obtain a Lagrangian tangle.
We only describe the construction near $\p_+ L$, the case for $\partial_- L$ is exactly the same.

Let us consider the case when $\partial_+ L$ has one connected component. If there are multiple components, we can treat each of them individually with some care near intersections in $\pi_1(\p_+ L) \subset \R^2$.  We assume, for simplicity of notation, that $\partial_+ L$ lies in $\R^3_0$, that is, $b=0$.

Let $\pi_1: \R^3_0 \to \R^2$ be the projection $(x_1, y_1, x_2,0) \mapsto (x_1, y_1)$. Then, by our assumption about genericity of the link $\p_+ L$, $\pi_1 (\partial_+ L)$ is an immersed $S^1$. Let us fix an immersion $g: S^1 = \{(x,y) \in \R^2|x^2 + y^2 = 1\} \to \R^2$, such that $\im (g) = \pi_1(\p_+ L)$.
We can extend $g$ to an immersion (again called) $g: A \defeq S^1 \times [-\delta, \delta] \to \R^2$ for some small $\delta > 0$. Using this immersion, we pull back the standard symplectic form on $\R^2$ to get a symplectic form $\omega_g \defeq g^*(dx_1 \wedge dy_1)$ on $A$.

The height of points on $\partial_+ L$ or the value of the $x_2$-coordinate gives us a Hamiltonian $f$ on $g^{-1}(\pi_1(\partial_+ L))$, namely, $f: S^1 \times \{0\} \to \R$ is such that $(g,f, 0): S^1 \to \R^4$ gives a parametrization of $\p_+ L$. We extend $f$ to a Hamiltonian $f :  A \to \R$ that is supported in a neighbourhood of $S^1 \times \{0\}$ and away from $S^1 \times \{ -\delta\} \cup S^1 \times \{ + \delta\} = \p A$. 

Consider a smooth function cut off function $\beta: [0,1] \to [0,1]$ such that 
\begin{itemize}
\item $\beta(1) = 0$,
\item $\beta(t) = 1$ on $[0, \mu)$ for $0 < \mu \ll 1$,
\item $\beta(t) \neq 0$ for $t \neq 1$.
\end{itemize}
Now, consider the time-dependent Hamiltonian $f_t$, $t \in [0,1]$, on $ A$ given by
\begin{align*}
f_t(\theta, s) \defeq \beta(t)f(\theta, s), \text{ for } (\theta,s) \in S^1 \times [-\delta, \delta]= A.
\end{align*}
So, we have $ f_0(\theta, s) = f(\theta, s)$ and $ f_1 \equiv 0$.

On $A$, we get a Hamiltonian vector field $X_t$ that is the symplectic dual to $df_t$. Namely, for any point $p \in A$ and vector $v \in T_p A$ 
\begin{align*}
(df_t)_p(v) = \omega_g(X_t, v).
\end{align*}  Let $\phi_t$ be the flow of the vector field $X_t$, that is, $\phi : A \times [0,1] \to A$, $((\theta, s),t) \mapsto \phi_t(\theta,s)$ such that $\frac{d}{dt} \phi_t(\theta, s) = X_t \circ \phi_t (\theta,s)$ for all $(\theta,s) \in A$, and $\phi_0 = \rm{id}$. We now look at only $S^1 \times \{0\} \subset A$ and suppress the $s = 0$ in our notation.

Consider the map
\begin{align*}
F : S^1 \times [0,1]& \to A \times \R^2\\
(\theta, t) & \mapsto (\phi_t (\theta), f_t(\phi_t(\theta)), t).
\end{align*}
If we put the product symplectic form on $A \times \R^2$ given by $\omega_g \oplus dx \wedge dy$, $F$ is a Lagrangian embedding. Indeed, we can check
\begin{align*}
F^*(\omega_g \oplus dx \wedge dy) & = \phi_t^* \omega_g + \frac{\partial (f_t \circ \phi_t)}{\partial \theta} d\theta \wedge dt\\
& = \omega_g \left(\frac{\partial \phi_t }{\partial \theta}, \frac{\partial \phi_t}{\partial t}\right)d\theta \wedge dt + df \frac{\partial \phi_t}{\partial \theta} d\theta \wedge dt\\
& = - \omega_g \left(X_t, \frac{\partial \phi_t }{\partial \theta} \right)d\theta \wedge dt + \omega_g \left(X_t, \frac{\partial \phi_t }{\partial \theta} \right) d\theta \wedge dt\\
& = 0.
\end{align*}
$F$ is injective as $\phi_t$ is a diffeomorphism on $S^1$ for each $t$. $F$ is an immersion because, in local coordinates,
\begin{align*}
dF = \begin{pmatrix}
\frac{\partial \phi^1}{\partial \theta} & \frac{\partial \phi^1}{\partial t}\\
\frac{\partial \phi^2}{\partial \theta} & \frac{\partial \phi^2}{\partial t}\\
(df_t)_{\phi_t(\theta)} \left( \frac{\partial \phi}{\partial \theta}\right)& (df_t)_{\phi_t(\theta)} \left( \frac{\partial \phi}{\partial t}\right)\\
0& 1
\end{pmatrix},
\end{align*}
which is rank $2$ as $\phi_t$ is a diffeomorphism implies $\frac{\partial \phi_t}{\partial \theta} (\theta)= \begin{pmatrix}
\frac{\partial \phi^1}{\partial \theta}(t, \theta), & \frac{\partial \phi^2}{\partial \theta}(t, \theta)
\end{pmatrix}$ is non-zero.

Now, we can compose $F$ with the immersion $g: S^1 \times \{0\} \to \R^2$ to get 
\begin{align*}
G: S^1 \times [0,1] & \to \R^2 \times \R^2\\
(\theta,t)&  \mapsto (g \circ \phi_t(\theta), f_t(\phi_t(\theta)), t).
\end{align*}
As $g$ is an immersion that preserves the symplectic form, $G$ is  a Lagrangian immersion. Note that $G(\theta, 0) = (g(\theta), f(\theta), 0)$ is exactly a parametrization of $\partial_+ L$. So, we can ``attach" this collar to our original Lagrangian, $L$. We get smoothness from our choice of $\beta$. This gives us the required $L'$ with $\epsilon = 1$. By changing the time interval of the Hamiltonian $f$ to $[0, \epsilon]$ and taking $f$ such that $f_\epsilon \equiv 0$, we can attach a collar of height $\epsilon$ for any $\epsilon > 0$. 

We lose injectivity only at $t=1$ as $f_1 \equiv 0$. Indeed, as $\partial_+ L$ is embedded, $f_0$ is injective in a neighbourhood of the points where $g$ is not injective. By taking the support of $f_t$ to be a very small neighbourhood of $S^1$ in $A$, we can ensure that $\phi_t$ is close to identity. This will guarantee that $G$ is injective as long as $\beta(t) \neq 0$, that is, for $t \neq 1$. As $g^{-1}(p)$ for any point $p\in \R^2$ is at most two points, the same holds for $G^{-1} (q)$ for any $q \in \R^4$. Thus, $L'$ is embedded on the interior and has double points on the boundary.

$L'$ intersects itself transversely. Indeed, this follows if we write out $dG$ in local coordinates and use the fact that $g$ was an immersion and therefore, $dg$ is injective and $X_t$ is the symplectic dual to $df_t$ and the symplectic form is non-degenerate.
$L'$ has flat boundary as $f_1 \equiv 0$. 

It remains to check whether $L' \subset \R^4_{[a,b]}$ is relatively exact away from its boundaries. Note that $L' \cap \R^4_{[a, b+ \epsilon -\delta]}$ deformation retracts to $L$ for all $\epsilon > \delta > 0$ as we introduce new self-intersection points only at $y_2 = b+\epsilon$. So, relative exactness of $L$ implies relative exactness of $L'$.

\end{proof}

\begin{rem}\label{rem_standard collar}
Suppose we have a Lagrangian cobordism $L \subset \R^4_{[a,b]}$. We may assume that near the boundaries $L$ is standard in the following sense: There exists $\epsilon > 0$ such that for $t \in [a, a+ \epsilon]$ the crossings of $\pi_1(L_t)$ in $\R^2$ remain constant. The same is true for $t \in [b-\epsilon, b]$. We can assume this because we may construct a small Lagrangian collar $L^+ \subset \R^4_{[b-2\epsilon, b+ \epsilon]}$ that has the following properties:
\begin{itemize}
\item $L^+_b = L_b = \p_+ L$;
\item the double points of $\pi_1(L^+_t)$ remain constant for $t \in [b-\epsilon, b]$.
\end{itemize}
We can construct such $L^+$ by repeating the above type of constructions while taking a Hamiltonian on $g^{-1}(\pi_1 (\p_+ L))$ that has partial derivative $\frac{\p H}{\p \theta} = 0$ at the pre-images of the double points of $\pi_1(\p_+ L)$. Then, using Lemma \ref{lem_standard nbhd of slice} to make the neighbourhood of $\p_+ L$ in $L$ equal to the neighbourhood of $\p_+ L$ in $L^+$ we get a standard neighbourhood without changing $\p_+ L$. An alternate approach to construct $L^+$ is to use the Lagrangian moves given in Section \ref{sec_moves}.
\end{rem}

\begin{rem}\label{lem_collar_2}
Given a Lagrangian tangle $L \subset \{a \leq y_2 \leq b\}$ and $\epsilon > 0$, there exists an immersed Lagrangian $L' \subset \{a - \epsilon \leq y_2 \leq b + \epsilon\}$ such that $L = L' \cap \R^4_{[a , b]}$, and all the points of singularities of $L'$ are transverse double points away from $\p L'$. Further, the set of double points of $L'$ is equal to that of $L$, $\Delta(L) = \Delta(L')$.

The construction of $L'$ is practically the same as of Lemma \ref{lem_collar_1}, but the initial height function or Hamiltonian is the constant function $0$ and we take a time-dependent Hamiltonian satisfying $f_t \neq 0$ for $t > 0$ in a neighbourhood of $S^1$ in $A$.
\end{rem}

\subsection{Lagrangian Movies}\label{sec_lag movies}

In this section, we include some lemmas about Lagrangian ``movies," which are known popularly, for completeness. The term movies appears in other related works like \cite{Sauvaget2004} and \cite{lin}.
A \textbf{Lagrangian movie} is a convenient way of visualizing a Lagrangian via a family of slices of the Lagrangian. A family of slices fit together to give a Lagrangian provided they satisfy the partial differential equation (\ref{eqn_change in x_2}).  
\begin{prop}\label{prop_lag movie}
A generic Lagrangian surface $L \subset \R^4$
parametrized by
\begin{align} \label{eqn_lagrangian movie parametrization}
r(\theta, t) = (x_1(\theta,t), y_1 (\theta,t), x_2(\theta,t), t),
\end{align}
where $\theta \in S^1$, $t \in \R$, satisfies
\begin{align}\label{eqn_change in x_2}
\frac{\p x_2}{\p \theta} (\theta,t) = \omega \left( \frac{\p z}{\p t}, \frac{\p z}{\p \theta} \right) \text{ for } z(\theta,t) = (x_1(\theta,t), y_1(\theta,t)).
\end{align}

Conversely, if there is a family of knots parametrized by 
\begin{align*}
(z(\theta,t), x_2(\theta,t)) = (x_1(\theta,t), y_1 (\theta,t), x_2(\theta,t))
\end{align*} (with parameter $t$) in $\R^3$ satisfying Condition \ref{eqn_change in x_2} above, then
\begin{align*}
r(\theta,t) = (x_1(\theta,t), y_1 (\theta,t), x_2(\theta,t), t)
\end{align*}
gives a Lagrangian surface in $\R^4$.
\end{prop}
\begin{rem}
For any $t \in \R$, a tangent to $\pi_1(L \cap \R^3_t)$ is given by $(\frac{\p x_1}{\p \theta}, \frac{\p y_1}{\p \theta})$ and so, the normal $\eta = \ii \frac{\p}{\p \theta}(z)$ is given by $(\frac{\p y_1}{\p \theta}, - \frac{\p x_1}{\p \theta})$.
Then, the right hand side of (\ref{eqn_change in x_2}) is equal to
$
\langle\frac{\p z}{\p t}, \eta\rangle.
$
\end{rem}
\begin{proof}
If we pull back the form $\omega$ on $\R^4$ via $r$, we get
\begin{align*}
r^* \omega & = r^* (dx_1 \wedge dy_1 + dx_2 \wedge dy_2)\\
& = \left( \frac{\p x_1}{\p \theta} d\theta + \frac{\p x_1}{\p t} dt \right) \wedge \left( \frac{\p y_1}{\p \theta} d\theta + \frac{\p y_1}{\p t} dt \right) + \left( \frac{\p x_2}{\p \theta} d\theta + \frac{\p x_2}{\p t} dt \right) \wedge dt\\
& = \left( \frac{\p x_1}{\p \theta} \frac{\p y_1}{\p t}  -   \frac{\p y_1}{\p \theta}\frac{\p x_1}{\p t}  + \frac{\p x_2}{\p \theta} \right)d\theta \wedge dt\\
& =   \left( \omega\left(\frac{\p z}{\p \theta}, \frac{\p z}{\p t}\right) + \frac{\p x_2}{\p \theta} (\theta,t) \right) d\theta \wedge dt.
\end{align*}
Therefore, $r^* \omega \equiv 0$ if and only if Condition (\ref{eqn_change in x_2}) holds.
\end{proof}

For a Lagrangian, $L$, parametrized by $r(\theta,t)$ as above, 
fix a $\theta_0 \in S^1$ and define $$\Delta x_2(\theta,t) = x_2(\theta,t) - x_2(\theta_0, t),$$ and
\begin{align*}
\area(\theta,t) \defeq \frac{1}{2} \int_{\theta_0}^\theta \left(x \frac{\p y}{\p \theta}  - y \frac{\p x}{\p \theta} \right)d\theta = \frac{1}{2} \int_{\theta_0}^\theta \omega\left(z, \frac{\p z}{\p \theta} \right)d\theta = \frac{1}{2} \int_{\theta_0}^\theta \inner{z(\theta,t)}{\eta(\theta,t)} d\theta.
\end{align*}
We call this quantity ``area" even though the arc over which we integrate, namely, $\pi_1(L \cap \R^3_t)$ from $z(\theta, t)$ to $z(\theta_0, t)$ does not bound a bounded region unless $z(\theta_0, t) = z(\theta, t)$. When $z(\theta_0, t) = z(\theta, t)$, then $z_t|_{[\theta_0, \theta]}$ defines the boundary of a bounded region and $\area(\theta,t)$ actually measures the area of that region.
\begin{cor}\label{lem_derivative of area}
Given a Lagrangian in $\R^4$ parametrized by 
\begin{align*}
r(\theta,t) = (x_1(\theta,t), y_1 (\theta,t), x_2(\theta,t), t), \quad (\theta, t) \in S^1 \times \R,
\end{align*}
suppose that for $\theta_0, \theta_1 \in S^1$ and $t \in (c,c+\epsilon) \subset \R$, \begin{align}\label{eqn_intersection does not move}
(x_1(\theta_1,t), y_1(\theta_1,t)) = (x_1(\theta_0,t), y_1(\theta_0,t)).
\end{align}
Then, for $\Delta x_2(\theta, t)$ and area$(\theta, t)$ defined with reference point $\theta_0$ as above,
\begin{align*}
\Delta x_2(\theta_1,t) = \frac{\p}{\p t} \area(\theta_1,t).
\end{align*}
\end{cor}
\begin{proof}
Note that
\begin{align*}
\frac{\p}{\p t} (2 \, \area(\theta,t)) & = \int_{\theta_0}^{\theta_1} \innerb{\frac{\p}{\p t} z(\theta,t)}{\eta(\theta,t)}d\theta + \int_{\theta_0}^{\theta_1} \innerb{ z(\theta,t)}{\frac{\p}{\p t} \eta(\theta,t)} d\theta\\
& = \int_{\theta_0}^{\theta_1} \innerb{\frac{\p}{\p t} z(\theta,t)}{\eta(\theta,t)}d\theta + \int_{\theta_0}^{\theta_1} \left( x \frac{\p^2 y}{\p t \p \theta} - y \frac{\p^2 x}{\p t \p \theta} \right) d\theta\\
& = \int_{\theta_0}^{\theta_1} \innerb{\frac{\p}{\p t} z(\theta,t)}{\eta(\theta,t)}d\theta + \left[ x \frac{\p y}{
\p t} - y\frac{\p x}{\p t}\right]_{\theta_0}^{\theta_1} + \int_{\theta_0}^{\theta_1} \innerb{\frac{\p}{\p t} z(\theta,t)}{\eta(\theta,t)} \\
& = 2\int_{\theta_0}^{\theta_1} \frac{\p x_2 }{\p \theta}(\theta,t) d\theta + 0 = 2 \Delta x_2(\theta_1, t), 
\end{align*}
where the middle term is zero using the assumption, and we use integration by parts in the earlier step.
\end{proof}
If we do not have assumption (\ref{eqn_intersection does not move}), we can still put bounds on the area change by assuming that the intersections points do not move much. We would have another term that depended on the derivative $\frac{\p \ol\theta}{\p t}$ where $\ol \theta(t)$ is such that
\begin{align*}
(x_1, y_1)(\ol \theta(t),t) = (x_1, y_1)(\theta_0,t)
\end{align*} while we are assuming that one of the $S^1$-coordinates of the intersection points, namely $\theta_0$, does not move.  
This shows that by controlling the height of Lagrangian collars in Lemma \ref{lem_collar_1}, that is, the range of $y_2$-values it lies in, we may control the total change in areas bound by sectors of slices.

The above discussion leads to the following
lemma that has previously appeared in \cite{lin} and \cite{Sauvaget2004}. It was used to prove the existence of combinatorial moves on Lagrangian diagram in \cite{lin}. We will be using it to show the existence of enriched knot diagram moves, which are truly only a rephrasing of Lin's Lagrangian moves. 
\begin{lem}\label{lem: lin lagrangian }\cite[Lemma 2]{lin}
Suppose we are given a smooth map
\begin{align*}
z( \theta, t) = (x(\theta, t), y(\theta, t)):  S^1 \times [0,T] \to \R^2
\end{align*}
such that $z(-, t) : S^1 \to \R^2$ is an immersion with total signed area $C$ fixed with respect to $t$. Suppose additionally, we are given $h(-, 0): S^1 \to \R$. Then there exists $h:  S^1 \times [0,T] \to \R$ given by
\begin{align*}
h(\theta, t) = h(\theta, 0) + \int\limits_0^t \omega\left( \frac{\p z}{\p t}, \frac{\p z}{\p \theta}\right)dt,
\end{align*}
extending $h(-, 0)$ such that
\begin{align*}
 \phi: S^1 \times [0,T]  \to \R^4,\quad
 (\theta, t) \mapsto (x(\theta, t), y(\theta, t), h(\theta, t),t),
\end{align*}
is a Lagrangian immersion. Further, it is an embedding if and only if whenever $$z(\theta,t) = z(\theta', t), \,\, h(\theta,t) \neq h(\theta', t).$$
\end{lem}

 \section{Enriched Knot Diagrams}
In this section, we define precisely what we mean by an enriched knot diagram of a knot or link. We describe equivalence classes of enriched knot diagrams, and describe and discuss the undercut relation between these equivalence classes. In particular, we show that the undercut relation defines a strict partial order on equivalence classes of exact enriched knot diagrams.
 Lastly, we describe enriched knot diagram moves, which are adapted from \cite{lin}. Compare enriched knot diagrams with Lagrangian diagrams in \cite[Definition 1]{lin}. In particular, the crossings appear reversed as Lin looks at slices $x_2 =$ constant, whereas we consider $y_2 =$ constant.
 
Consider an immersed curve $\cD \subset \R^2$ whose only singularities are transverse double points. Denote by $X(\cD)$ the set of double points of $\cD$, which we call {\bf crossings of $\bf \cD$}.
\begin{defn}\label{defn_corners}
 For $p \in X(\cD)$ and $\epsilon > 0$ small enough, $B_\epsilon(p) \setminus \cD$ has four connected components that we call \textbf{corners} at $p$. We call two corners at $p$ \textbf{adjacent corners} if they have a common boundary arc and otherwise \textbf{opposite corners}. Note that the corners are defined only up to a choice of $\epsilon$ and we may also replace the ball $B_\epsilon(p)$ by a small open neighbourhood of $p$. We call two corners $C_1$ and $C_2$ at $p$ \textbf{equal} if there exists $\epsilon > 0$ such that $C_1 \cap B_\epsilon(p) = C_2 \cap B_\epsilon(p)$.
\end{defn}
\begin{defn}\label{defn_cut out}
The complement $\R^2 \setminus \cD$ consists of disjoint union of  connected sets $\cup_{j=1}^k D_j$. All but one $D_j$ are a bounded subsets of $\R^2$ that are diffeomorphic to the open disk, $\D^\circ = \{(x,y) \in \R^2| x^2 + y^2 <1\}$. In fact, these are images of holomorphic disks with corners  $u_j : \D \to D_j$ (see Definition \ref{defn_hol disk}).
We refer to these holomorphic disks with corners as \textbf{disks cut out by curve $\cD$}. 
\end{defn}
\begin{defn}\label{defn_diagram}
An \textbf{enriched knot diagram} is a tuple $(\cD, \sigma, \cA)$ where (refer Figure \ref{fig:defn_diagrams}):
\begin{itemize}
\item $\cD$ is an  immersed, closed curve (possibly disconnected) in $\R^2$ whose only singularities are transverse double points;
\item (SIGNS) $\sigma$ is a function assigning sign $\pm 1$ to each corner of $\cD$ such that at each double point of $\cD$, adjacent corners have opposite signs and opposite corners have the same sign;
\item (AREA) 
$\cA$ assigns to each disk cut out by $\cD$ the positive real number equal to the disk's area with respect to the standard metric on $\R^2$. We can also forget the geometric aspect and view it as a function on the set of all disks cut out by $\cD$ taking values in positive real numbers.
\end{itemize}
\end{defn}
\begin{defn}\label{defn_exact diagram}
An enriched knot diagram, $[\cD, \sigma, \cA]$, is said to be exact is the total signed area, $$\int_\cD x_1 dy_1 = 0.$$
\end{defn}
\begin{figure}[h!]
\begin{center}
\includegraphics[width = 5in]{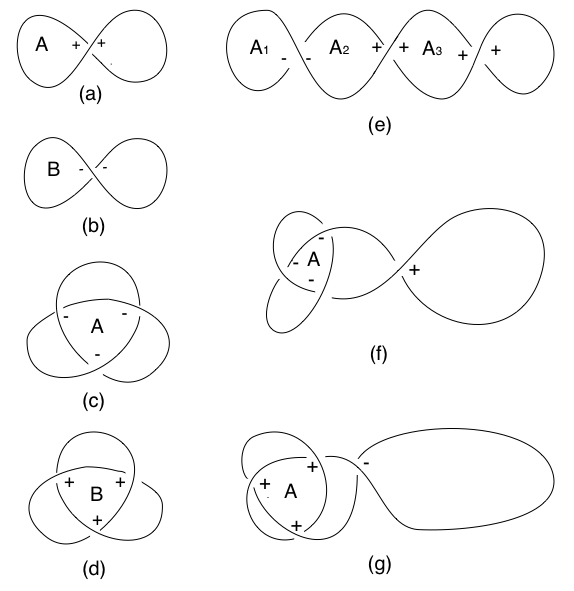}
 \end{center}
 \caption[caption]
    {\tabular[t]{@{}l@{}}Enriched knot diagrams (a) $8_+(A)$ (b) $8_-(B)$ (c) $T_-(A)$  (d) $T_+(B)$ \\ (e) $C^{-++}(A_1, A_2,A_3)$ (f) $E_-(A)$ (g) $E_+ (A)$\\
Note that we draw broken strands to clarify the crossing type, even though the\\ actual enriched knot diagram consists of an immersed curve with crossings.\\ Additionally, the information conveyed by the signs and the broken strands is the\\ same. We continue to use both, even though this information is redundant, for\\ clarity. Compare with Lagrangian diagrams in \cite{lin}.   
   \endtabular}
  \label{fig:defn_diagrams}
\end{figure}

Given a generic link $K \in \R^3$, we obtain an enriched knot diagram as follows:
Take $\cD_K$ to be the image $\pi_1(K)$. We get a ``height function"  $x_2 : \cD_K \setminus X(\cD_K) \to \R$ defined by $\pi_1^{-1}(z) = (z, x_2(z)) \in K $ outside the crossings. Generic $K$ gives $\cD_K$ that is immersed with only finitely many transverse double points as its singularities. We decorate $\cD_K$ to get a diagram as follows:
\begin{itemize}
\item(SIGNS) At each crossing $q$, by taking a small enough neighbourhood $U \ni q$, we get four corners in $U$ divided by $\cD$. Intuitively, if for a corner, $C$, the height value jumps up $q$ as we traverse $\p C$ in the anticlockwise direction, we assign $\sigma_K(C) = -1$. If the height value jumps down, we assign $\sigma_K(C) = 1$. More precisely, let us denote by $\lim_{\p C \to q^-} x_2$ the limit of $x_2$-values as we approach $q$ along the boundary of $C$ (on $\cD$) in the anticlockwise direction. Similarly, let $\lim_{\p C \to q^+} x_2$ denote the limit in the clockwise direction. Then we assign
\begin{align*}
\sigma_K(C) = -1 \text{ if } &\lim_{\p C \to q^-} x_2 < \lim_{\p C \to q^+} x_2, \text{ and }\\
\sigma_K(C) = 1 \text{ if } &\lim_{\p C \to q^-} x_2 > \lim_{\p C \to q^+} x_2.
\end{align*}
\item(AREAS) The area $\cA_K(D)$ assigned to each disk ,$D$, cut out by $\cD_K$ is the area of $D$ with respect to the standard metric on $\R^2$.
\end{itemize}
Then $(\cD_K, \sigma_K, \cA_K)$ is the enriched knot diagram of $K$.

For any diagram $(\cD, \sigma, \cA)$, $\cD$ is the image of an immersion $\iota: \sqcup_{j = 1}^k S^1_j \to \R^2$ where each $S^1_j$ is a copy of the unit circle. The above construction of obtaining a diagram from a link can be easily reversed by assigning a height function $x_2: \sqcup_{j = 1}^k S^1_j \to \R$, so that the combined map $(\iota, x_2): \sqcup_{j = 1}^k S^1_j \to \R^3$ is an embedding that respects the signs $\sigma$. That is, if we formed a diagram from this link following the above method, we get back the original diagram $(\cD, \sigma, \cA)$. Let us denote a link obtained from $(\cD, \sigma, \cA)$ by $K_\cD$. The height function for $K_\cD$ and therefore, the link itself is not uniquely determined by the diagram $(\cD, \sigma, \cA)$. This will not concern us as the undercutting relation, existence of Lagrangian cobordisms, and the obstructions we develop will only depend on equivalence classes of enriched knot diagrams, which we define next. But first, we define the enriched knot diagrams for the boundaries of a Lagrangian tangle.
\begin{defn}
Consider a Lagrangian tangle, $L \subset \R^4_{[a,b]}$. Let $\iota_b: \R^2_b \to \R^2$ be the canonical isomorphism $(x_1, y_1, 0, b) \mapsto (x_1, y_1)$.
We say a diagram, $(\cD, \sigma, \cA)$, is the enriched knot diagram of $\p_+ L$ if $\cD = \iota_b(\p_+ L)$, for any disk $D$ cut out by $\cD$, $\cA(D)$ is equal to the area of $D$ with respect to the standard Riemannian metric on $\R^2$, and the signs are the signs for the corresponding corners of $L_{b-\epsilon}$ for small $\epsilon > 0$. This makes sense as the crossings of $\pi_1(L_{b-\epsilon})$ are in on-to-one correspondence with the crossings of $\pi_1(\p_+ L)$ (see Remark \ref{rem_standard collar}). Similarly, we define the enriched knot diagram for $\p_- L = L_a$ using  $\iota_a: \R^2_b \to \R^2$ and the signs from $L_{a+\epsilon}$.
\end{defn}

\begin{defn}\label{defn_diagram equivalence}
Two diagrams, $(\cD_1, \sigma_1, \cA_1)$ and $(\cD_2, \sigma_2, \cA_2)$, are said to be \textbf{equivalent} if there exists an orientation preserving diffeomorphism $\phi: \R^2 \to \R^2$ such that 
\begin{enumerate}
\item $\phi(\cD_1) = \cD_2$,
\item for all corners $C$ of $\cD_1$, $\sigma_2 (\phi(C)) = \sigma_1(C)$,
\item for all disks $D$ cut out by $\cD_1$, $\cA_2 \circ \phi (D) = \cA_1 (D)$.
\end{enumerate}
We denote the equivalence class of an enriched knot diagram $(\cD, \sigma, \cA)$ by $[\cD, \sigma, \cA]$.
We sometimes drop the $\sigma$ and $\cA$ from the notation and write $[\cD]$ with the hope that it makes reading easier and does not confuse the reader.
\end{defn} 
The above definition is trivially an equivalence relation as it is defined via diffeomorphisms.
\begin{lem}\label{lem_equivalent defns of ekd equivalence}

For an enriched knot diagram, $(\cD_0, \sigma_0, \cA_0)$, any diffeotopy (isotopy through diffeomorphisms) $\psi_t: \R^2 \to \R^2$ gives rise to a family of enriched knot diagram $\psi_t^*(\cD_0, \sigma_0, \cA_0)$ where $\psi_t^*(\cD_0) = \psi_t(\cD_0)$; for any corner $C$ of $\cD_0$, $\psi_t^*\sigma_0(\psi_t(C)) = \sigma_0(C)$; and for any disk $D$ cut out by $\cD_0$, $\psi_t^*\cA_0(\psi_t(D))$ is the area of $\psi_t(D)$ with respect to the standard metric on $\R^2$. 

Two enriched knot diagrams $(\cD_0, \sigma_0, \cA_0)$ and $(\cD_1, \sigma_1, \cA_1)$ are equivalent if and only if there exists a compactly supported Hamiltonian isotopy $\psi_t: \R^2 \to \R^2$ such that $\psi_1(\cD_0) = \cD_1$, and $\psi_1^*\sigma_0 = \sigma_1$.

\end{lem}
\begin{proof}
The existence of a Hamiltonian isotopy $\psi_t: \R^2 \to \R^2$ such that $\psi_1(\cD_0) = \cD_1$, and $\psi_1^*\sigma_0 = \sigma_1$, implies equivalence of the diagrams $(\cD_0, \sigma_0, \cA_0)$ and $\cD_1, \sigma_1, \cA_1)$ via the diffeomorphism $\psi_1$. As $\psi_1$ is the time-$1$ map of a Hamiltonian isotopy, it is area preserving.

For the converse, suppose there exists a diffeomorphism $\psi: \R^2 \to \R^2$  making $(\cD_0, \sigma_0, \cA_0) \sim (\cD_1, \sigma_1, \cA_1)$ as in Definition \ref{defn_diagram equivalence}. Note that by using appropriate cut off functions outside a compact set $K$ that contains both $\cD_1$ and $\cD_2$, we may assume that the diffeomorphism $\psi$ is equal to identity outside $K$.
We show that the existence of such a diffeomorphism, $\psi$, implies existence of a symplectomorphism, $\phi$, such that $\phi(\cD_0) = \cD_1$, and for all corners, $C$, of $\cD_0$, $\sigma_2 (\phi(C)) = \sigma_1(C)$. Then, as the group of symplectomorphisms $\mathrm{Symp} (\R^2)$ is  isomorphic to the group of compactly supported Hamiltonian isotopies (see \cite{barsamian}), we get that $\phi = \phi_1$ for a compactly supported Hamiltonian isotopy $\phi_t: \R^2 \to \R^2$.  The  double points criterion is a generic condition about transverse self-intersections. Note that it is not necessary that $\psi = \phi$.

 To construct the symplectomorphism $\phi$, we carry out a standard Moser argument. View $\cD_0$ as the $1$-skeleton of a CW-complex structure on $\R^2$ with one $2$-cell of infinite area and proceed by induction on the dimension of the skeleta. First we want a symplectomorphism at the crossings of $\cD_0$ such that $d\phi_p(v)$ is a positive multiple of $d\psi_p(v)$. We achieve this by rescaling by $\frac{1}{\det d\psi_p}$. Now along each arc of $\cD_0$ we define $\phi$ such that it takes the arc $\alpha$ to $\psi(\alpha)$ and it is symplectic. We achieve this by taking its derivative $d\phi$ as a rescaling of $d\psi$ such that the symplectic condition holds. Now we want to extend across the $2$-cells. Note that by using cutoff functions we may extend $\phi$ to all of $\R^2$ such that it matches $\psi$ outside a neighbourhood of the $1$-skeleton, $\cD_0$. Let us call this map $\wt \phi$. Then, $\omega_t \defeq (1-t)\wt \phi^* \omega + t \omega$ is a cohomologous path of symplectic forms on each $2$-cell such that $\omega_t = \omega$ on $\cD_0$ and outside $K$. Then for each $2$-cell, $D$, by \cite[Theorem 3.17]{MS_Intro}, there exists an isotopy of diffeomorphisms $\eta^D_t: D \to D$ such that $\eta^D_0 = \id$, $(\eta^D_t)^* \omega_t = \omega$ for all $t$ and $\eta^D_t = \id$ for all $t$ on $\cD_0$ (and outside $K$ for the $2$-cell that is unbounded). As $\eta^D$'s for adjacent $D$'s match on $\cD_0$, we may glue $\eta_t$ together to get $\eta_t : \R^2 \to \R^2$.  Then $\phi = \wt \phi \circ \eta_1$ is the required symplectomorphism.
\end{proof}

\subsection{Undercut relation}
In this section we define undercut relation and discuss some of its properties.

Define a \textbf{diagram pair} to be an ordered pair of two diagrams $([\cD_1, \sigma_1, \cA_1],[\cD_2, \sigma_2, \cA_2])$. We refer to the first diagram, $[\cD_1, \sigma_1, \cA_1]$, as the \textbf{lower diagram} and $[\cD_2, \sigma_2, \cA_2]$ as the \textbf{upper diagram} of the pair. Given a diagram pair $([\cD_1, \sigma_1, \cA_1], [\cD_2, \sigma_2, \cA_2])$, we say a disk $A$ is \textbf{cut out by the pair} if it is either cut out by $\cD_1$ or by $\cD_2$. Recall we defined area of an oriented knot as $a(K) = |\int_K \lambda|$, which does not depend on the chosen orientation of the knot. So, we may talk about relatively exact cobordisms without mentioning an orientation.

To show that the undercut relation is well-defined on equivalence classes of enriched knot diagrams we prove the following lemma.
\begin{lem}\label{lem_relation well defined}
Suppose $K_1$ and $K_2$ are two links in $\R^3$ with equivalent enriched knot diagrams $(\cD_1, \sigma_1, \cA_1)$ and $(\cD_2, \sigma_2, \cA_2)$, respectively. 
Then, for a link $K_3$ if there exists Lagrangian cobordism $K_3 \prec  K_1$, there exists a Lagrangian cobordism $K_3 \prec K_2$. Similarly, $K_1 \prec  K_3$ implies $K_2 \prec  K_3$.
\end{lem}

\begin{proof}
Suppose $L \subset \R^4_{[a,b]}$ is a Lagrangian cobordism from $K_3$ to $K_1$. Assume for simplicity that $K_1$ and $K_2$ are knots, that is, they have one connected component each. 
We want to change $L$ in a neighbourhood of the boundary $\p_+ L$ so that the new boundary is $K_2$. We do this by constructing a Lagrangian collar $\wt L$ such that $\wt L_{b-\epsilon} = L_{b-\epsilon}$ for some $\epsilon > 0$ and $\wt L_b = \iota_b(K_2)$. Note that the main difficulty in this proof is to make the boundary of $L$ exactly $K_2$, including achieving the areas of the underlying diagram.
 
 Choose $\epsilon > 0$ small compared to $b-a$ such that $L \cap \R^4_{[b-2\epsilon,b]}$ has a parametrization,
 \begin{align*}
 S^1 \times [b-2\epsilon, b]  \to \R^2,\quad
 (\theta, t)  \mapsto (z(\theta, t), x_2(\theta, t), t).
 \end{align*} 
More conditions on the choice of $\epsilon > 0$ will be clear as the proof progresses. Label $K_1' \defeq L_{b-\epsilon}$ and $K_2'' \defeq L_{b-2\epsilon}$.
Assume that we have a time dependent Hamiltonian
 \begin{align*}
 H: \R^2 \times [b - 2 \epsilon, b] \to \R,\quad
 (z, \tau) \mapsto H(z, \tau)
 \end{align*}
such that for Hamiltonian vector field $X_\tau$ defined by $\omega(X_\tau, \cdot) = dH_\tau$, the flow 
\begin{align*}
\phi: \R^2 \times [b - 2 \epsilon, b] \to \R^2
\end{align*} 
of $X$ gives Hamiltonian isotopy such that $\phi_1(\cD_1) = \cD_2$ and $\phi_1^* \sigma_1 = \sigma_2$, as in Lemma \ref{lem_equivalent defns of ekd equivalence}. Note that this means $\frac{d}{d\tau}\phi_x(\tau) = X_{\phi_x(\tau)}$ for all $x \in \R^2$.
 Then, as
  \begin{align*}
 S^1 &\to \R^3,\quad
 \theta \mapsto ( z(\theta, b), x_2(\theta, b))
 \end{align*}
 is a parametrization of $K_1$,  we get a parametrization of $K_2$ 
 \begin{align*}
 S^1 \to \R^3,\quad
 \theta \mapsto (\phi_b \circ z(\theta, b), h(\theta))
 \end{align*}
for an appropriately chosen $h:S^1 \to \R$. Let us denote $w \defeq \phi_b \circ z$.

 Consider the map
 \begin{align*}
z': S^1 \times [b-2\epsilon, b-\epsilon]  \to \R^2,\quad
 (\theta, t) \mapsto \phi_{t} \circ z(\theta, t).
 \end{align*}
 We would like to use $z'$ as the $(x_1, y_1)$-part of a Lagrangian parametrization. Note that 
 \begin{align*}
 \frac{\p z'}{\p \theta}(\theta, t) &=  \frac{\p \phi_t}{\p z}(z(\theta, t))  \frac{\p z}{\p \theta}(\theta,t) = d\phi_t \circ \frac{\p z}{\p \theta};\\
  \frac{\p z'}{\p t}(\theta, t) &=  \frac{\p \phi_t}{\p z}(z(\theta, t))  \frac{\p z}{\p t}(\theta,t) +  \frac{\p \phi}{\p \tau}(z(\theta, t), {t - b + 2\epsilon});\\
  \omega\left(\frac{\p z'}{\p \theta}, \frac{\p z'}{\p t}\right) &= \omega\left( d\phi_t \circ \frac{\p z}{\p \theta}, d\phi_t  \circ \frac{\p z}{\p t}\right) + \omega\left( d\phi_t  \circ \frac{\p z}{\p \theta}, \frac{\p \phi}{\p \tau}\right)\\
  & = \omega\left(  \frac{\p z}{\p \theta},  \frac{\p z}{\p t}\right) + \omega\left( d\phi_t  \left( \frac{\p z}{\p \theta}\right), \frac{\p \phi}{\p \tau}\right)\\
  & = \frac{\p x_2}{\p \theta} + \omega\left( d\phi_t \left( \frac{\p z}{\p \theta}\right), \frac{\p \phi}{\p \tau}\right)\\
  & =  \frac{\p x_2}{\p \theta} + dH_{t} \circ d\phi_t  \left(\frac{\p z}{\p \theta}\right) = \frac{\p}{\p \theta} (x_2 + H_t \circ \phi_t \circ z).
 \end{align*}
 Define $x_2': S^1 \times [b-2\epsilon, b-\epsilon] \to \R$ to be
 \begin{align*}
 x_2'(\theta,t) = x_2(\theta, t) + H_t \circ \phi_t \circ z(\theta, t).
 \end{align*} Then we have 
 \begin{align*}
 x_2'(\theta, b-2\epsilon) = x_2(\theta, b - 2 \epsilon),\quad
 \frac{\p x_2'}{\p \theta}  = \frac{\p}{\p \theta} (x_2 + H_t \circ \phi_t \circ z).
 \end{align*}
So $L' \subset \R^4_{[b-2\epsilon, b-\epsilon]}$ defined by the following parametrization
 \begin{align*}
 S^1 \times [b-2\epsilon, b-\epsilon] \to \R^2, \quad
 (\theta, t)  \mapsto (z'(\theta, t) , x_2'(\theta, t), t),
 \end{align*}
is a Lagrangian that we can attach to $L \cap \R^4_{[a,b-2\epsilon]}$. 
 
 Now we want to build another Lagrangian $L'' \subset \R^4_{b-\epsilon, b}$ such that $L''_{b-\epsilon} = L'_{b-\epsilon}$ and $L_b'' = K_2$. So we need functions 
 \begin{align*}
 z'' : S^1 \times [b - \epsilon, b] \to \R^2, \quad x''_2 : S^1 \times [b - \epsilon, b] \to \R,
 \end{align*}
 such that
 \begin{align*}
 z''(\theta, b-\epsilon) & = z'(\theta, b-\epsilon);
 \quad z''(\theta, b) = w(\theta);\\
 x_2''(\theta, b-\epsilon) & = x_2'(\theta, b-\epsilon); \quad
 x_2''(\theta, b) = h(\theta);\\
 \frac{\p x_2''}{\p \theta} & = \omega\left( \frac{\p z''}{\p \theta},  \frac{\p z''}{\p t}\right).
 \end{align*}
 
Assume that the isotopy $\phi$ is ``fast" on $\tau \in [b-2\epsilon, b-\epsilon]$ and we are very ``close" to the diagram we want to achieve. More concretely, we want that $z'_{b-\epsilon} = z'(\theta, b - \epsilon)$ to have the following properties:
\begin{enumerate}
\item Image of $z'_{b-\epsilon}$ is in a neighbourhood of $\im (w)$ such that for each arc $\gamma$ of $\im z'_{b-\epsilon} \setminus X(\im z'_{b-\epsilon})$, the corresponding arc of $w$ is graphical over it. Here, $X(\im z'_{b-\epsilon})$ denotes the double points of $z'_{b-\epsilon}$.
\item For each sector $S$ of area $A$ cut out by $w$, the corresponding sector of $z'_{b-\epsilon}$ has area $A'$ such that
\begin{align*}
\sign ( A - A') = \sign \left( \sum_{p \text{ positive corner of} S} \delta h(p) -  \sum_{p \text{ negative corner of} S} \delta h(p) \right).
\end{align*}
Here, $\delta h (p) = |h(\theta_1) - h(\theta_2)|$, whenever $w(\theta_1) = w(\theta_2)=p$.
\end{enumerate}
 We may achieve this by controlling the speed of $\phi_x$, that is, $\frac{\p \phi_x}{\p \tau}$ for each $x \in \R^2$, by reparametrizing with respect to the $\tau \in [b-2\epsilon, b]$ coordinate of the Hamiltonian $H$. The second property we always get as we assumed the signs at corresponding corners of $K_1$ and $K_2$ are the same and $x_2'(\theta , t) = x_2 (\theta, t) + H_t \circ \phi_t \circ z(\theta, t)$, implies $\delta x_2 = \delta x_2'$ as $H_t$ is a function on $\R^2$ and so, (see Lemma \ref{lem_derivative of area})
\begin{align*}
\delta x_2'(p) = \delta x_2(p, t) = \frac{d}{dt} \area \text{ for all } t \in [b-2\epsilon, b].
\end{align*}
 
 By the first assumption, for each continuous arc $\gamma$ of $z'_{b-\epsilon} \setminus \Delta(z'_{b-\epsilon})$ we can take a neighbourhood $N_\gamma$ and symplectomorphisms $\psi_\gamma: N_\gamma \to I_\gamma \times (-\delta, \delta)$ for some interval $I_\gamma$ of appropriate length and $\delta > 0$, where $\gamma$ is mapped to $I_\gamma \times \{0\}$ and the corresponding arc of $w$ is mapped to $(p, \eta_\gamma(p)) \in I \times (-\delta, \delta)$ for some function $\eta_\gamma$ on $I_\gamma$ because of the graphical assumption. 
 
 Denote by $h^\gamma_1, h^\gamma_2: I \to \R$
 functions such that when $(p,0) = \psi_\gamma(z'(\theta, b-\epsilon)$,
 \begin{align*}
 h^\gamma_1(p) = x_2'(\theta, b-\epsilon), \quad h^\gamma_2(p) = h(\theta).
\end{align*}  
 If we pick Hamiltonians $H_s^\gamma: I \times (-\delta, \delta) \to \R$ for $s \in [b-\epsilon, b]$, and $\gamma$ arc of $z'_{b-\epsilon} \setminus \Delta(z'_{b-\epsilon})$ such that
 \begin{align*}
 H_s^\gamma(p,q) = H_s^\gamma(p), \quad
 H^\gamma_0  = 0, \quad h^\gamma_1 + H^\gamma_1  = h^\gamma_2,\quad \text{ and }
 -\int_0^1 \frac{\p H^\gamma_s}{\p p} (\psi(p,0) (s))ds & = \eta_\gamma(p),
 \end{align*} 
 and consecutive $H^\gamma$ match at end points, then the function
 \begin{align*}
 S^1 \times [b-\epsilon, b] & \to \R^4\\
(\theta, s) & \mapsto ( \psi_\gamma^{-1} \circ \phi_s \circ \psi_\gamma \circ z'(\theta, b-\epsilon) , x_2'(\theta, b-\epsilon) + H^\gamma_s \psi_\gamma \circ z'(\theta, b-\epsilon),s)
 \end{align*}
 is the required $L''$ parametrization. We can obtain such $H^\gamma$ as the boundary conditions and the derivatives are on different coordinates.
 
 The assumption on the signs and the area means we can choose $H^\gamma_s$'s such that this is an embedding and there are no double points. We have to just ensure that if $h^\gamma_1(p_1) > h^\gamma_1(p_2)$(resp. $h^\gamma_1(p_1) < h^\gamma_1(p_2)$) for $p_1, p_2$ endpoints of some $I_\gamma$, then
 $h_1 + H_s^\gamma(p_1) > h_1 + H_s^\gamma(p_2)$(resp $<$) for all $s \in [b-\epsilon, b]$. This is possible as the assumption on signs of $K_1$ and $K_2$ implies $h_2(p_1) > h_2(p_2)$(resp $<$) whenever $h_1(p_1) > h_1(p_2)$(resp $<$). 

 Now we may use Lemma \ref{lem_gluing lagrangians smoothly} to glue $L \cap \R^4_{(-\infty,b-2\epsilon]}$, $L'$, and $L''$ together to get the required Lagrangian cobordism from $K_3$ to $K_2$. As relative exactness is a condition on the topology of $L$ and we do not change $H_2(\R^4, L)$ and $a(\p_+ L) = a(\p_- L)$, relative exactness holds for the new Lagrangian also.
\end{proof}
Note that the above Lemma implies Lemma \ref{lem_collar_1} if we allow $K_2$ to be an immersed link in $\R^2$, that is, we take the height function $h \equiv 0$. In that case we get double points at $y_2 = b$ as before. In fact, the proof of Lemma \ref{lem_relation well defined} implies the stronger statement, Lemma \ref{lem_tangle with same diagrams}.  We still include the separate proof of Lemma \ref{lem_collar_1} as we are able to write down the Hamiltonian more precisely.

\begin{lem}\label{lem_tangle with same diagrams}
Suppose we have a Lagrangian cobordism $L \subset \R^4_{[a,b]}$ with diagram of $\p_+ L$ equal to $\cD_2$ and diagram of $\p_- L$ equal to $\cD_1$. Then we can obtain a Lagrangian tangle $L' \subset \R^4_{[a,b]}$ such that diagram of $\p_+ L'$ equal to $\cD_2$ and diagram of $\p_- L'$ equal to $\cD_1$.
\end{lem}

We are now ready to prove Lemma \ref{lem_partial order}, except Part 3 of the partial relation statement.
\begin{proof}[Proof of Lemma \ref{lem_partial order}]
\begin{enumerate}
\item Well-definedness follows directly from Lemma \ref{lem_relation well defined}.
\item
Transitivity follows by using Lemma \ref{lem_relation well defined} and Lemma \ref{lem_gluing lagrangians smoothly}. Suppose $$[\cD_1] \prec_{L_1} [\cD_2] \text{ and }[\cD_2] \prec_{L_2} [\cD_3],$$ that is, we have Lagrangian cobordisms $L_1$ and $L_2$ such that 
$\p_- L_1$, $\p_+ L_1$, $\p_- L_2$, and $\p_+ L_2$ have diagrams $\cD_1$, $\cD_2$, $\cD_2$, and $\cD_3$, respectively. We are suppressing the notation for signs and areas. Lemma \ref{lem_relation well defined} implies we can assume $\p_+ L_1 = \p_- L_2$, and so we may glue $L_1$ and $L_2$ using Lemma \ref{lem_gluing lagrangians smoothly} to get $[\cD_1] \prec [\cD_3]$. 
\end{enumerate}

Now we prove the additional properties that make $\prec$ a partial order on exact diagrams.
\begin{enumerate}
\item  Recall that if we restrict to exact diagrams, that is, if the total signed area bound by the diagram is zero, the relatively exact condition of the Lagrangian cobordism means that it is weakly exact.

Suppose, if possible, $[\cD] \prec [\cD]$. Then, by using Lemma \ref{lem_relation well defined} we may assume that there exists an exact Lagrangian $L$ with $\p_+ L = \tau_{b-a} (\p_- L)$, where $\tau_c : \R^4 \to \R^4$ is given by $(x_1, y_1, x_2, y_2)  \mapsto (x_1, y_1, x_2, y_2 + c)$. 

Consider the quotient space $\R^4 / (z \sim \tau_{b-a} (z))$. 
Then the image, $(L/ \sim)$, of the Lagrangian $L$ under the quotient map  gives a weakly exact closed Lagrangian in $(\R^4 / \sim) \simeq \C \times V$, where $V \simeq T^* S^1$ with the appropriate symplectic structure. By Lemma \ref{lem_gluing lagrangians smoothly}, we may assume smoothness of $(L/ \sim)$.  As $\Pi_2 (\R^4, L) \cong \Pi_2(\R^4/\sim, L/\sim)$, using the push forward of the quotient map, $\omega|_{\Pi_2(\R^4, L)} \equiv 0$ implies $\omega|_{\Pi_2(\R^4/\sim, L/\sim)} \equiv 0$. So, $(L/ \sim)$ is weakly exact in $\R^4/\sim$.

On the other hand, such a weakly exact closed Lagrangian cannot exist in a symplectic manifold of the type $\C \times V$ with the product symplectic structure by \cite[Theorem $2.3.B_1$]{gro}. Hence, no such Lagrangian cobordism exists.
 
\item Anti-symmetry follows by combining anti-reflexivity and transitivity. Suppose we have two diagrams, $[\cD_1] \neq [\cD_2]$ with relations $[\cD_1] \prec_{L_1} [\cD_2]$ and $[\cD_2]  \prec_{L_2} [\cD_1]$. Then by first gluing $L_1$ and $L_2$ and then quotienting by the appropriate translation we obtain a closed exact Lagrangian in $T^* S^1 \times \C$ as in the proof of non-reflexivity, which is not possible.
\item We show in Corollary \ref{cor_no relation exists} that, for $0 < A < B$, 
$$8_+^1(A) \nprec C^{+-+}(A, B, B) \text{ and }C^{+-+}(A, B, B) \nprec 8_+^1(A).$$
\end{enumerate}
\end{proof}

\subsection{Enriched Knot Diagram Moves}\label{sec_moves}

In this section we describe moves akin to Reidemeister moves in knot theory that can be achieved via Lagrangian cobordism. We include area conditions so that it makes sense in the enriched knot diagram context. This section is included for completeness and as a natural look at existence questions alongside obstructions. 

These moves almost completely appear in \cite{lin}. We add in area conditions that are more global to the enriched knot diagram compared to the local nature of moves described in \cite{lin}, but everything described here follows from \cite{lin}. Our diagrams appear to have the opposite signs to that of \cite{lin} because Lin works with $x_2$-slices of $\R^4$ and we work with $y_2$-slices. So, we redraw all the moves with our conventions. The arguments remain analogous.

Recall that we defined the area of an enriched knot diagram $[\cD_1]$ by $a(\cD_1) = |\int_{\cD_1} x_1 dy_1|$.

\begin{prop}\label{prop_moves}
Suppose we have enriched knot diagrams $[\cD_1, \sigma_1, \cA_1]$ and $[\cD_2, \sigma_2, \cA_2]$. Suppose $\{A_1, A_2, \dots, A_k\}$ is the set of all areas of disks cut out by $\cD_1$. If the disk corresponding to $A_i$ survives, that is, if there exists a corresponding disk cut out by $\cD_2$, denote the area of the corresponding disk of $[\cD_2]$ by $A_i + \epsilon_i$.

Then $[\cD_1, \sigma_1, \cA_1] \prec [\cD_2, \sigma_2, \cA_2]$ if the following conditions hold:
\begin{enumerate}
\item  $a(\cD_1) = a(\cD_2)$.
\item  If $\epsilon_j < 0$, then $|\epsilon_j| < A_j$.
\item If a disk has only positive corners and survives, then $\epsilon_j > 0$.
\item If a disk has only negative corners and survives, then $\epsilon_j < 0$.
\item $\cD_2$ differs from $\cD_1$ by one of the moves in Figure \ref{fig:moves} assuming $\cD_1$ satisfies the area conditions written under the arrow.

\end{enumerate}
\end{prop}
\begin{figure}[h!]
\begin{center}
  \includegraphics[width = 6in]{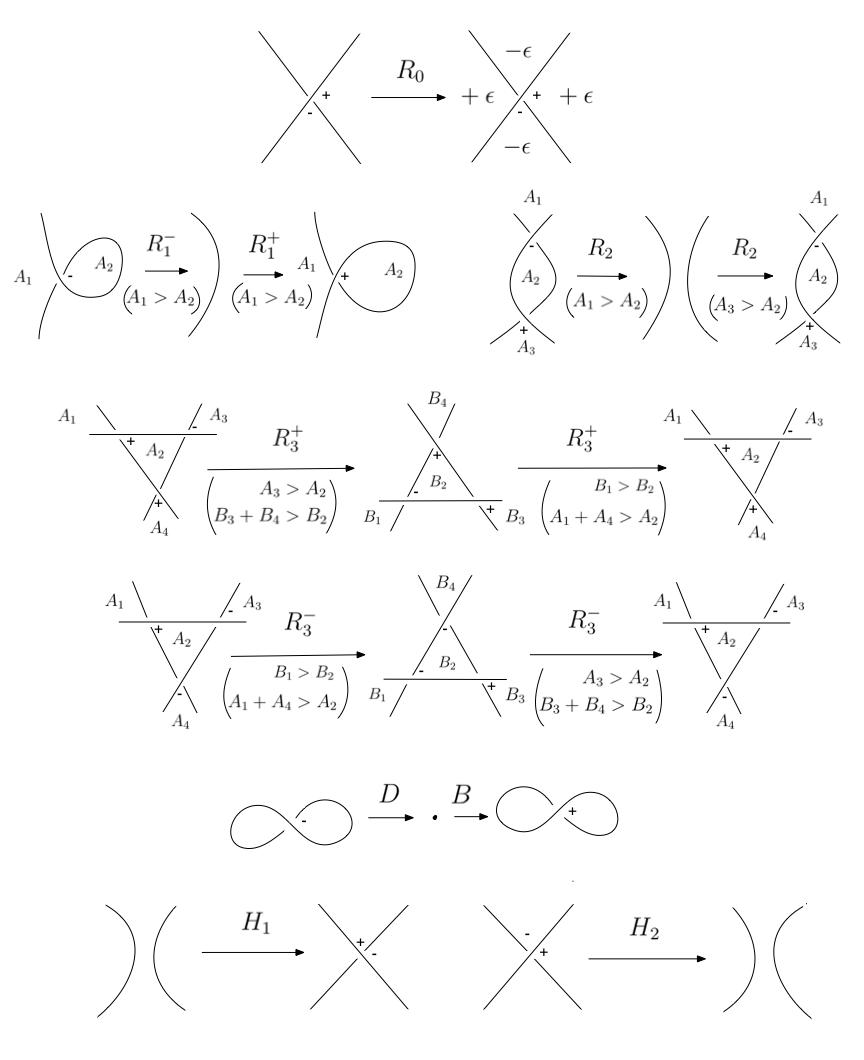}
  \end{center}
 \caption[caption]{\tabular[t]{@{}l@{}}Enriched knot diagram moves.\\ Compare with the combinatorial moves on Lagrangian diagrams\\ in \cite[Figure 11]{lin}.\endtabular}
  \label{fig:moves}
\end{figure}

\begin{rem}
The areas of disks cut out by $\cD_2$ is determined by condition 1 and which moves were performed on $\cD_1$ to obtain $\cD_2$.
For example, suppose $I \subset \{1, \dots, k \}$ is such that
\begin{align*}
\sum_{i \in I} A_i - \sum_{i \notin I} A_i = a(\cD_1).
\end{align*}
Condition 1 implies, if we only performed $R_0$, that is, if all the disks cut out by $\cD_1$ survived in $\cD_2$, then 
\begin{align*}
\sum_{i \in I} \epsilon_i - \sum_{i \notin I} \epsilon_i = 0.
\end{align*}
\end{rem}

\begin{proof}[Proof of Proposition \ref{prop_moves}]
Let us suppose that $\cD_2$ differs from $\cD_1$ by one of the moves described in Figure \ref{fig:moves} that we refer to as $M$. We show that we can construct a Lagrangian cobordism from $\cD_1$ to $\cD_2$ using Lemma \ref{lem: lin lagrangian }.
For the move $M = R_0$, this Lagrangian was described exactly in \cite{lin}.
For $M$ different from $R_0$, note that given the area conditions on $\cD_1$, we can use move $R_0$ to enclose $M$ within a disk of fixed radius. Then it is in the set up on Lin's moves. Thus, $M = R_2$, $R_3$, $H_1$, and $H_2$ are Lin's $R_0 \circ R_2 \circ R_0$, $R_0 \circ R_3 \circ R_0$, $R_0 \circ H_1 \circ R_0$, and $R_0 \circ H_2 \circ R_0$, respectively.

The birth and death moves $M = D$ and $B$ can be described as follows.
The functions
\begin{align*}
z(\theta, t) = (t\cos\theta, 2 t^2 \cos \theta \sin \theta), \,\, h(\theta, 0) = t\sin \theta
\end{align*}
for $r \in [0, T]$, gives a birth move and
\begin{align*}
z(\theta,t) = (t\cos\theta, -2 t^2 \cos \theta \sin \theta),\,\, h(\theta, 0) = t \sin \theta
\end{align*} 
gives a death move $D$. This is not a unique description. As we have seen before, the exact geometry of the movie is not unique. The moves $R_1^+$ and $R^-_1$ follow from using a piece of the birth or death moves, respectively.
\end{proof}

\section{Moduli Spaces of Holomorphic Disks}\label{sec_moduli spaces}
Our main tool for studying Lagrangian cobordisms is holomorphic disks that have boundary on Lagrangian tangles. In this section we define the holomorphic functions we work with, and then review some standard notions about moduli spaces of holomorphic curves See \cite{mcduffsalamon_fat} (Chapter 5) and \cite{AJ}. We then show compactness holds in our set up.  

We consider the complex structure $(\R^4, \ii)$ obtained by identifying $\R^4$ with $(\C^2, \ii)$. Let $\D \defeq \{z \in \C | |z| \leq 1 \}$ be the closed unit disk. For any subset of $\C$, like the unit disk $\D$, we get a complex structure on it as a subset of $\C$. We denote this sructure also by $\ii$. There should not be any confusion as these are all essentially multiplication by the same $\ii$. We denote the open unit disk by $\D^\circ \defeq \{z \in \C | |z| < 1\}$. The boundary of the unit disk is denoted by $ \p \D \defeq \{z \in \C| |z| = 1\} = S^1 $.
Fix a Lagrangian tangle, $L$, in $\R^4$.

\subsection{Holomorphic Disks}\label{subsec_hol disks}

\begin{defn}\label{defn_hol disk}
A \textbf{holomorphic disk with corners} is a function $u: (\D, \p \D) \to (\R^4, L)$ such that the following conditions hold:
\begin{itemize}
\item $u$ is holomorphic on the interior of $\D$, that is, $u$ satisfies the Cauchy-Riemann equation
\begin{align*}
du(p) \circ \ii = \ii \circ du(u(p))
\end{align*}
for all $p \in \D^\circ$.
\item There are finitely many distinct points $z_1, \dots, z_m \in \p \D$, known as \textbf{marked points} that are mapped to double points of $L$ and where $\p u$ jumps from one leaf of $L$ to another. The images of these points are called \textbf{corners} of $u$.
\item In the complement of the corner points $u$ is $C^2$-regular.
\item $u$ is an immersion. 
\end{itemize}
We denote by $\p u$ the restriction of $u$ to the boundary $\p \D$ of the disk $\D$. We can orient $\p u$ with the anti-clockwise orientation on $\p \D$, which is also the complex orientation.
\end{defn}
We make the assumption that at each marked point the disk jumps between leaves of $L$ because we do not allow marked points where $u$ is smooth. Note that unless $\p u$ jumps leaves at a marked point, it would be a smooth point. See the discussion about the trace operator in the proof of Lemma \ref{lem_interior breaking in pairs} for a proof.

\begin{defn}\label{defn_bound disk}
As in Definition \ref{defn_hol disk}, we may define holomorphic disks with corners taking values in $\R^2$ that have boundaries on any smooth curve. Given a diagram $(\cD, \sigma, \cA)$,
we say a holomorphic disk with corners $u: (\D, \p \D) \to (\R^2, \cD)$ such that each corner point of $u$ is a double point of the diagram $\cD$ is a disk \textbf{bound} by the diagram. We say it is a disk \textbf{bound by the pair} $((\cD_1, \sigma_1, \cA_1), (\cD_2, \sigma_2, \cA_2))$ if it is a disk bound by $(\cD_1, \sigma_1, \cA_1)$ or by $(\cD_2, \sigma_2, \cA_2)$.
\end{defn}
Note that the area of a disk $u$ bound by a diagram is the sum of the areas of disks cut out by the diagram that together make $u$, counted with the appropriate sign and multiplicity. 
\begin{defn}\label{defn_horizontal disk}
Suppose $L$ is a Lagrangian tangle in $\R^4$ (Definition \ref{defn_lag tangle}). 
A \textbf{horizontal disk} is a holomorphic disk with corners whose image lies completely in a complex plane $\R^2_c$ for some $c \in \R$. 

For this paper, we only consider horizontal holomorphic disks with convex corners.
\end{defn}
A holomorphic disk with corners will be horizontal if $\im(\p u) \subset \p_+ L$ or $\im(\p u) \subset \p_- L$, see Lemma \ref{lem_cpt_1} for a proof. 

We get signs at the corners of a horizontal disk just by assigning the sign of the corner as per the diagram of $\p_+ L$ or $\p_-L$. We would like to extend this notion of corner sign to all holomorphic disks in a way that is useful in describing broken disks in moduli spaces.
Recall that at each double point $q \in \Delta L \subset \p L$, the Lagrangian has an higher leaf $L^h$, and a lower leaf $L^l$.
\begin{defn}
Let $u: (\D, \partial \D) \to (\R^4, L)$ be a holomorphic disk with corner $q \in \Delta(L)$. If $\p u$ traverses the leaf $L^h$ first and then $L^l$, we say that $u$ has \textbf{positive sign} at $q$. Otherwise, we say $u$ has \textbf{negative sign} at $q$. 
More explicitly, just as in the definition of enriched knot diagram coming from a knot $K$, 
let us denote by $\lim_{\p u \to q^-} x_2$ the limit of $x_2$-values as we approach $q$ along the boundary $\p u$. Similarly, let $\lim_{\p C \to q^+} x_2$ denote the limit in the reverse direction. Then we assign
\begin{align*}
\sign_q(u) = -1 \text{ if } &\lim_{\p u \to q^-} x_2 < \lim_{\p u \to q^+} x_2, \text{ and }\\
\sign_q(u) = 1 \text{ if } &\lim_{\p u \to q^-} x_2 > \lim_{\p u \to q^+} x_2.
\end{align*}
\end{defn}
If $u$ is an immersed disk, it may pass through $q$ multiple times, but always in the same direction. So, this is well-defined.

If $u$ is a horizontal disk, the signs are part of the information of the enriched knot diagram of $\partial_-L$ or $\partial_+L$, whichever $u$ has boundary on. A horizontal disk can have both positive and negative signs at corners. On the other hand, we show in Lemma \ref{lem_signs of vert disks} that non-horizontal disks can have only positive signs at the upper boundary and only negative signs at the lower boundary of a Lagrangian tangle.

\subsection{Stable Disks}
A tree is a connected graph without cycles. We think of a tree as a finite set $T$ of vertices along with a relation $E \subset T \times T$ on $T$ such that $\alpha, \beta \in T$ are related, that is, $ \alpha E \beta$, if and only if they are connected by an edge. An \textbf{$m$-labeling} $\Lambda$ of $T$ is a function
\begin{align*}
\Lambda: I &\to T\\ i &\mapsto \alpha_i,
\end{align*}
where $I = \{1, \dots, m\}$. Such an $m$-labeling partitions $I$ into disjoint subsets
\begin{align*}
\Lambda_\alpha = \{i \in I | \alpha_i = \alpha\}.
\end{align*}
\begin{figure}[h!]
\begin{center}
  \includegraphics[width = 5in]{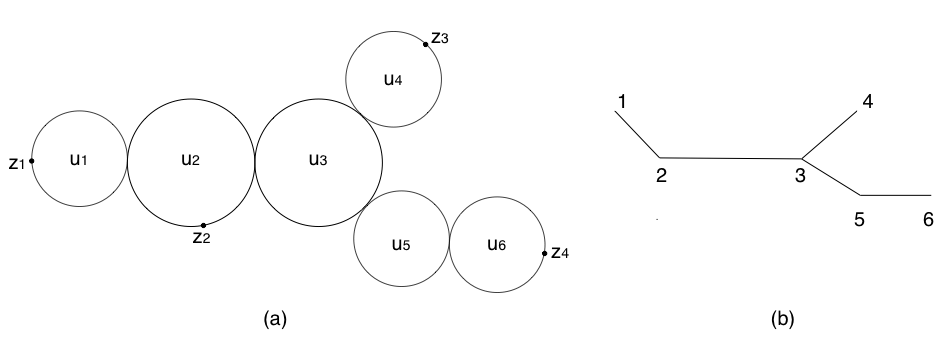}
  \end{center}
 \caption{(a) Stable disk modelled on tree $T$ (b) Tree $T$}
  \label{fig:stable disk}
\end{figure}
\begin{defn}[Stable Disks]\label{defn_stable map} Let $n \geq 0$ be a non-negative integer. Let $(T, E, \Lambda)$ be an $n$-labelled tree. A \textbf{stable holomorphic disk with ${\bf m}$ marked points (modelled over ${\bf (T, E, \Lambda)}$)}, or simply a \textbf{stable disk}, is a tuple
\begin{align*}
({\bf u}, {\bf z}) = (\{u_\alpha\}_{\alpha \in T}, \{z_{\alpha\beta}\}_{\alpha E \beta}, \{ \alpha_i, z_i\}_{1\leq i\leq m}),
\end{align*}
consisting of a collection of holomorphic disks with corners $u_\alpha: (\D, \partial \D) \to (\R^4, L)$ labelled by vertices $\alpha \in T$, a collection of \textbf{nodal points} $z_{\alpha\beta} \in \partial \D$ labelled by the oriented edges $\alpha E \beta$, and a sequence of $m$ \textbf{marked points} $z_1, \dots, z_m \in \partial \D$, such that the following conditions are satisfied. 
\begin{itemize}
\item (NODAL POINTS) If $\alpha, \beta \in T$ with $\alpha E \beta$, then $u_\alpha (z_{\alpha \beta}) = u_\beta (z_{\beta \alpha})$. The set of nodal points on the $\alpha$-disk is denoted by \begin{align*} Z_\alpha = \{z_{\alpha \beta } | \alpha E \beta\}.\end{align*}
\item (SPECIAL POINTS) For every $\alpha \in T$ the nodal points $z_{\alpha \beta}$ and the marked points $z_i$ (for $\alpha_i = \alpha$) are pairwise distinct. These are called \textbf{special points}. We denote the set of special points on the $\alpha$-disk by
\begin{align*}
Y_\alpha = Z_\alpha \cup \{z_i| \alpha_i = \alpha\}.
\end{align*}
\item (STABILITY) If $u_\alpha$ is constant function then $\# Y_\alpha \geq 3$. 
\item (HOLOMORPHICITY) The maps $u_\alpha$ are holomorphic on the complement of the special points, i.e.
\begin{align*}
u_\alpha|_{\D \setminus Y_\alpha} : \D \setminus Y_\alpha \to \R^4
\end{align*}
is holomorphic. 
\end{itemize}
(Compare with \cite[Definition 5.1.1]{mcduffsalamon_fat}.)
\end{defn}

We exclude marked points on the interior of any disk because they would only arise if we allowed sphere bubbles and sphere bubbles cannot form inside $\R^4$, as $H_2(\R^4) = 0$.

The domain of each of the maps $u_\alpha$ is the closed unit disk $\D = \{z \in \C | |z|^2 \leq 1 \}$. We fix the holomorphic structure on the domain as the one coming from the standard complex structure on $\C$. The domain of the stable holomorphic map $({\bf u,z })$ can be represented as the quotient 
\begin{align*}
\Sigma = T \times \D / \sim
\end{align*}
where the equivalence relation on $T \times \D$ is given by $(\alpha, z) \sim (\beta, w)$ if and only if either $\alpha = \beta$ and $z = w$ or $\alpha E \beta$ and $z_{\alpha \beta} = w_{\beta \alpha}$. The domain $\Sigma$ gets the quotient holomorphic structure which we denote by $j_\Sigma$. We denote the map with domain $\Sigma$ by $u : \Sigma \to \R^4$

Let us denote the set of singularities of $\Sigma$, namely the nodes, by $\rm{si}(\Sigma)$.
We will be considering domains $\Sigma$ that have \textbf{genus zero}, by which we mean that the double $\Sigma \cup_{\partial \Sigma} \bar{\Sigma}$ is a genus $0$ compact Riemann surface whose only singularities are nodes. Equivalently, this means that there is a deformation (in the sense of Hummel \cite[Chapter V, Section 1]{hummel} ) $f : \D \to \Sigma$ that is a continuous surjective map with the following properties.
\begin{enumerate}
\item The preimage $f^{-1}(\{z\})$ of each singular point, $z \in \Sigma$, is an arc with end points on $\partial \D$.
\item $f|_{\D \setminus f^{-1}(\rm{si} \Sigma)}$ is a diffeomorphism onto $\Sigma \setminus \rm{si}(\Sigma)$.
\end{enumerate}
In \cite{hummel}, they allow the preimage $f^{-1}(\{z\})$ of a singular point $z \in \Sigma$ to also be a simple closed loop, but we need not consider that case because we do not have sphere bubbles.
 
Recall that, $\ii$ and $\omega$ define a Riemannian metric on $\R^4$ as $\ii$ is compatible with the symplectic structure $\omega$. This metric is equal to the standard Riemannian metric. Using this, we define the energy of a $j_\Sigma$-holomorphic curve $u: (\Sigma, \p \Sigma) \to (\R^4, L)$ defined on a compact Riemann surface with boundary $(\Sigma, j_\Sigma, d\vol_\Sigma)$, where the boundary is mapped to a Lagrangian $L$, to be
\begin{align*}
E(u) \defeq \frac{1}{2} \int\limits_\Sigma |du(z)|^2 d\vol_\Sigma = \int\limits_\Sigma u^*\omega.
\end{align*}
Hence, the energy of a $j_\Sigma$-holomorphic curve defined on a closed Riemann surface with boundary, if the boundary is mapped to a Lagrangian, is a topological invariant.
If $({\bf u,z})$ is a holomorphic disk tree we denote the \textbf{energy} of ${\bf u}$
\begin{align*}
E({\bf u}) = \sum_{\alpha \in T} E(u_\alpha).
\end{align*}
Thus, the energy $E({\bf u})$ of ${\bf u}$ vanishes if and only if all the maps $u_\alpha$ are constant.

\subsection{Boundary Data and Corner Points}
In the case the Lagrangian is embedded we can take appropriate isomorphism classes of $({\bf u,z})$ to define moduli spaces. In our case, the Lagrangian is immersed. So, one needs to keep track of some extra information. In this section, we repeat a small construction from \cite[Section 4]{AJ} with simplifying assumptions that suffice for our purposes.

For $({\bf u,z})$ as above we would like to think of $\partial \Sigma$ as a smooth circle, but this is not true if $\Sigma$ has nodes and $\im(u)$ contains points of $\Delta(L)$. 
Let $S^1 = \{z \in \C : |z| = 1 \}$ be the unit circle in $\C$ with the anti-clockwise orientation. The boundary of $\partial \Sigma$ has the orientation induced by the complex structure of each copy of the disk $\D$, and there is a continuous and orientation preserving map $l : S^1 \to \partial \Sigma$ unique up to reparametrization such that
\begin{itemize}
\item the inverse image of a singular point, that is each node, of $\partial \Sigma$ consists of two points;
\item the inverse image of a smooth point of $\partial \Sigma$ consists of one point.
\end{itemize}

The main observation here is that we can define this map by following along the boundary $\partial \Sigma$, and at each singular point making the choice of branch that is consistent with the orientation. We also need to make sure we do not miss any part of $\p \Sigma$. Write $\zeta_i \defeq l^{-1} (z_i)$ for marked points $z_i$, $i = 1, \dots, m$.

 For a stable disk, $({\bf u,z})$, as above, $u \circ l$ is a continuous map $S^1 \to L$. Recall that for any immersed Lagrangian, $L$, there exists an immersion $\iota: \tilde{L} \to \R^4$ from a surface $\wt L$ with image $\iota(\tilde{L}) = L$. The map $u\circ l$ can not always be lifted locally to a continuous map $\bar{u} : S^1 \to \tilde{L}$ with $\iota \circ \bar{u} \equiv u \circ l$. 
 For a point $q \in \Delta(L) \subset \R^4$ we have $\iota^{-1}(q) = \{q_+, q_-\}$, that is, there are two points, $q_+ \neq q_-$, in $\tilde{L}$ that map to one point $q$ under the immersion $\iota$, and $L$ has two sheets near $q$ that are the images under $\iota$ of disjoint open neighbourhoods of $q_+$ and $q_-$ in $\tilde{L}$.

If $u \circ l(\zeta) = q$ for some $\zeta \in S^1$, it may be the case that $u \circ l$ jumps at $\zeta$ from one sheet of $L$ to another near $q$ in $\R^4$. Then $u \circ l$ cannot be lifted to a continuous map $\bar{u} : S^1 \to \tilde{L}$ near $\zeta$, since $\bar{u}$ would have to jump discontinuously between $q_+$ and $q_-$ near $\zeta$. We make the following definition so that we can keep track of such jumps. 

We consider $(u, z)$ in which $u \circ l$ jumps at $z_i$ between two sheets of $\iota(L)$ in this way for all $i \in \{1, \dots, m\}$. 

\begin{defn}(Stable disks with corners) Let $L$ be an immersed Lagrangian coming from immersion $\iota: \wt L \to \R^4$. Define $R$ to be the set of ordered pairs $(q_-, q_+) \in \tilde{L} \times \tilde{L}$ such that $q_-\neq q_+$ but $\iota(q_-) = \iota(q_+)$. 
Fix a map $\lambda: \{1, \dots, m\} \to R$. Consider a tuple $({\bf u, z}, \lambda, l, \bar{u})$, where $({\bf u,z})$ is a stable disk as in Definition \ref{defn_stable map}, $l : S^1 \to \partial \Sigma$ is a continuous and orientation preserving map as above, and $\bar{u}: S^1 \setminus \{z_1, \dots, z_m\} \to \tilde{L}$ is a continuous map, satisfying the following conditions:
\begin{itemize}
\item $z_1, \dots, z_m$ are ordered anti-clockwise on $S^1$;
\item $\iota \circ \bar{u} \equiv u \circ l$ on $S^1 \setminus \{z_1, \dots, z_m\}$; and
\item $(\lim_{\theta \to 0^-} \bar{u}(e^{\rm{i}\theta} z_i), \lim_{\theta \to 0^+} \bar{u}(e^{\rm{i} \theta} z_i))  = \lambda(i)$ in $R$ for all $i \in \{1, \dots, m\}$.
\end{itemize}
Such a tuple $(u,z,\lambda, l, \bar{u})$ is called a \textbf{stable map with corners}. (Compare with \cite[Definition 4.2]{AJ}.) Notice that we are only specifying what happens at the marked points of $u$. At the nodes, a priori we might get a jump between leaves, but they do not appear because of how $l$ was chosen.

Note that we could have that $u \circ l$ jumped leaves of $L$ only at some of the marked points. This is the case in \cite[Definition 4.2]{AJ}. In that case we would consider $\lambda: I \to R$ for some $I \subset \{1, \dots, m\}$ and refer to $z_i$ for $i \in I$ as pre-corner points. For our case, we only consider the case when $I = \{1, \dots, m \}$. So, we drop the name pre-corner point and call them marked points.
\end{defn}

\subsection{Reparametrization Groups}
\begin{defn}\label{equivalence of stable maps}
Two stable disks, $({\bf u,z})$ and $({\bf \tilde{u}, \tilde{z}})$, modelled over trees, $T$ and $\tilde{T}$, respectively, are called \textbf{equivalent} if there exists a tree isomorphism, $f: T \to \tilde{T}$, and a function, $T \to G \defeq \Aut(\D), \alpha \mapsto \phi_\alpha$, which assigns to each vertex an automorphism of the disk, such that 
\begin{align*}
\tilde{u}_{f(\alpha)} \circ \phi_\alpha = u_\alpha,\quad \tilde{z}_{f(\alpha)f(\beta)} = \phi_\alpha(z_{\alpha\beta}), \quad \tilde{z}_i = \phi_{\alpha_i}(z_i).
\end{align*}
\end{defn}
We can think of $f$ as a map that assigns to each $\phi_\alpha \in Aut(\D)$ its target disk, $\D_{f(\alpha)}$,  and $\phi_\alpha$ is a map, $\phi_\alpha: \D_\alpha \to \D_{f(\alpha)}$. (Compare with \cite[Definition 5.1.4]{mcduffsalamon_fat}.)

Note that by our definition, if we have a holomorphic disk with corners, the tree that it is modelled on is a single vertex $\{v\}$. So, $f: \{v\} \to \{v\}$ is the identity map and the only information for equivalence is the element $\phi \in \Aut \D$ such that 
\begin{align*}
\wt u \circ \phi = u, \quad \phi(z_i) = z_i.
\end{align*}
If the number of marked points is high, namely greater than or equal to $3$, the identity map $\D \to \D$ is the only candidate for $\phi$.

For every tree, $T$, there is an associated group, $G_T$, that acts on the set of stable disks modelled on $T$ and whose orbit space is the corresponding set of equivalence classes. The elements of $G_T$ are tuples  
\begin{align*}
g = (f, \{\phi_\alpha\}_{\alpha \in T}),
\end{align*}
where $f$ is an automorphism of $T$ such that 
\begin{align*}
\Lambda_\alpha = \Lambda_{f(\alpha)} \text{ and } \phi_\alpha \in \Aut(\D) \text{  for }\alpha \in T.
\end{align*}
The group operation is given by composition,
\begin{align*}
g' \cdot g = (f' \circ f, \{ \phi'_{f(\alpha)} \circ \phi_\alpha \}_{\alpha \in T} ). 
\end{align*}

We would like to consider the subgroup of $G$ that is compatible with the corner points and boundary data.
\begin{defn}\label{equivalence of stable maps with corners}
Two stable disks with corners $({\bf u, z},\lambda,l,\bar{u})$ and $({\bf u', z'},\lambda',l',\bar{u}')$ \\
(where the corresponding stable disks $({\bf u,z})$ and $({\bf u', z'})$ are modelled over trees $T$ and $T'$ respectively) are called \textbf{equivalent} if all of the following conditions hold: 
\begin{itemize}
\item $({\bf u,z})$ and $({\bf u',z'})$ are equivalent as stable disks and the equivalence is given by some $(f, \{\phi_\alpha\}_{\alpha \in T}) \in G_T$;
\item there exists an orientation preserving homeomorphism $ \bar{\phi} : S^1 \to S^1$ such that 
\begin{align*}
\phi_\alpha \circ l & = l' \circ \bar{\phi} & \text{ on $S^1 \setminus \{z_1 , \dots, z_m \} \cap \D_\alpha$, and }\\
\bar{u}' \circ \bar{\phi} & = \bar{u} & \text{ on $S^1 \setminus \{z_1 , \dots, z_m\}$. }
\end{align*}
\end{itemize} (Compare with \cite[Definition 4.2]{AJ}.)
\end{defn}

\subsection{Definition of Moduli Spaces}\label{defn_moduli space}
Fix a relative homology class, $A \in H_2(\R^4, L;\Z)$, and a map, $\lambda : \{1, \dots, m\} \to R$. 
A stable disk with $m$-marked points $({\bf u,z})$ is said to \textbf{represent the class A} if
\begin{align*}
A = \sum_{\alpha \in T} u_{\alpha*}[\D],
\end{align*}
where $[\D] \in H_2(\D, \partial \D; \Z)$ is the fundamental class.
Note that the energy of the stable curve, $E({\bf u})$, depends only on the class, $A \in H_2(\R^4, L; \Z)$, it represents.

We denote the {\bf moduli space of unparametrized disks in the class $\bf A$ with $\bf m$ marked points} by
\begin{align*}
\ol \M_{m}(A;L) = \{({\bf u,z }) & \text{ stable holomorphic disk with} \\
& m \text{ marked points that has boundary on } L \}/\sim 
\end{align*}
where $\sim$ is the equivalence defined in Definition \ref{equivalence of stable maps}.
Similarly, we get {\bf moduli space of unparametrized disks in the class $\bf A$ with $\bf m$ corners} by
\begin{align*}
\ol \M_{m}(A, \lambda; L) = \{({\bf u, z},\alpha, l, \bar{u}) & \text{ stable holomorphic disk that has boundary on } L \text{ and } \\
& m \text{ corners labelled by } \lambda\}/\sim
\end{align*}
where $\sim$ is the equivalence in Definition \ref{equivalence of stable maps with corners}. (Compare this definition of $\ol \M$ with definitions in \cite[Section 5.1]{mcduffsalamon_fat}.) We sometimes denote elements just by $[{\bf u, z}]$ or $[{\bf u},z_1, \dots, z_m]$ for ${\bf z} = (z_1, \dots, z_m)$, but $\alpha, l$, and $\bar{u}$ are always included in the given information.

For each labelled tree, $T = (T, E, \Lambda)$, and homology class, $A \in H_2(\R^4, L; \Z)$, we denote the equivalence classes of stable disks modelled on $T$ and representing $A$ by
\begin{align*}
\ol \M_{T} (A;L) = \{({\bf u,z }) \in \ol  \M_{n}(A) | \text{modelled on } T\}/ G_T.
\end{align*}
Note that the spaces $\ol \M_{T}(A;L)$ correspond to decompositions of $A$ into
\begin{align*}
A = \sum_{\alpha \in T} A_\alpha
\end{align*}
of integral relative homology classes.
 
If there is only one Lagrangian in the discussion, then we drop the $L$ from the notation as well. We sometimes use a representative holomorphic disk to ``define" the moduli space. What we mean is, suppose we have $(u, {\bf z})$, a particular fixed holomorphic disk with corners in $\Delta(L)$. Then, $u$ gives us a homology class $A \defeq [u] \in H_2(\R^4, L)$ and also a map $\lambda: \{1, \dots, m\} \to \Delta(L)$, namely, $\lambda(j) = u(z_j)$. We denote
\begin{align*}
\ol \M(u) \defeq \ol \M_{m}(A, \lambda).
\end{align*}
We sometimes refer to this moduli space as the \textbf{space of variations of $u$}. 

Let us denote by $\M(u)$, $\M(A, \lambda)$, the corresponding moduli space with only genuine holomorphic disks with corners.

\subsection{Gromov Convergence}
Suppose $X$ is a compact subset of $\D$.
A sequence of maps $v^\nu : \D \to \R^4$ is said to  \textbf{converge u.c.s.} (uniformly on compact sets) on $\D \setminus X$ to $v : \D \to \R^4$ if it converges to $v$ in the $C^\infty$-topology on every compact subset of $\D \setminus X$.

For a tree $T$ and $\alpha, \beta \in T$ the interval $[\alpha, \beta] \subset T$ denotes the set of vertices along the path of edges connecting $\alpha$ and $\beta$. When $\alpha E \beta$, we denote by $T_{\alpha \beta}$ the subtree containing $\beta$ after removing the edge connecting $\alpha$ and $\beta$, i.e. $T_{\alpha \beta} \defeq \{\gamma \in T | \beta \in [\alpha, \gamma] \}$. If $u$ is a stable disk modelled on $T$ with $\alpha E \beta$, we denote 
\begin{align*}
m_{\alpha \beta} (u) \defeq \sum_{\gamma \in T_{\alpha \beta}} E(u_\gamma).
\end{align*}
\begin{figure}
\begin{center}
\includegraphics[width=5in]{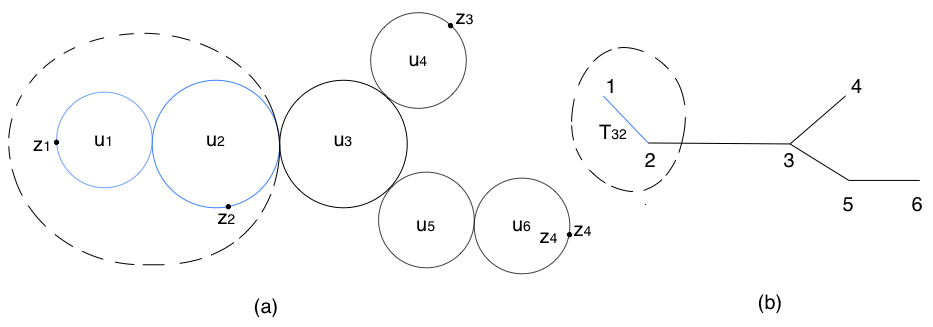}
\end{center}
\caption[caption]{\tabular[t]{@{}l@{}}(a) Stable disk modelled on subtree $T_{32}$ is encircled by dotted lines,\\ (b) Subtree $T_{32}$\endtabular}
  \label{fig:subtree disk}
\end{figure}

\begin{defn} (Gromov convergence)
Let
\begin{align*}
({\bf u,z}) = (\{u_\alpha\}_{\alpha \in T}, \{z_{\alpha\beta}\}_{\alpha E \beta}, \{ \alpha_i, z_i\}_{1\leq i\leq m}) 
\end{align*}
be a stable disk and let $u^\nu : (\D, \partial \D) \to (\R^4, L)$ be a sequence of holomorphic maps with $m$ distinct marked points $z_1^\nu, \dots, z_m^\nu \in \partial \D$. The sequence $(u^\nu, {\bf z}^\nu) = (u^\nu , z_1^\nu, \dots, z_m^\nu)$ is said to \textbf{Gromov converge} to the {\bf Gromov limit}, $({\bf u,z})$, if there exists a sequence of sets $\{\phi_\alpha^\nu\}_{\alpha \in T}^{\nu \in \mathbb{N}}$ of disk automorphisms $\phi_\alpha^\nu \in \Aut \D$, such that the following hold:
\begin{itemize}
\item (MAP) For every $\alpha \in T$ the sequence $u_\alpha^\nu \defeq u^\nu \circ \phi^\nu_\alpha : (\D, \partial \D) \to (\R^4, L)$ converges u.c.s. to $u_\alpha$ on $\D \setminus Z_\alpha$.
\item (ENERGY) If $\alpha E \beta$, then
\begin{align*}
m_{\alpha \beta} (u) = \lim_{\epsilon \to 0} \lim_{\nu \to \infty} E(u_\alpha^\nu; N_\epsilon(z_{\alpha \beta}))
\end{align*}
where $E(u_\alpha; N_\epsilon(z_{\alpha \beta})) = \int_{N_\epsilon(z_{\alpha\beta})} (u_\alpha^{\nu})^* \omega$ is the energy of $u_\alpha^\nu$ concentrating in an epsilon neighbourhood around $z_{\alpha \beta}$ in $\D$. As $z_{\alpha \beta} \in \partial \D$, $N_\epsilon (z_{\alpha\beta}) \defeq \D \cap B_\epsilon (z_{\alpha\beta})$ for the ball, $B_\epsilon (z_{\alpha\beta})$, of radius $\epsilon$ around $z_{\alpha\beta}$ in $\C$.
\item (RESCALING) If $\alpha E \beta$, then the sequence $\phi^\nu_{\alpha\beta} \defeq (\phi^\nu_\alpha)^{-1} \circ \phi_\beta^\nu$ converges to $z_{\alpha \beta}$ u.c.s. on $\D \setminus \{z_{\beta \alpha}\}$.
\item (MARKED POINTS) $z_i = \lim_{\nu \to \infty} (\phi^\nu_{\alpha_i})^{-1} (z^\nu_i)$ for $i = 1, \dots, m$.
\end{itemize}
\end{defn}
It can be shown that for a given sequence, the Gromov limit limit is unique. Gromov  convergence defines a Hausdorff topology on $\ol \M_{m}(A, \lambda)$ called the ${\bf C^\infty}${\bf -topology}. We do not include a proof of Hausdorffness here but
proofs of similar statements can be found in \cite[Section 5.4]{mcduffsalamon_fat}.

\subsection{Gromov Compactness}
\begin{thm}[Gromov compactness]\label{thm_cptness} Let $L \subset \R^4_{[a,b]}$ be a Lagrangian tangle. Fix $m \in \Z$, $A \in H_2(\R^4 , L)$, and a map $\lambda: \{1, \dots, m\} \to \Delta(L)$. Then, the moduli space $\ol \M_{m}(A, \lambda; L)$ with the $C^\infty$-topology is compact.
\end{thm}

\begin{proof}
It is sufficient to show sequential compactness. Let $u^\nu: (\D, \partial \D) \to (\R^4, L)$ be a sequence of holomorphic disks with marked points $z^\nu = (z^\nu_1, \dots, z^\nu_m)$, where $(z^\nu_1, \dots, z^\nu_m)$ is a sequence of $m$-tuples of pairwise distinct points on $\partial \D$, with $u^\nu(z_i^\nu)$ to be double points of $L$. Then, $(u^\nu, z^\nu)$ has a Gromov convergent subsequence.

This is essentially the same compactness as in \cite[Section $4.1$]{AJ}, but our Lagrangian has boundary (with transverse double points) and the ambient manifold is not compact. However, we can adapt their argument to our situation as follows.

To apply the compactness results by M. Akaho and D. Joyce in \cite[Section 4.1]{AJ}, we first extend our Lagrangian by attaching small collars to $\partial_+ L$ and $\partial_- L$ as in Lemma \ref{lem_collar_2} to get a slightly larger Lagrangian, $L' \supset L$, with the same set of double points. In particular, we have $L' \cap \R^4_{[a ,b]} = L$. We denote by $\ol \M(A, \lambda; L')$ the moduli space of stable disks with corners that have boundary on $L'$. Note that $\{[u,z] \in \ol \M(A, \lambda; L') | \rm{Im}(u) \subset \R^4_{[a ,b]}\} = \ol \M(A, \lambda; L)$. Now, we argue that if we restrict our attention to those elements of $\ol \M(A, \lambda; L')$ that have image in $\R^4_{[a ,b]}$, we get a sequentially compact space.

Consider sequence $[(u^\nu, z^\nu)] \in \ol \M(A, \lambda; L)$. 
Note that $$\rm{Im}(\partial u^\nu) \subset L \subset \R^4_{[a,b]} \defeq \{a \leq y_2 \leq b\}$$ implies, by the maximum principle for harmonic functions, $\rm{Im}(u^\nu) \subset \R^4_{[a,b]}$ and therefore the entire sequence of holomorphic disks have image outside an $\epsilon$-neighbourhood of $\partial L'$ for some $\epsilon \geq 0$. As, $L$ is the image of a compact surface with boundary, it is contained in a compact set. In particular, there exists $a_i, b_i \in \R$ for $i = 1, \dots, 4$ such that $L \subset K = [a_1, b_1] \times \dots \times [a_4, b_4]$ and all of the $u^\nu$ have image inside $K$. These two observations say that arguments for compactness carried out in \cite{AJ} or \cite{liu} will carry through and $[u^\nu, z^\nu]$ has a convergent subsequence with limit in $\ol \M(A, \lambda ; L')$. As, convergence in the Gromov sense also gives convergence of the images, the limit will also have image within $\R^4_{[a,b]}$ and therefore be contained in $\ol\M(A, \lambda; L)$.
\end{proof}

\section{Dimension Calculations and Automatic Transversality}\label{sec_dimension plus automatic}
In this section, we cover some analytical results regarding the moduli spaces $\M(u)$ we defined. For this section let $L \subset \R^4_{[a,b]}$ be a Lagrangian tangle.

\subsection{Dimension Calculations}\label{sec_dimension calculations}
In this section, we recall index counts for holomorphic disks with corners from
Suppose $u$ is a horizontal disk (Definition \ref{defn_horizontal disk}). Then, the tangent bundle of $u$ is canonically identified with $\C \times \{0\} \subset \C \times \C \simeq u^*(T\R^4)$. So, the normal bundle is canonically identified with $\{0\} \times \C \subset \C \times \C$. The real line bundle $TL$ along $\p \D$ also splits into $T L = T^1 L \times T^2 L$, where $T^1 L = \pi_1(TL)$ and $T^2 L = \pi_2(TL)$ for $\pi_j: \R^4 \to \R^2, (x_1, y_1, x_2, y_2) \mapsto (x_j, y_j)$. The Maslov index of the bundle pair $(\C, T_2L)$ is called the \textbf{normal Maslov index}, $\mu_2$. Note that the path $T^2 L|_{u(\theta)}$ for $\theta \in S^1 = \p \D$ in the Lagrangian Grassmanian is discontinuous at each marked point as the boundary jumps between leaves of the Lagrangian. We make the path $T^2 L|{u(\theta)}, \theta \in S^1$ continuous by adding the ``canonical short path" that is obtained by rotating in the clockwise (negative K\"ahler) direction. To get $\mu_2$, we count the number of half-rotations of $T^2 L$ (after completing using the canonical short paths) as we go along $\p u$ in the orientation given by the complex structure, that is, anticlockwise on $\p \D$.

We are considering the dimension of the moduli space, $\M(u)$, of unparametrized holomorphic disk with corners. The (virtual) dimension of the moduli space defined in \ref{defn_moduli space} is given by
\begin{align*}
\dim \M(u) = \mu_2(u) + 2 - b - 2g.
\end{align*}
See, for example, \cite[Equation (2)]{HLS}.
In our case, $b$ is the number of boundary components, hence $b = 1$. Genus is $g = 0$ as we are working with disks. So 
\begin{align*}
\dim \M(u) = \mu_2(u) + 1.
\end{align*}

In the case of Lagrangian tangles, we can track the rotation of $T^2 L$ along the boundary of a disk explicitly. We can assume for simplicity of calculations that at each self-intersection point in $\p_+ L$, $T^2 L^h$ (the normal direction for the higher leaf) is spanned by $\p_{y_2}$ and $T^2 L^l$ is spanned by $\p_{y_2}+\p_{x_2}$. This is equivalent to assuming we have a parametrization like \ref{eqn_lagrangian movie parametrization} in Section \ref{sec_lag movies} with a standard collar near $\p_+ L$ (see Remark \ref{rem_standard collar}) and as we increase $y_2 = t$, the top leaf is stationary and the bottom leaf moves up to meet the top leaf. In general, $T^2L^l$ is spanned by some vector in the first quadrant, but that does not change our calculations as we close up the path using canonical short paths. Similarly, $T^2 L^h$ may be spanned by $\p_{y_2} + \p_{x_2}$ or any vector in the second quadrant but this does not change the calculated number in the end. Further note that,
along each of the continuous arcs of $\p u$, $T^2 L$ is never parallel to $x_2$-axis as we assume that $L$ intersects $\R^3_t$ transversely for $b-\epsilon < t \leq b$, for small $\epsilon > 0$ (see Definition \ref{defn_lag tangle}).

These assumptions give us that we can calculate $\mu_2$ by referring to the following tables: 

\noindent If the boundary of $u$, $\im(\partial u)\subset \partial_+ L$:
\begin{center}
\begin{tabular}[center]{|c| c |c|}
\hline
For every path running from   & to  & add\\
\hline
negative corner& negative corner& $-1$\\
negative corner& positive corner&$ -\frac{1}{4}$\\
positive corner& positive corner& $0$\\
positive corner& negative corner& $-\frac{3}{4}$\\
\hline
\end{tabular}
\end{center}
For example, if we go from a positive corner to a positive corner in $\partial_+ L$, we go from a higher leaf to a lower leaf adding a rotation of $\frac{1}{4}$. Then we add the rotation from the canonical short path at the terminal corner to add an additional $-\frac{1}{4}$. This leaves us with a total of $0$. The rest of the calculations are similar.

For an intersection point in $\p_- L$, all the computations are mirrored. We may assume that $T^2 L^h$ is spanned by $\p_{x_2} + \p_{y_2}$ and $T^2 L^l$ by $\p_{y_2}$. This is same as assuming that the lower leaf remains stationary with respect to $y_2$ near the intersection point and the higher leaf moves away by moving up in $x_2$. Similarly to the $\p_+ L$ case, $T^2 L$ is never parallel to the $x_2$-axis as $L$ intersects $\R^3_t$ transversely for $a \leq t < a+ \epsilon$. So, 
if $\im(\partial u) \subset \partial_- L$:
\begin{center}
\begin{tabular}{|c| c |c|}
\hline
For every path running from   & to  & add\\
\hline
negative corner& negative corner& $0$\\
negative corner& positive corner&$ -\frac{3}{4}$\\
positive corner& positive corner& $-1$\\
positive corner& negative corner& $-\frac{1}{4}$\\
\hline
\end{tabular}
\end{center}

It is clear from the above tables that for any holomorphic disk with corners, $u$, with $\partial u \subset \partial_+ L$, a negative corner adds at least $-1$ to the virtual dimension. Whereas, if $u$ had all positive corners, the virtual dimension $\dim(\M([u]; q_1, \dots, q_m)) = 1$ as $\mu_2 = 0$. Similarly, if $\partial u \subset \partial_- L$, then any positive corner decreases the virtual dimension from $1$ by at least $1$, and if $u$ has all negative corners, then $\dim(\M([u]; q_1, \dots, q_m)) = 1$. This gives us the following lemma:
\begin{lem}\label{lem_dim of horizontal disk}
Suppose $u$ is a horizontal holomorphic disk with corners, that has boundary on a Lagrangian tangle $L$. 
\begin{itemize}
\item If $\partial u \subset \partial_+ L$, then $\dim \M(u) = 1$ if and only if $u$ has all positive corners.
\item If $\partial u \subset \partial_- L$, then $\dim \M(u)= 1$ if and only if $u$ has all negative corners.
\end{itemize}
\end{lem}

We also need to keep track of the {\bf total Maslov index}, namely, the Maslov index of the bundle $(u^* T\R^4, u^* L)$. For horizontal disks, keeping in mind the above discussed splitting, we refer to the Maslov index of the tangent bundle $(\R^2, T^1 L)$ as $\mu_1$. This records the number of half rotations of the tangent of $\p u$. Note that, if we have $m$ corner points, this number is always $2 - m$. The total Maslov index, denoted by $\mu(u)$, is obtained by taking  Maslov$(\p u) + m$, where Maslov$(\p u) = \mu_1 + \mu_2$ in our notation (see \cite[Section 2.1]{RET}). Notice that the two contributions of $m$ cancel each other out and we get that
\begin{align*}
\mu(u) = \mu_2(u) + 2.
\end{align*}
For a sanity check, let us compute the expected dimension of $\M(u)$ using $\mu(u)$. From \cite[Section 2.1]{RET},
\begin{align*}
\dim \M(u) = (n-3) + \mu(u) = (2-3) + \mu_2(u) + 2 = \mu_2(u) + 1,
\end{align*}
which matches our earlier calculations using the normal bundle.

\subsection{Automatic Transversality}\label{sec_automatic transversality}
In this section we extend the arguments in \cite{HLS} to the case of holomorphic disks with corners, to get transversality. For this section, we do not really use the Lagrangian tangle structure of $L$. In fact, all results in this section hold if $L$ is an immersed Lagrangian with only transverse double point singularities. 

First, we formalize the idea that nearby holomorphic disks are holomorphic sections of the normal bundle. This is similar to \cite[Section 1]{HLS}, with changes to account for the presence of corners in the boundary in our case. Second, in Section \ref{sec_cx str on normal bundle}, we discuss how the normal bundle we defined has a complex structure. Then we state the main theorem of this section, Theorem \ref{thm_automatic transversality}, which is automatic transverality in the presence of corners. Then, in Section \ref{sec_smoothing the bundle}, we discuss how we can obtain a smooth bundle pair relevant to the problem by adjusting the bundle over the corners. Finally, We prove Theorem \ref{thm_automatic transversality} in Section \ref{sec_automatic transversality proof}.

An unparametrized regular holomorphic disk with corners in $\R^4$ that has boundary on $L$ is a real immersed surface of genus zero and one boundary component, $\Sigma \subset \R^4$ of class $C^0$ that is $C^2$ on $\wt \Sigma \defeq \Sigma \setminus \{q_1, \dots, q_m\},$ with $\partial \Sigma \subset L$, and $q_i \in \partial \Sigma \cap \Delta(L)$, such that the tangent space at each point $z \in \wt \Sigma$ is $J$-invariant. To simplify notation, we assume $\Sigma$ is embedded. Let us denote by $\p \wt \Sigma \defeq \p \Sigma \cap \wt \Sigma$. Note that Definition \ref{defn_hol disk} gives us a parametrized holomorphic disk with corners.

 If $\Sigma$ is such a disk, then $J$ induces on $T\wt \Sigma$ an integrable almost complex structure $j$, that is, a Riemann surface structure on $\wt \Sigma$ such that the identity map from $\wt \Sigma$ to $\R^4$ defines a parametrization of the holomorphic curve with Lagrangian boundary and corners, that is, it satisfies $df \circ j = J(f) \circ df$ on points of $\wt \Sigma$ and  $\partial \Sigma \subset L$.

Let us now fix such a curve $\Sigma_0$ with corners $q_1, \dots, q_m \in \Delta(L)$ and denote $$\wt\Sigma_0 \defeq \Sigma_0 \setminus \{q_1, \dots, q_m\}.$$ By the tubular neighbourhood theorem, we can identify $\R^4$ near $\wt \Sigma_0$ (as a differential manifold) with the normal bundle $N = T_{\wt \Sigma_0} \R^4 / T\wt\Sigma_0$. Note that here we are defining $T\wt\Sigma_0$ on the boundary points by taking the vector space spanned by the half-space of tangent vectors to $\Sigma_0$. 
Then, $N$ is a complex line bundle over $\wt \Sigma_0$. Here the tubular neighbourhood of $\Sigma_0$ is not really an open neighbourhood of $\wt \Sigma_0$, as $\Sigma_0$ has boundary. For points of the boundary the tubular neighbourhood does not include a ball around it, instead only half-balls are included.

We get a totally real subbundle of $N|_{\partial \wt\Sigma}$ given by $$F \defeq T_{\partial \wt\Sigma_0} L / T(\partial \wt \Sigma_0) = TL \cap N_{\partial \wt \Sigma_0}.$$ Here intersection means the image of $TL$ under the projection $TV \to N$ coming from the identification $N =TV / T\wt\Sigma$. This is a real line bundle along $\partial \wt \Sigma_0$ that cannot be extended to $\partial \Sigma_0$ as the boundary jumps from one leaf of the Lagrangian to another at the corners.

Every (real) surface that is $C^1$-close to $\Sigma_0$ and belongs in $\M(\Sigma_0)$ can be seen as the graph of a $C^1$-small section $\phi \in \Gamma(N)$ that has zeroes at $q_1, \dots, q_m$ and $\phi|_{\partial \Sigma_0} \in \Gamma(F)$. In fact, there exists $a_1, \dots, a_j > 0$ such that near the corner  $q_j$, 
 \begin{align*}
 \phi(z) = \exp_{q_j} \xi(z)
\end{align*}  
where $\xi(z) \in L^p_{k, a_j}(T_{q_j} \R^4)$. Here $\xi \in L^p_{k, a_j}(T_{q_j} \R^4)$ if $\xi$ converges in the norm
\begin{align*}
||\xi||_{k,p, a} \defeq ||e^a \xi||_{k,p}
\end{align*}
where $e^a(z) = e^{a\tau(z)}$ for $z \mapsto (t(z,)\tau(z) )$ identifying a deleted neighbourhood of $q_j$ with $[0,1] \times [0, \infty)$, $\lim_{\tau \to \infty} (t(z), \tau(z)) = q_j$ uniformly in $t$ (See \cite{floer_morselag}). Let $a = (a_1, \dots, a_m)$ and $W^{k,p}_a(N,F)$ denote $W^{k,p}$-sections of $(N,F)$ with the above convergence property at each $q_j, j=1, \dots, m$. Assume $k >1$ and $p > 2$ guaranteeing $W^{k,p} \subset C^0$.

In particular, $\Sigma_0$ is identified with the zero section. From now on we restrict our attention to the tubular neighbourhood of $\Sigma_0$ and therefore, assume $V = N$.

Fix $z_1, \dots, z_m$ ordered anti-clockwise on $\partial \D$. 
Let us pick a parametrization of $\Sigma$,
\begin{align*}
u: \D \to \Sigma
\end{align*}
with $u(z_i) = q_i$. Denote $\wt \D \defeq \D \setminus \{z_1, \dots, z_m\}$. Then we get a bundle $u^*N \to \wt \D$ by pulling back the normal bundle $N \to \Sigma$ via $u$. 
The total space of $u^* N$ gets a $C^\infty$ almost-complex structure by pulling back the complex structure on $TN$  from $\R^n$, via $u$.

\subsubsection{Complex Structure on the Normal Bundle}\label{sec_cx str on normal bundle}
In this section we discuss the complex structure on the bundle $u^*N \to \wt \D$ for which $F$ is totally real.
Let us fix a connection $\nabla$ on $TN$ compatible with $J$. We may take the standard connection on $\R^4 = \C^2$. Using the decomposition $TN = T\Sigma_0 \oplus N$ into horizontal and vertical subspaces, we define
\begin{align*}
J_0(n) = \begin{pmatrix}
\jj(x) & 0 \\ 0 & \ii
\end{pmatrix}
\end{align*}
where $x \in \Sigma_0$ is the projection of $n$, $\jj$ is the (integrable) complex structure on $T\Sigma_0$ and $\ii$ is the fiberwise complex structure on $N$. On $\Sigma_0$ we have $J = J_0$, and near $\Sigma_0$, we can write $J(n) = \Phi(n)^{-1} J_0(n) \Phi(n)$ where $\Phi(n)$ is an automorphism of $TN$ equal to the identity on $\Sigma_0$. 
This complex structure can be pulled back to the trivial bundle $\C \times \wt \D$ by taking $J(v) = J(u_* (v))$ and $\jj(z) = \jj(u(z))$ for all $v \in \C \times \wt\D$ and $z \in \wt \D$. We again denote this by $(N,F) \to \wt \D$.

As shown in \cite[Section 1]{HLS}, there is an operator $\delbar_u$ going from an open subset $U \subset W^{k,p}_a(N,F)$ containing the zero section to $\Omega^{0,1}(N,F)$, and is given by
\begin{align}\label{defn_delbar}
\delbar_u (\phi) = P(\phi) \circ j - \ii \circ P(\phi)
\end{align}
where $P: T\Sigma_0 \to N$. Zeroes of this section are the nearby holomorphic curves in the moduli space $\M(u)$. By nearby we mean any elements with image in the tubular neighbourhood, $N$.
Let us denote the linearization of $\delbar_u$ as $L_u$. We consider it as an operator
\begin{align*}
L_u : W^{1,p}_a(N,F) \to L^p(\Omega^{0,1}(N,F)).
\end{align*}
This linearization of the equation for holomorphic curves is the usual linearization, but now it acts only on sections that have boundary values lying in $F$ and appropriate zeroes at $z_i$'s. It is elliptic, thanks to the Lagrangian boundary conditions, with index given by 
\begin{align*}
\ind(L_u) = \mu_2 + 2 - b - 2g = \mu_2 + 1,
\end{align*}
where $\mu_2$ is the Maslov index of the bundle pair $(N, F)$, that is, the normal Maslov index as in the previous section. Here, $b = 1$ is the number of boundary components and $g = 0$ is the genus of $\Sigma_0$.

\begin{thm}\label{thm_automatic transversality}
If $\mu_2(u) \geq -1$, then $L_u$ is surjective. Thus, the space of holomorphic disks with corners, $(\Sigma, \p \Sigma) \subset (\R^4, L)$, near $(\Sigma_0, \p \Sigma_0)$ and with the same corners as $\Sigma_0$ is a manifold of dimension $\mu_2 + 1$.
\end{thm}

\subsubsection{Smoothing the Bundle}\label{sec_smoothing the bundle}
In this section we construct another complex line bundle pair $(E,G)$ on $(\D, \jj)$ with a generalized $\delbar$-operator that we may think of as the $\delbar$-operator for a different complex structure $\jj$ on $\D$ such that we get an isomorphism between holomorphic sections of this bundle and our original normal bundle, $u^*(N, F) \to \wt \D$. We call this a ``smoothing" because we are changing the bundle in such a way that the real line bundle $u^* F$ along the boundary $\p \D$ ``becomes" continuous over all points.

The main idea here is to reparametrize the bundle locally near the corners such that the Lagrangian lines have the same limit from both sides at a marked point. Suppose the Lagrangian near a double point was given by $L_1 \cup L_2$, for $L_1 = \{ y_1 = 0, y_2 = 0\}$ and $L_2 = \{x_1 = 0, x_2 = 0\}$, and our holomorphic corner is the first quadrant in $\C \times \{0\}$, that is, $I = \{y_2 = x_2 = 0, x_1 \geq 0 , y_1 \geq 0\}$. We can apply the logarithm $\log$ to see our corner as the strip $(-\infty, 1] \times [0, \frac{\pi}{2}]$. Simultaneously, we can apply the derivative of $\log$, multiplication by $\frac{1}{z}$ on each fiber. 
Along the $y_1$-axis boundary of $I$, rotating by $\frac{1}{z}$ makes 
\begin{align*}
\lim_{\ii y_1 \to 0} \frac{1}{\ii y_1} (T_{(\ii y_1)} L_2) = \lim_{\ii y_1 \to 0} \frac{1}{\ii y_1}  \{x_1 = x_2 = y_1 = 0\} = \{x_1 = y_1 = y_2 = 0\} = \lim_{ x_1 \to 0} \frac{1}{x_1} (T_{(x_1)} L_2).
\end{align*}
 So now, at the marked point $0$, both limits of $TL_2$ match. So, we can extend the Lagrangian line bundle along the boundary to the corner point.

We want to do a similar construction at each corner of $\Sigma$.
Recall the notation $T^1 L_1$, $T^1 L_2$ denote the projections of the tangent spaces $TL_1$, $TL_2$ to the tangent bundle of $\Sigma$, and $T^2 L_1$ and $T^2 L_2$ projections to the normal bundle. At any of the marked pints, $z_j \in \p \D$, the real subbundle $u^* F$ has two limits coming from either side. Let us denote
\begin{align*}
u^* F_1^j \defeq \lim_{t \to 0-} u^* F_{z_j e^{it}}, \,\,u^* F_2^j \defeq \lim_{t \to 0+} u^*F_{z_j e^{it}}.
\end{align*}
By restricting $u$ to the boundary $\partial \D$ with the anti-clockwise orientation we get a parametrization of $\partial \Sigma_0$ that we denote by $l$. We denote the parametrization in the opposite direction by  $\bar{l}$. Note that $l$ is not differentiable at the points $z_1, \dots, z_m$. We denote by $l'(z_j)$ the one-sided derivative of the part of $l$ starting at $z_j$, that is, of $l|_{[z_j, z_je^{i\delta})}$ for some small $\delta > 0$. Similarly, $\bar{l}'(z_j)$ is the derivative of $\bar{l}|_{(z_je^{-i\delta}, z_j]}$ at $z_j$. Let 
\begin{align*}
v_1^j = \frac{l'(z_j)}{||l'(z_j)||}, \,\,v_2^j = \frac{(-l)'(z_j)}{||(-l)'(z_j)||}.
\end{align*}
Then, $v_1^j$ is a unit vector  that generates $T_{z_j}^1 L_1$ in $\R^2$ and $v_2^j$ generates $T_{z_j}^1 L_2$ (as $\R$-vector spaces.) Let 
\begin{align*}
v_1^j = e^{i\theta_1^j} \text{ and } v_2^j = e^{i\theta_2^j}. 
\end{align*}
Then, $\theta^j \defeq \theta_1^j - \theta_2^j$ is the {\bf angle swept} by $\Sigma_0$ at $q_j$. Let $ I_j$ denote the cone in $\R^2$ between $T_{p_j}^1 L_1$ and $T_p^1 L_2$, namely,
\begin{align*}
I_j \defeq \{av^j_1 + bv^j_2| a\geq 0, b \geq 0\},\,\,\,\,
\wt I_j \defeq \{av^j_1 + bv^j_2| a\geq 0, b \geq 0\} \setminus \{(0,0)\}.
\end{align*}

Pick a small $\epsilon > 0$. We will need to take $\epsilon$ small enough so that some properties are satisfied which will become clear as the section progresses. For any $p \in \R^2$, let $B_\epsilon(p)$ be the ball of radius $\epsilon$ around $p$ and $\wt B_\epsilon (p) = B_\epsilon \setminus \{ p \}$ the deleted ball.
Let $C_j \defeq (B_\epsilon (z_j) \cap \D)$ and  $\wt C_j \defeq (\wt B_\epsilon (z_j) \cap \D)$. For small enough $\epsilon > 0$, there exists injective map
\begin{align*}
\rho_j :  C_j \to  I_j
\end{align*}
such that $\rho_j(0) = 0$ and $q_j + \rho_j : C_j \to \C$ is $C^1$-close to the inclusion $\iota_j : C_j \to \C$.

On $\wt I_j $ we have a biholomorphism, 
\begin{align*}
\wt \psi_j : \wt I_j \to S_j, \quad \wt \psi_j (z)  = \log ( e^{-i \theta_j} z),
\end{align*}
to the strip $S_j \defeq [0, \theta_j] \times (-\infty, R_j]$ for some $R_j \in \R$,which depends on $\epsilon$.

On $\wt C_j$, define coordinate chart
\begin{align*}
\psi_j \defeq \wt \psi_j \circ \rho_j :  \wt C_j \to S_j.
\end{align*}

Now we apply a change of coordinates on the fibers of the normal bundle.
Let
\begin{align*}
w^j_1 = \lim_{t \to 0^-} w_{z_je^{\ii t}}, \quad w^j_2 = \lim_{t \to 0^+} w_{z_je^{\ii t}},
\end{align*}
where $w_{z_j}$ is the unit vector of $F_{z_j}$ that lies in the upper half plane $\mathbb{H} \subset \C$ with respect to the chosen trivialization. Then $w_1^j$ and $w_2^j$ are unit vectors in $F_1^j$ and $F_2^j$, respectively. Suppose
$w^j_1 = e^{i \alpha^j_1}$ and $w^j_2 = e^{i \alpha^j_2}$ for $\alpha_1^j, \alpha_2^j \in S^1$. Let $\alpha^j = \alpha_1^j - \alpha_2^j$.
Define a real-valued function 
\begin{align*}
f_j : [\theta^j_1, \theta^j_2] \to [\alpha^j_1, \alpha^j_2],\quad
x \mapsto \frac{(x - \theta_1)(\alpha^j )}{(\theta^j)} + \alpha_1.
\end{align*}
Define a continuous map, using polar coordinates on $\C = \R^2$,
\begin{align*}
\wt \eta_j : \wt I_j  \to \C,\quad
re^{i}  \mapsto r e^{ i f_j(x)}.
\end{align*}
Then, $\wt \eta_j( T^1_{z_j} L_1) \subset T^2_{q_j} L_1$, and $\wt \eta_j(T^1_{q_j} L_2) \subset T^2_{q_j} L_2$.  As we removed $0$, $\wt \eta_j$ is holomorphic. 

Define smooth ``cut-off" functions $\beta_j : \C \to [0,1]$, for $j =1,2$, such that $\beta_1 + \beta_2 \equiv 1$, and
\begin{align*}
\beta_1 = \begin{cases}
1 \text{ on } B_{\frac{\epsilon}{4}} (0) \\
0 \text{ on } B_{\epsilon}(0) \setminus B_{\frac{3\epsilon}{4}} (0)
\end{cases}, \,\,\,\,\, \beta_2 = \begin{cases}
0 \text{ on } B_{\frac{\epsilon}{4}} (0) \\
1 \text{ on } B_{\epsilon}(0) \setminus B_{\frac{3\epsilon}{4}} (0)
\end{cases}.
\end{align*} 

Define a smooth function $\eta: \wt \D \to \C$ as follows.
\begin{itemize}
\item On $\wt C_j$ define $\eta(z) = \beta_1(u^{-1}(z) - u^{-1}(z_j))\wt \eta_j(z) + \beta_2(u^{-1}(z) - u^{-1}(z_j))$,
\item on $\wt{\D} \setminus (\cup_j \wt C_j$ define $\eta(z) = 1$.
\end{itemize}
Note that $ \eta$ is a nowhere vanishing diffeomorphism on $\wt \D$. 

Take the trivial bundle $\wt E \defeq \C \times \wt \D$ on $(\wt \D, \ii)$. We get a real subbundle along $\p \wt \D$ given by $\wt G_z \defeq \eta(z) F_z$. Note that, because of how $\eta$ was defined, $G$ can be extended to a real subbundle of $E \defeq \C \times \p \D$ over $\D$. We denote this bundle by $(E,G)$. So we have
$$
\begin{tikzcd}
(\wt E,\wt G) \arrow[hookrightarrow, r] \arrow[d]& (E,G) \arrow[d]\\
\wt \D \arrow[hookrightarrow, r]& \D
\end{tikzcd}
$$
with the obvious inclusions. Note that all sections of $(\wt E, \wt G) \to \wt \D$ can be extended over the marked points to get sections of $(E, G) \to \D$ which are $W^{1,p}$-regular. Holomorphic sections of $(\wt E, \wt G) \to \wt \D$ have a unique extension.

Note that, if we choose $\epsilon$ small enough in the above construction, the rotation we are introducing  via multiplication by $\eta(z)$ is the same as the rotation added  from the addition of the canonical short path. So, the bundle pair $(E, G)$ on $\D$ that has the same Maslov number as the normal Maslov number $\mu_2$ of $u$, which is the same as $\mu(N, F)$ defined via adding canonical short paths, namely, 
\begin{align}\label{eqn_equality of Maslov indices}
\mu(E, G) = \mu_2(u) = \mu(N,F).
\end{align}

\subsubsection{Proof of Theorem \ref{thm_automatic transversality}}\label{sec_automatic transversality proof}
In this section, we conclude the proof of Theorem \ref{thm_automatic transversality} by putting together the constructions from Sections \ref{sec_cx str on normal bundle} and \ref{sec_smoothing the bundle}.

The operator $\delbar_u$ as described in Equation (\ref{defn_delbar}) gives an operator $\delbar_{u,n}$ going from an open subset $U \subset \Gamma(u^* N, u^* F)$ to $\Omega^{0,1} (E, G)$ given by
\begin{align*}
\delbar_{u,\eta} (\phi) = \delbar_u(\eta^{-1} \phi), \quad phi \in W^{k,p}(E, G).
\end{align*}
For any $W^{k,p}$ section $\phi \in W^{k,p}(ED, G)$, the section $\eta^{-1} \phi \in \Gamma(E, G)$ is defined by extending (non-uniquely) over the marked points. This non-uniqueness does not impact well-definedness as we work with $W^{k,p}$ space and the marked points are a measure zero set.

Let $L_{u,\eta}: \Gamma(E,G) \to \Omega^{0,1}(E,G)$ be the linearization of $\p_{u, \eta}$ at the zero section. Then, $L_{u, \eta}$ is a generalized $\delbar$-operator as shown in \cite[Section 2]{HLS}.
Equation \ref{eqn_equality of Maslov indices}, implies that
\begin{align*}
\ind (L_{u,\eta}) = \ind (L_u).
\end{align*}
Therefore, $\ind (L_u) \geq 0$ implies $\ind (L_{u,\eta}) \geq 0$. This implies, by \cite[Theorem 2]{HLS}, that $L_{u,\eta}$ is surjective.

We note that the following diagram commutes:
$$ 
\begin{tikzcd}
W^{1,p}_a (N,F) \arrow[r, "L_u"]\arrow[d, "\eta"] & L^p(\Omega^{0,1} (N,F))\arrow[d, "d\eta"]\\
W^{1,p}(E,G) \arrow[r, "L_{u,\eta}"] &L^p(\Omega^{0,1}(E,G))
\end{tikzcd}
$$
where $\eta$ is point-wise multiplication by $\eta(z)$ at $z \in \wt \D$, and $d\eta$ is point-wise multiplication of sections by $d \eta$. With the above definitions of the spaces of sections, $\eta|_{\ker (L_u)} : \ker(L_u) \to \ker(L_{u, \eta})$ is an isomorphism. Additionally, we have canonical isomorphisms $\ker(L_u)^* \cong \coker L_u$ and $\ker(L_{u,\eta})^* \cong \coker L_{u, \eta}$. So, if $L_{u, \eta}$ is surjective, $L_u$ is surjective. This completes the proof of Theorem \ref{thm_automatic transversality}. 

\section{Properties of Boundaries of Moduli Spaces}\label{sec_boundary points}
In this section, we describe properties of the stable disks that appear in the boundary of moduli spaces $\ol \M(A, \lambda)$ or $\ol \M(u)$, that we defined in the previous sections. 
\begin{defn}
We define and distinguish between two types of nodes. For a stable holomorphic disk $[{\bf u, z}]$, consider a node $u_\alpha(z_{\alpha\beta}) = u_\beta(z_{\beta\alpha}) = q$.
\begin{itemize}
\item \textbf{Corner node:} The node $q$ is called a {\bf corner node} if $q$ is a corner for both $u_\alpha$ and $u_\beta$. This, by Definition \ref{defn_hol disk} and discussion in proof of Lemma \ref{lem_interior breaking in pairs}, can only happen at double points of $L$, that is, if $q \in \Delta(L)$.
\item \textbf{Smooth node:} The node $q$ is called a {\bf smooth node} if $q$ is a smooth point for both $u_\alpha$ and $u_\beta$.
\end{itemize} 
\end{defn}
We want to understand what all the boundary points of a given one-dimensional moduli space $\ol \M(u)$, look like. 
\begin{defn}
A stable holomorphic disk with corners, $[{\bf u, z}]$, is said to be of \textbf{Type 1} if there exists $\alpha E \beta$ such that $u_\alpha(z_{\alpha \beta}) = u_\beta(z_{\beta\alpha})$ is a smooth nodal point. Otherwise, we say $[{\bf u, z}]$ is of \textbf{Type 2}. Note that, if $[{\bf u, z}]$ is of Type 2, then for all $\alpha E \beta$, $u_\alpha(z_{\alpha \beta}) = u_\beta(z_{\beta\alpha})$ is a corner for $u_\alpha$ and $u_\beta$.
\end{defn}
Note that any holomorphic disk with corners, $u: (D, \partial D) \to (\R^4, L)$, is of Type 2 by the above definition, as it has no nodes and hence, no smooth nodes. Holomorphic disks with corners also share many properties with type 2 broken disks. So, it a natural choice to consider holomorphic disks with corners as Type 2.

\begin{defn}\label{defn_aligned disk}
Let $L \subset \R^4_{[a,b]}$ be a Lagrangian tangle.
Consider a horizontal holomorphic disk with corners, $u: (\D, \p \D) \to (\R^2_a \sqcup \R^2_b, \p_- L \sqcup \p_+ L)$.
Let $[z,z']$ denote the part of $S^1$ going from $z$ to $z'$ in the anticlockwise direction.
The boundary $\p u$ can be divided into $m$ arcs given by 
\begin{align*}
\gamma_j = u|_{[z_j, z_{j+1}]}: &[z_j, z_{j+1}] \to L & j = 1,2, \dots, m-1;\\
\gamma_m = u|_{[z_m, z_1]}: &[z_m, z_1] \to L.&
\end{align*}
Fix an orientation of $L$. This gives an orientation on $\p L$.  Then, each $\gamma_j$ gets two orientations --- one from the complex orientation of $u$, that is, from $q_j$ to $q_{j+1}$, and another from the orientation of $\p L$. Accordingly, we assign a sign to each $\gamma_j$ -- \textbf{positive} if these two orientations match and \textbf{negative} otherwise.

We call the horizontal disk, $u$, \textbf{aligned} if the signs on all $\gamma_j's$ are equal.
\end{defn}

\begin{defn}\label{defn_big disk hol}
A horizontal  holomorphic disk, $u$, that has boundary on a Lagrangian tangle $L$ is called \textbf{big} if:
\begin{itemize}
\item $u$ has all convex corners;
\item $\dim \ol \M(u) = 1$ or equivalently $\mu_2(u) = 0$;
\item $u$ is an aligned disk as in Definition \ref{defn_aligned disk};
\item either $a(\p_+L) = a(\p_- L) = 0$ or $\int_{\D} (u^* \omega) \leq a(\p_+ L)$.
\end{itemize}
\end{defn}
Note that the definition of big holomorphic disk is analogous to the definition of big disk bound by a pair of diagrams (Definition \ref{defn_big little disk}).

 \begin{figure}[h!]
\begin{center}
  \includegraphics[width = 5in]{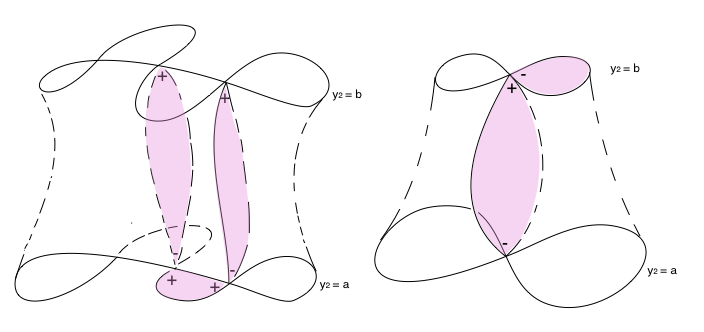}
  \end{center}
 \caption{Depiction of possible Type 2 stable holomorphic disks.}
  \label{fig:type 2 disks}
\end{figure}
Note that for a big holomorphic disk $(u,z)$, $\M(u)$ is a compact $1$-dimensional manifold.
In the following theorem, we state all the properties of boundary points $\M(u)$. The proof is spread over the following subsections - part 1 is proved in Section \ref{sec_horizontal disks are boundary points}, parts 2 and 3 are proved in Section \ref{sec_type 1 dont appear}, part 4 is proved in Section \ref{sec_constraints on signs}, and part 5 is proved in Section \ref{sec_non horizontal disk sign}.
\begin{thm}\label{thm_properties of type 2 disks}
Suppose $L \subset \R^4_{[a,b]}$ is a Lagrangian tangle and $(u, {\bf z})$ is a big holomorphic disk with corners that has boundary on $L$. Consider the moduli space $\ol \M(u)$.
\begin{enumerate}
\item Horizontal disks in $\ol\M(u)$ are boundary points. Non-horizontal holomorphic disks with corners, which do not have any nodes, are not boundary points.
\item 
The number of stable holomorphic disks of Type 2 in $\p \ol \M(u)$ must be even.
\item (\textsc{Nodal Points}) For a stable disk, $[{\bf v, z}] \in \ol \M(u)$, each nodal point $z_{\alpha \beta}$ is mapped to a double point of $L$;
\item (\textsc{Signs}) For a stable disk, $[{\bf v, z}] \in \ol \M(u)$, two nodal disks have opposite signs at the common node, that is, if $q = v_\alpha (z_{\alpha \beta}) = v_\beta(z_{\beta \alpha})$ then 
\begin{align*}
\sign_{q} (v_\alpha) = - \sign_{q} (v_{\beta}).
\end{align*} 
\item (\textsc{Existence of Horizontal Disks})  For any stable disk, $[{\bf v,z}] \in \ol \M(u)$, every node has exactly one horizontal disk attached to it. 
\end{enumerate}
 \end{thm}

 \subsection{Horizontal Disks are Boundary Points}\label{sec_horizontal disks are boundary points}
In this section we prove Part 1 of Theorem \ref{thm_properties of type 2 disks}. We show that horizontal disks give us boundary points of $\ol \M(u)$. We also show that a non-horizontal disk (without nodes) cannot be a boundary point of the moduli space.

Given a Lagrangian tangle $L$, consider a small extension to get an immersed Lagrangian $L'$ as in Remark \ref{lem_collar_2}. Let $\M(u; L')$ denote the moduli space of variations of $u$ in $L'$. Note that, for a disk $u$ that has boundary on $L$, $\dim \M(u;L) = \dim \M(u;L')$.
\begin{lem}\label{lem_no intersection hor disk}
Suppose $[{\bf u}, {\bf z} = (z_1, \dots, z_m)]$ is a horizontal disk with $m$ corners and \\$\dim(\M(u; L')) = 1.$ Then, any holomorphic disk in $\M(u, \lambda; L')$, which is $C^1$-close to $[{\bf u, z}]$, intersects $u$ at only the corner points, $u(z_1), \dots, u(z_m)$.
\end{lem}
In fact, it is not necessary that $u$ is a horizontal disk but we will only be using this Lemma for horizontal disks and the notation is easier in this case. So, we state and prove the above lemma only for horizontal disks. The proof extends easily to all holomorphic disks with corners and boundary on a Lagrangian tangle.
\begin{proof}
Just as in proof of Theorem \ref{thm_automatic transversality}, we get a normal bundle 
$$
\begin{tikzcd}
(N, F) \arrow[d]\\
\wt \D
\end{tikzcd},
$$
where $\wt \D \defeq \D \setminus \{z_1, \dots, z_m\}$. In fact, as $u$ is a horizontal disk the normal bundle is canonically identified with $\C \times \wt \D \to \wt \D$. 
We also obtain a bundle 
$$
\begin{tikzcd}
(E, G) \arrow[d]\\
\D
\end{tikzcd}
$$
that is related to the normal bundle $(N,F)$ via the smooth function $\eta$, as in proof of Theorem \ref{thm_automatic transversality}.
Recall that any holomorphic disk in $\M(u)$ near $\im (u)$ will be given by the graph of a holomorphic section
$$
\begin{tikzcd}
(N,F) \arrow[d, "\pi"]\\
\D \arrow[u, bend left, "s"]
\end{tikzcd}.
$$
Note that we are only considering sections that are zero at each $z_j$.
By point-wise multiplication by $\eta$, which was defined in proof of Theorem \ref{thm_automatic transversality}, we get a section $\eta \cdot s$ of $(E,G)$. We made choice of function $\eta$ such that $\eta \cdot s$ does not vanish at the marked points, $z_j$,  even when $s$, a holomorphic section of $(N, F)$, vanishes at the marked points $z_j$. 

The bundle pair $(E,G)$ has Maslov index equal to the normal Maslov index $\mu_2(u)$ and so,
\begin{align*}
\mu(E,G) = 0.
\end{align*}
Hence, by doubling we get a bundle $E \cup \ol E$ on sphere $S^2$ with Chern class
\begin{align*}
c_1(E \cup \ol E) = \mu(E,G) = 0.
\end{align*}
This means that a holomorphic section of $E \cup \ol E$ does not intersect the zero section, as all intersections would be positive intersections (positivity of intersections in dimension $4$, see \cite{mcduffsalamon_fat}). Therefore, the section $\eta \cdot s$ also cannot have any intersections with the zero section. As $\eta$ has no zeroes on $\wt \D$, this means that the original section $s$ cannot intersect the zero section outside of corner points.

\end{proof}

\begin{lem}\label{lem_horizontal disks are boundary points}
For a horizontal holomorphic disk with corners, $u$, that has boundary on a Lagrangian tangle, $L \subset \R^4_{[a,b]}$, and normal Maslov index $\mu_2(u) = 0$, $[u] \in \M(u)$ is a boundary point of $\M(u)$.
\end{lem}

\begin{proof}
Suppose, we extend $L$ to $L'$ as in Remark \ref{lem_collar_2} and hence, we extend $\ol \M(u;L)$ to $\ol \M(u;L')$. We claim that Lemma \ref{lem_no intersection hor disk} implies that $[u]$ divides the moduli space $\ol\M(u;L')$ near $[u]$ into two parts --- those with image completely contained in $\R^4_{y(-\infty, b ]}$ and those with image completely in $\R^4_{[b, \infty) }$.
As $\ol \M(u;L')$ is also a $1$-dimensional manifold and $$\ol \M(u, L) = \{[v] \in \ol \M(u;L')\,| \,\im(v) \subset \R^4_{[a,b]}\},$$ this means that $[u]$ is a boundary point of $\ol \M(u;L)$.

 Assume $\im (\p u) \subset \p_+ L = L_b$. To show the claim, suppose if possible, that $[v] \in \M(u; L')$, $C^1$-close to $u$, such that there exists $p_1 \neq p_2 \in \D$ with $\pi_{y_2} v(p_1) > b$ and $\pi_{y_2} v(p_2) < b$. By the maximum principle for harmonic functions we can assume $p_1, p_2\in \p \D$.
If both $p_1$ and $p_2$ are in the same component of $\p \D \setminus \{z_1, \dots, z_m\}$, then by continuity there must exist a smooth point $p \in \p \D$ with $\pi_{y_2}(v(p)) = b$. As $L' \cap \R^3_{b } = L_b$, this would mean $v$ is intersecting $u$, which cannot happen by Lemma \ref{lem_no intersection hor disk}. If $p_1$ and $p_2$ are on different components of $\p \D \setminus \{z_1, \dots, z_m\}$, then those entire components have to lie either entirely in $\R^4_{(b, \infty)}$ or in $\R^4_{(-\infty, b)}$, respectively. In this case, by taking a path of paths $\gamma:[0,1]\times [0,1] \to \D$, such that each path $\gamma_t : [0,1] \to \D$ starts on the component of $\p \D \setminus \{z_1, \dots, z_m\}$ containing $p_1$ and ends on the one containing $p_2$, we can conclude there is a path in the interior of $\D$ on which $\pi_{y_2} \circ v$ takes constant value $b$. This would have to mean $\pi_{y_2} \circ v$ takes value $b$ at all points of $\D$ as $\pi_{y_2} \circ v$ is harmonic. This implies $\im (\p v) = \im (\p u)$ which contradicts Lemma \ref{lem_no intersection hor disk}.

\end{proof}

Now we show that a non-horizontal disk (without nodes) cannot represent a boundary point of $\M(u)$. We do so in two steps. First, note that if a holomorphic disk with corners, $(u, z)$, has boundary $u(\p \D \setminus \{z_0, \dots, z_m\} \subset L^\circ = L \setminus \p L$, then we can find a $1$-dimension worth of variations of $(u,z)$ by looking at the normal bundle we defined in Section \ref{sec_automatic transversality}. In Section \ref{sec_cx str on normal bundle}, we showed that any solution of Equation \ref{defn_delbar} gives a nearby element of $\M(u)$.  Unlike the case of a horizontal disk (as discussed in Lemma \ref{lem_horizontal disks are boundary points}), there are no additional constraints posed on the solutions. So, $u$ which corresponds to the $0$-section of the normal bundle is an interior point.

Second, we show that any disk touching the boundary of the Lagrangian tangle, $L$,  is a horizontal disk via the following lemma. This will complete the proof that non-horizontal disks are not boundary points of a moduli space $\ol \M(u)$ for big $u$.
\begin{lem}\label{lem_cpt_1}
Let $L \subset \R^4_{[a,b]}$ be a Lagrangian tangle.
Consider holomorphic disk $u: (\D, \partial \D) \to (\R^4, L)$ with corners $q_1, \dots, q_m \in \Delta(L)$. Suppose $\im(u)$ intersects $\partial L$ at a smooth point. Then it must be a horizontal disk.
\end{lem}

\begin{proof} \footnote{This proof idea was told to me by Laurent C\^{o}t\'{e}.} Firstly, suppose $\im(u)$ intersects $\partial L$ at an interior point of $\D$. By the maximum principle for the harmonic functions, $\pi_{y_2} \circ u: \D \to \R$ where $\pi_{y_2}: \R^4 \to \R$ is given by $(x_1, y_1, x_2, y_2) \mapsto y_2$, must be a constant. So, $u$ must be a horizontal disk.

So we look at the case when $\im(u)$ intersects $\p L$ on $\im(\p u)$.
Suppose, if possible, there exists $q \in u(\partial \D) \cap \partial_+ L$. (The proof for the case where the point of intersection lies in $\partial_- L$ is analogous.) Let $q = u(p)$ for $p \in \partial \D$. As $p$ is a smooth point of $u$, there is a neighbourhood of $p$ in $\D$ where the map $u$ is 
holomorphic. 

There must exist an open ball $B \subset \D^\circ = \D \setminus \p \D$ such that $p \in \partial \bar{B}$ and $u( \bar{B}) \cap \partial_+ L = \{ q \}$. This is because, if no such ball exists, there must exist a sequence of points in $\D^\circ$ that map to $\p_+ L$ and have $p$ as a limit point. So there would be a set of points with an accumulation point mapping to the complex plane $\R^2_a$ which would imply, by Claim \ref{clm_accumulation point}, $u$ maps to $\R^2_a$ and $u$ is a horizontal disk. 

Let us denote the polar coordinates on $\R^2$ by $(r, \theta)$. Then $i(\partial_r) = \partial_\theta$. Recall $\pi_{y_2} \circ u$ is harmonic and $\pi_{y_2} \circ u (z) < b = \pi_{y_2} \circ u (p)$ for all $z \in B$.  So, by Hopf lemma, the derivative $D(\pi_{y_2} \circ u)_p(\partial_r) >0$.  By holomorphicity of $u$, $$D(\pi_{y_2} \circ u)_p(\partial_r) = D(\pi_{y_2} \circ u)_p(-i_p \partial_\theta) = (D\pi_{y_2})_q \circ Du_p \circ i_p (-\partial_\theta) = (D\pi_{y_2})_q \circ \ii_q \circ Du_p(-\partial_\theta).$$ So, 
\begin{align*}
(D\pi_{y_2})_q (\ii_q \circ Du_p(-\partial_\theta)) > 0.
\end{align*}
On the other hand, observe that $q \in \partial_+ L$ implies $\p u$ is tangent to $\partial_+ L$ at $q$. This is because the image of $\partial u$ is contained within $L$, $\partial u$ is $C^1$ near $p$ and it is touching the boundary $\partial_+ L$ at $q$. This implies $Du_p(\partial_\theta)$ is tangent to $\partial_+ L$ at $q$. So, $Du_p(\partial_\theta) \in \R^2 \times \{0\} \subset T_q \C^2$ . As $\R^2 \times \{0\}$  is closed under the action of $\ii_q$, $\ii_q \circ Du_p(\partial_\theta) \in \R^2 \times \{0\}$. So, 
\begin{align*}
D\pi_q (\ii_q \circ Du_p(\partial_\theta))  = 0,
\end{align*}
contradicting our previous observation. Thus, such a $u$ cannot exist.
\end{proof}

\begin{clm}\label{clm_accumulation point}
Suppose $u : (\D, \partial \D) \to (\R^4, L)$ is a holomorphic map, for Lagrangian tangle, $L$. Suppose $x_n \in \D^\circ$ is a sequence with limit point $p \in \partial \D$ such that $u(x_n) \in \R^2_a$ for some $a \in \R$. Then im$ (u) \subset \R^2_a$.
\end{clm}
\begin{proof}
A holomorphic function $u: \D \to \R^4$ gives two holomorphic maps to $\C$ by post composing by $\pi_j: \R^4 \to \R^2$ given by $(x_1, y_1, x_2, y_2) \mapsto (x_j, y_j)$ for $j = 1,2$. Let us denote $u_j = \pi_j \circ u$.
Note that $\im (u) \subset \R^2_a$ is equivalent to $u_2$ is the constant map to $(0,a)$. As $u(x_n) \in \R^2_a$, $u_2(x_n) = (0,a)$ for all $x_n$. Note that as the boundary of the disk is mapped to $L$, which is a real analytic and totally real surface (namely, $TL \cap \ii (TL) = \{0\}$), by using Schwarz reflection we can extend the map $u$ (and therefore $u_2$) to a larger (open) domain $E \supset \D$ such that $p \in E$. Now we have a holomorphic function $u_2$ on $E$ that is constant on a set with an accumulation point in $E$. So, $u_2$ must be constant and the image of $u$ must be completely contained in $\R^2_a$.
\end{proof}
Note that Lemma \ref{lem_cpt_1} only cares about whether $u$ was touching the boundary $\p L$ of $L$ at a smooth point of $u$. It is possible that it was touching $\p L$ at a double point $q$ of $L$ without $\p u$ jumping between leaves at $q$. As the concerned point was a smooth point of $u$, only one of the leaves would feature in the above argument.

 \subsection{Type 1 Stable Holomorphic Disks Do Not Appear}\label{sec_type 1 dont appear}
In this section we analyze nodes that appear in stable disks in the moduli spaces, $\M(u)$, we consider. This will prove Part 2 and Part 3 of Theorem \ref{thm_properties of type 2 disks}. As a first step we show that all nodes which appear away from double points of $L$ are smooth nodes in Lemma \ref{lem_interior nodes are smooth nodes}. The rest of the section shows smooth nodes do not appear when we consider moduli spaces, $\M(u)$, for aligned disks $u$ (Definition \ref{defn_aligned disk}).

\begin{lem}\label{lem_interior nodes are smooth nodes}
Suppose we have $[{\bf v,z}] \in \ol\M([u])$ with a node $q \in L$. Suppose $v_1$ and $v_2$ are the two disks, part of the stable disk $v$, whose common node is $q$, that is, for special points $z_1$ and $z_2$ on $\p \D$,
\begin{align*}
v_1(z_1) = v_2(z_2) = q \in L^\circ.
\end{align*}
Then 
$q$ is a smooth point for both $v_1$ and $v_2$. 
\end{lem}
\begin{proof}
As $v_1$ is holomorphic on the interior of $\D$, $v$ belongs to the Sobolev space $W^{k,p} (\D)$ for any non-negative integers $k$ and $p$. Let us take $k > 2$ and $p > 1$, such that $k - \frac 1 p > 1$. Then, by using the trace operator
\begin{align*}
T_k : W^{k,p}(\D) \to \prod_{l=0}^{k-1} W^{\frac{k-l-1}{p}, p}(\p \D),
\end{align*}
we see that $(v_1)|_{\p \D} \in W^{1, p} (\p \D)$. Then, by \cite[Theorem 4.5]{bedford-gaveau}, we get that $v_1$ is smooth at $q$, as $L$ is smooth at $q$. Similarly, $v_2$ is also smooth at $q$.
\end{proof}

Recall that we let $[z,z']$ denote the part of $\p \D$ going from $z$ to $z'$ in the anticlockwise direction.
Consider the moduli space $\M(u)$ where $u$ is a horizontal disk with $\partial u \subset \p L$ and corner points $q_1, \dots, q_m \in \Delta(L)$. The boundary $\p u$ can be divided in to $m$ arcs given by 
\begin{align*}
\gamma_j = u|_{[z_j, z_{j+1}]}: &[z_j, z_{j+1}] \to L & j = 1,2, \dots, m-1;\\
\gamma_m = u|_{[z_m, z_1]}: &[z_m, z_1] \to L.&
\end{align*}
Now onward, we will write $[z_j, z_{j+1}]$ for all $j$ with the understanding that $m+1$ denotes $1$, when the number of corner points is $m$. 

Suppose $[{\bf v, z}] = \lim_{\nu \to \infty} [u_\nu, {\bf z}_\nu]$, for $[u_\nu, {\bf z}_\nu] \in \M(u)$.
A boundary node of {\bf Type H} (see \cite[Definition 3.4]{liu}) is a singularity locally isomorphic to
\begin{align*}
(0,0) \in \{z^2 - w^2 = 0\}/A
\end{align*}
where $(z,w)$ are coordinate on $\C^2$, and $A(z,w) = (\ol z, \ol w)$ is complex conjugation. We can view this as the image $u^\nu (\alpha)$ of an arc, $\alpha:[0,1] \to \D$ with $\alpha(0) \in [z_j, z_{j+1}]$ and $\alpha(1) \in [z_k, z_{k+1}]$, shrinking to a point as $\nu \to \infty$. Another way to view this node is it corresponds to a boundary arc, $\gamma_j^\nu$, ``intersecting" itself or another arc, $\gamma_k^\nu$ in the limit. The other types of possible nodes as outlined in \cite[Definition 3.4]{liu} are singularities isomorphic to:
\begin{itemize}
\item[(1)] $(0,0) \in \{zw = 0\}$ (interior node),
\item[(2)] $(0,0) \in \{z^2 + w^2 = 0\}/A$ (boundary node of type E).
\end{itemize}
Interior nodes do not appear for us because $\Pi_2(\R^4)= 0$ does not allow sphere bubbles to form. Boundary nodes of type $E$ correspond to a boundary component shrinking to a point. As we consider only one boundary component, a sphere bubble would be formed if that boundary shrank to a point. So, boundary nodes of type $E$ also do not form. Thus, all nodes of $[{\bf v,z}] $ are of Type H.

Suppose $L$ is a Lagrangian tangle. So, there exists a surface $\wt L$ with boundary and an immersion $\iota: \wt L \to \R^4$ such that $L = \iota(\wt L)$. We assumed $L$ is oriented. We fix an orientation on $\wt L$ such that $\iota: \wt L \to L$ is orientation preserving.
Any continuous path $\gamma: [0,1] \to L$ has a continuous lift, $\wt \gamma: [0,1] \to \wt L$, unless it passes through a double point $q \in \Delta(L)$ and $\gamma$ jumps from one leaf to the other at $q$.

Before we proceed, we define a switching operation on arcs on a surface with boundary.
\begin{figure}[h!]
\begin{center}
  \includegraphics[width = 5in]{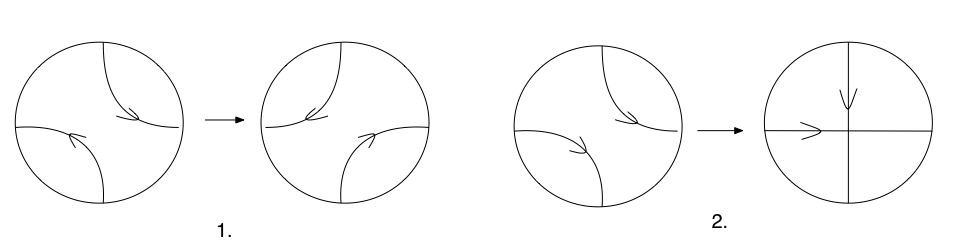}
  \end{center}
 \caption{Two cases for the switching operation defined in Definition \ref{defn_switching operation}}
 \label{fig:switching operation}
\end{figure}
\begin{defn}\label{defn_switching operation} Suppose $\wt L$ is an oriented surface with boundary. It may have any finite number of boundary components and any genus. Suppose, we have two arcs $\gamma_1$ and $\gamma_2$ on the boundary such that there exists a small disk $\D_\epsilon$ in the interior of $\wt L$ such that, up to a diffeomorphism, $(\gamma_1 \cup \gamma_2) \cap \D_\epsilon$ looks like a neighbourhood of $0 \in \R^2$ with $\gamma_1$ and $\gamma_2$ mapped on to  $\{xy = \epsilon\}$ for some small enough $\epsilon > 0$. Then, the switching operation swaps out the two hyperbola pieces $\{xy = \epsilon\}$ with the limit $\{xy = 0\}$ or with the hyperbolas $\{xy = -\epsilon\}$ depending on their orientation as in Figure \ref{fig:switching operation}. Then we match up the end point smoothly to get continuous arcs $\delta_1$ and $\delta_2$ on $\wt L$. Note that if before the switching we had arcs going from $\gamma_1(0)$ to $\gamma_1(1)$ and from $\gamma_2(0)$ to $\gamma_2(1)$, after the switching we get arcs from $\gamma_1(0)$ to $\gamma_2(1)$ and from $\gamma_2(0)$ to $\gamma_2(1)$. 

Within the small disk $\D_\epsilon$, we may have two orientation situations as shown in Figure \ref{fig:switching operation}. In the first case, the two new arcs do not intersect and in the second case, they intersect only once. 
\end{defn}

\begin{lem}\label{lem_interior breaking in pairs}
If $u$ is a big disk, Type 1 stable disks don't appear in $\ol \M(u)$.
\end{lem}
\begin{rem}
Of course, this means the number of boundary points corresponding to type 2 stable disks is even. Note that it is crucial here that $u$ is an aligned holomorphic disk, which is part of the definition of big disk. This assumption will rule out interior smooth breakings. If we did not have the aligned assumption, we can have smooth nodes in the interior of the tangle.
\end{rem}
\begin{proof}
Suppose we have a point $[{\bf v,z}] \in \ol\M([u])$ with a node $q$ on the interior of $L$. For simplicity in notation let us assume that there is a representative stable disk with two components $(v, z_1, \dots, z_m)$ and $(v', z'_1, \dots, z'_{m'})$ that are holomorphic disks with boundaries.

Note that the sets of marked (corner) points $z_1, \dots, z_m$ and $z'_1, \dots, z'_{m'}$ have to both be non-empty. 
Otherwise, we would have a disk bubble, that is, a holomorphic disk in $\Pi_2(\R^4, L^\circ)$ where $L^\circ = L \setminus \p L$ denotes the interior of $L$. We are considering a relatively exact Lagrangian $L^\circ$. This means we either have $\omega|_{\Pi_2(\R^2, L^\circ)} \equiv 0$, in which case a disk representing an element of $\Pi_2(\R^4, L^\circ)$ cannot be holomorphic as holomorphic disks have symplectic area strictly positive. The other case is when $\omega (\Pi_2(\R^4, L^\circ)) = a \Z$ and $\omega[u] \leq a$. A disk bubble would be an element $[v] \in \Pi_2(\R^4, L^\circ)$ with $0< \omega[v] < a$ which cannot happen.

As $q$ is a boundary node of type H, there is a path of elements $[u^\nu] \in \M([u]), \nu \in [0, \infty)$ limiting to the stable disk ${(v, v')}$, with domain $\Sigma$ (see Definition \ref{defn_stable map}). Let $\gamma_j^\nu$ and $\gamma_k^\nu$ be the corresponding arcs for the holomorphic disk $u^\nu$. Then $u = u^0$, $\gamma_j = \gamma_j^0$, and $\gamma_k = \gamma_k^0$.
Now, $q$ is a boundary node of type H in the interior of $L$, that is, there exist points $z \in \p \D$ and $z' \in \p \D$ such that $v(z) = v'(z') = q$ and $q \in L^\circ$.
Note that, by Lemma \ref{lem_interior nodes are smooth nodes}, $q$ is a smooth point for both $v$ and $v'$. Further, we may assume
the nodal point $q$ is a point of transverse intersection between $\p v$ and $\p v'$  as this is a generic property of the compactification. 

Let $\delta_j$ and $\delta_k$ be boundary arcs of $v$ and $v'$ that intersect, and $\wt \delta_j$ and $\wt \delta_k$ be their lifts to $\wt L$. Then
\begin{align*}
\wt \delta_j (0) = \wt\gamma_j (0), \wt \delta_j(1) = \wt\gamma_k(1), \wt \delta_k (0) = \wt\gamma_k (0), \wt \delta_k(1) = \wt\gamma_j(1).
\end{align*}
There exists $\epsilon > 0$ such that the arcs $\wt \delta_j$ and $\wt \delta_k$ are isotopes of arcs obtained by doing the switching operation inside a small disk $D_\epsilon$ on $\wt \gamma^\nu_j$ and $\wt \gamma^\nu_k$ for large enough $\nu$.
Let us denote the arcs obtained by switching $\wt\gamma_j^\nu$ and $\wt \gamma^\nu_k$ in $\D_\epsilon$ by $\hat{ \delta}_j$ and $\hat{\delta}_k$ that are respectively isotopic to $\wt \delta_j$ and $\wt \delta_k$. As the switching operation is local, the only intersections of $\hat{\delta_j}$ and $\hat{\delta_k}$ are inside the small disk $\D_\epsilon$. So, they will intersect in $0$ or $1$ points depending on whether we are in Case $1$ or Case $2$ of Figure \ref{fig:switching operation}.

There exist strong deformations $f_\nu : \D \to \Sigma$ in the sense of \cite[Definition 5.5]{liu}, such that $f_\nu^{-1}(z \sim z')$ is an arc in $\D^\nu$ for each $\nu \in [R, \infty)$, and $u_\nu \circ f_\nu^{-1}$ converges uniformly on contact sets to $(v,v')$. We can assume, by changing the parameter space, that this holds for all $\nu \in [0, \infty)$. Pick a parametrization $\alpha^\nu: [0,1] \to \D$ such that $u^\nu(\alpha^\nu(0)) \in \im \gamma_j$ and $u^\nu(\alpha^\nu(1)) \in \im \gamma_k$ for some $j, k \in \{1,\dots, m\}$. Note that convergence of $u^\nu \to (v,v')$ as $\nu \to \infty$ implies
\begin{align*}
\lim_{\nu \to \infty} u^\nu(\alpha^\nu(0)) = \lim_{\nu \to \infty} u^\nu(\alpha^\nu(1)).
\end{align*}
Define a path $\beta: [0,1] \to L$ be defined by
\begin{align*}
\beta(\nu) = \begin{cases}u^{\tan(\nu)} (\alpha^{\tan{\nu}(0))} &\nu \in [0,\frac12)\\
\lim_{\nu \to \infty} u^\nu(\alpha^\nu(0)) &\nu = \frac12\\
u^{-\tan(\nu)} (\alpha^{-\tan{\nu}(0))}  &\nu \in (\frac12, 0].
\end{cases}
\end{align*}
Then $\beta(0) \in \im(\gamma^0_j)$, and $\beta(1) \in \im(\gamma^0_k)$.

For any two paths $\gamma$, $\beta$ on an oriented surface $S$ that intersect, intersection points can be assigned an orientation sign (which is used to compute the algebraic intersection number) as follows: an intersection point $q = \gamma(t_1) = \beta(t_2)$ is a positive intersection if $(\gamma'(t_1), \beta'(t_2))$ forms a positive basis for $T_q S$.
As $u$ is an aligned disk, the boundary arcs $\gamma_j$ and $\gamma_k$ have the same sign according to Definition \ref{defn_aligned disk}. This means that the intersection point $\beta(0)$ of $\beta$ and $\gamma_j$, and the intersection point $\beta(1)$ of $\beta$ and $\gamma_k$, have opposite orientation signs. 
Let $\wt \beta$, $\wt \gamma_j$, and $\wt \gamma_k$ be the lifts of $\beta$, $\gamma_j$, and $\gamma_k$, respectively, to $\wt L$.
Then, the intersection point $\wt \beta(0)$ of $\wt \beta$ and $\wt \gamma_j$, and the intersection point $\wt \beta(1)$ of $\wt\beta$ and $\wt\gamma_k$, have opposite orientation signs, and so do the intersection points $\wt \beta \cap \wt \gamma^\nu_j$ and $\wt \beta \cap \wt \gamma^\nu_j$. So, $\wt \gamma^\nu_j$ and $\wt \gamma^\nu_k$ in $\D_\epsilon$ are in the configuration of Case $1$ in Figure \ref{fig:switching operation}. Therefore, $\hat{ \delta}_j$ and $\hat{\delta}_k$ do not intersect, and so the isotopes $\wt \delta_j$ and $\wt \delta_k$ intersect in zero points counted with orientation signs.

This implies that if $\delta_j$ and $\delta_k$ intersect, they do so at even number of points away from their end points. Say at $r_1, \dots, r_{2k}$.
Note that ``we can glue" near each intersection point $r_j$ to get a sequence of holomorphic disks in $\M(u)$ converging to the stable disk that has $v,v'$ as part of it. That is, we can do standard gluing technique like , for example, \cite{liu}, to get a sequence $[{\bf u}_n^j] \in \M([u])$ that Gromov converges as $n \to \infty$ to a stable disk ${\bf v}^j$ which has the other points $r_l, l \neq j$ as nodes. This means that the stable disk with smooth node $q$ lies in a codimension 2 subspace of the moduli space $\M(u)$. As $\M(u)$ is a $1$-dimensional manifold, this is not possible.
\end{proof}

\subsection{Constraints on Sign at Corner Nodes}\label{sec_constraints on signs}

In this section we prove Part 4 of Theorem \ref{thm_properties of type 2 disks} as the following lemma. As we already showed that the only relevant nodes are corner nodes, the following lemma suffices.

\begin{lem}[\textsc{Adjacent disks have opposite signs}]\label{lem_opposite sign}
Suppose a stable disk \begin{align*}
(\{u_\alpha\}_{\alpha \in T}, \{z_{\alpha \beta}\}_{\alpha E \beta}, \{z_i, \alpha_i\}_{1\leq i \leq m})
\end{align*}
 is the Gromov limit of a sequence, $(u_n, \{z^n_i\}_{1\leq i \leq m})_{n \in \mathbb{N}}$ of holomorphic disks with corners that has boundary on Lagrangian tangle $L$. Suppose $z_{\alpha \beta}$ is a nodal point mapped to a double point of $L$, that is,  
 \begin{align*}
 u_\alpha(z_{\alpha \beta}) = u_\beta (z_{\beta \alpha}) = q \in \Delta(L) \subset \partial L,
 \end{align*}
 and it is a corner node. 
 Then, 
 \begin{align*}
 \sign_{q}(u_\alpha) = - \sign_{q}(u_\beta).
 \end{align*}
\end{lem}

\begin{proof}
It is enough to show this for the case when $u_\alpha$ and $u_\beta$ are embedded. Indeed, if they are only immersed, we can restrict ourselves to a small enough neighbourhoods of $z_{\alpha \beta}$ and $z_{\beta \alpha}$, and repeat the same argument in those neighbourhoods. The notation and argument gets more cumbersome.

The main idea is that if we restrict our vision to a small enough neighbourhood of $q$, the situation looks exactly like that of converging strips in Lagrangian Floer homology setup. 

Choose a neighbourhood $q \in U \subset \R^4$ such that $(U \cap L) \setminus \{q\}$ has two connected components and $((U \cap L) \setminus \{q\} )\cap \Delta(L) = \emptyset$. Let us call these connected components union $\{q\}$ as $ L_1$ and $L_2$. That is, $U \cap L = L_1 \cup L_2$ and $L_1 \cap L_2 = \{ q \}$. We also take $U$ to be small enough so that $u_\alpha^{-1}(U)$ and $u_\beta^{-1} (U)$ have no special points other than $z_{\alpha\beta}$ and $z_{\beta\alpha}$ respectively, and $u_\alpha^{-1}(U) \setminus \{z_{\alpha \beta}\}$ and $u_\beta^{-1} (U) \setminus \{z_{\beta \alpha}\}$ are biholomorphic to the strip $\R \times [-1, 1] \subset \C$.
As $u_n$ converge to $(u_\alpha, u_\beta)$ in the Gromov sense, their images converge in $C^\infty$ to the union of the images, $\im (u_\alpha) \cup \im(u_\beta)$. So, by choosing $U$ to be small enough, we may assume $u_n^{-1} (U)$ contains no marked points and are biholomorphic to $\R \times [-1, 1] \subset \C$ as well. Take coordinates $t,s$ on the strip with $t \in \R$ and $s \in [-1, 1]$.

Let us define compact sets $K_\alpha$ and $K_\beta$ as follows. Denote by $\partial U$ the boundary of the closure of $U$. Note that we can assume sufficient regularity on the neighbourhood $U$, in particular, we may assume it is a domain in $\C \times \C$. Let $K_\alpha = \overline{ \partial U \cap \cO p (\im (u_\alpha))}$ and $K_\beta = \overline{ \partial U \cap \cO p (\im (u_\beta))}$ where $\cO p(\im (u_\alpha))$ and  $\cO p(\im (u_\alpha))$ are arbitrary, small neighbourhoods of $\im(u_\alpha)$ and $\im(u_\beta)$, respectively. We choose the open sets $\cO p(\im (u_\alpha))$ and $\cO p(\im (u_\beta))$ to be small enough that $K_\alpha \nsubseteq K_\beta$ and $K_\beta \nsubseteq K_\alpha$.

By restricting to large enough $n$ and pre-composing by appropriate biholomorphisms to the strip $\R \times [-1, 1] \subset \C$, we can assume we have $\tilde{u}_n: \R \times [-1, 1] \to \R^4$ for $n \in \N$, $n > N$ for some large $N$, with 
\begin{align*}
\lim_{t \to -\infty} \wt{u}_n(t,s) \in K_\alpha, \lim_{t \to \infty} \wt{u}_n(t,s) \in K_\beta;\\
\wt u_n(\R \times \{-1\}) \subset L_1, \wt u_n(\R \times \{1\}) \subset L_2.
\end{align*}
Of course, the $L_1$ and $L_2$ may be exchanged, but then it would be so for every $n$.

Additionally, $\tilde{u}_\alpha: \R \times [-1, 1] \to \R^4$ and $\tilde{u}_\beta: \R \times [-1, 1] \to \R^4$ with
\begin{align*}
\lim_{t \to \infty} \tilde{u}_\alpha (t,s) = q, \lim_{t \to -\infty} \tilde{u}_\alpha(t,s) \in K_\alpha;\\
\lim_{t \to \infty} \tilde{u}_\beta(t,s) \in K_\beta, \lim_{t \to -\infty} \tilde{u}_\beta (t,s) = q.
\end{align*}
Note that, here we have made some specific choices of biholomorphisms from \\
$\tilde{u}_\alpha^{-1} (U), \tilde{u}_\beta^{-1} (U)$, and $\tilde{u}_n^{-1}(U)$ to $\R \times [-1, 1]$. Our conclusions do not depend on this choice as any biholomorphism between domains (with boundary) of $\C$ preserves the orientation on the boundary. 

Now, the original Gromov convergence restricted to the neighbourhood $U$ is the Gromov convergence of $\tilde{u}_n$ to the pair $\tilde{u}_\alpha, \tilde{u}_\beta$ while preserving the above limit restrictions. So, the only allowed reparametrizations are translations of $\R \times [-1, 1]$ by real numbers in the $t$ coordinate. Actually there may be small translations in the $s$-direction as well, but we may ignore these for the purpose of this argument. We get this condition by keeping in mind the marked point conditions on the original convergence. Indeed, if both $u_\alpha$ and $u_\beta$ have two or more marked points each, we may view them as maps on strips already, and we readily get the limit conditions above by restricting to appropriate half strips. If they have less than two marked points, note that the allowed parametrizations are restricted because $u_\alpha$ and $u_\beta$ are non-constant. 
So, we can still get the above condition by adding an artificial marked point to keep track of these reparametrizations and can choose one (say $\tilde{z}_\alpha$ and $\tilde{z}_\beta$, respectively) appropriately, such that any path from $z_\alpha$ to $\tilde{z}_\alpha$ has to pass through $K_\alpha$. Then, if we set $z_\alpha$ to be the limit of this path at $\infty$ and $\tilde{z}_\alpha$ at $-\infty$ and want the chosen biholomorphism to preserve this linear order on the path, we get the required conditions on the biholomorphism.

Thus, from the convergence of $u_n$ and their boundary conditions, 
\begin{align*}
\tilde{u}_\alpha(\R \times \{-1\}) \subset L_1, \tilde{u}_\alpha(\R \times \{1\}) \subset L_2 \text{ and } \tilde{u}_\beta(\R \times \{-1\}) \subset L_1, \tilde{u}_\beta(\R \times \{1\}) \subset L_2.
\end{align*}
So, we see that $\partial \tilde{u}_\alpha$ jumps from $L_1$ to $L_2$ at $q$ and $\partial \tilde{u}_\beta$ goes from $L_2$ to $L_1$ at $q$. As, the orientation on the boundary was always preserved this implies $ \sign_{q}(u_\alpha) = - \sign_{q}(u_\beta)$.
 
\end{proof}

\subsection{Non-horizontal Disks have Fixed Sign at Corners}\label{sec_non horizontal disk sign} 
In this section, we show that corner signs for non-horizontal holomorphic disks are determined by which boundary of $L$ the corner is at, when $L$ is a Lagrangian tangle. This will give us the last part of Theorem \ref{thm_properties of type 2 disks}.
\begin{lem}\label{lem_signs of vert disks}
Let $L$ be a Lagrangian tangle.
Suppose $u$ is a non-horizontal holomorphic disk with corner at $q \in \partial_+ L$. Then $\sign_q(u) = +1$.
Similarly, $q \in \partial_- L$ implies $\sign_q(u) = -1$.
\end{lem}
Before we prove this lemma, let us discuss how we conclude Part 4 of Theorem \ref{thm_properties of type 2 disks}. Lemma \ref{lem_opposite sign} and the above Lemma \ref{lem_signs of vert disks} together give us that we cannot have two non-horizontal disks attached at a node. We claim that the case of two horizontal disks attached at a node also do not appear in our considered moduli space, $\ol \M(u)$. Suppose $\im(\p u) \subset \p_+ L = L_b$. In Section \ref{sec_dimension calculations}, we defined the real line bundle $T^1L$ along $\p \D$ as a subbundle of $\C \times \{0\} \subset \C \times \C \cong u^*(T\R^4)$ for the horizontal disk $u$. The Maslov index $\mu_1(u)$ of the bundle pair $(\C , T^1 L)$ will always be $2$ (for convex $u$) and $\mu_1(u)$ actually gets suppressed in the dimension calculations and we only see $\mu_2(u)$ in $\dim \M(u)$. Maslov index $\mu_1(u)=2$ implies the expected dimension of the moduli space of variations of $u$ within the plane is 
\begin{align*}
\dim(\mu(u, L_b)) = \mu_1(u) + 2 - 3 - 1 = 0.
\end{align*}
If we had two horizontal disks attached at a node $q \in \p_+ L$, we could glue within $\R^2_b$ to obtain a $1$-dimensional moduli space contained in $\M(u, L_b)$, which contradicts the dimension calculation. Thus, we can only have the situation where each node has exactly one horizontal disk and one non-horizontal disk attached, which is Part 4 of Theorem \ref{thm_properties of type 2 disks}. Now we prove the lemma.
\begin{proof}
First, we build a local model for $L$ near a double point $q$.
Suppose $q \in \partial_+L$. Let us assume that $\partial_+ L \subset \R^2_0$ and $q = 0 \in \p_+ L$.  Fix a small neighbourhood $U$ of $0$ in $\R^4$.  In $U$, the Lagrangian has linear approximation given by $(T_0 L^h \cup T_0 L^l) \cap \R^4_{-\infty, 0]}$, where $L^h$ and $L^l$ are the two leaves of $L$ near $q$. Upto a Hamiltonian isotopy, we may assume the Lagrangian is equal to this linear approximation in a small neighbourhood of $0$. We want to analyze what the Lagrangian $L$ looks like near $0$, so we analyze what two Lagrangian subspaces of $\R^4$ intersecting at one point can look like. Consider $L^h \cap (\R^2 \times \{0\})$. It is a line passing through $0$, let us call it $l_{11}$. It is given by an equation $a_1x_1+b_1y_1 = 0$ for $a,b \in \R$. Any Lagrangian plane containing this line lies in the symplectic complement of $\ii(l_{11})$. So, the other generator of $L_1$ must be a line in $\{0\} \times \R^2$, say $a_2x_2 + b_2x_2 = 0$, $a_2, b_2 \in \R$. Similarly, the subspace $T_0 L^l = \{c_1x_1 + d_1y_1 = 0, c_2x_2 + d_2y_2 = 0\}$, for some $c_1, d_1 \in \R$. To get a local model for the Lagrangian tangle $L$, we have to impose the condition $\{y_2 \leq 0\}$. 

Recall the notation $L^h$ is the higher leaf and $L^l$ is the lower leaf near $q = 0$ with respect to the $x_2$-coordinate.
For ease of notation we can fix coefficients without loss of generality as follows:
Locally, $L \cap U = L^h \cup L^l$, where
\begin{align*}
L^h = \{(x_1, y_1, x_2, y_2) \in U | y_1 =0, x_2 + y_2 = 0, y_2 \leq 0 \},\\
L^l = \{(x_1, y_1, x_2, y_2) \in U | x_1 =0, x_2 - y_2 = 0, y_2 \leq 0 \}.
\end{align*}

Now, we prove the lemma.
If $u: (\D, \partial \D) \to (\R^4, L)$ is a holomorphic map with corner at $q = 0$, then $\pi_2 \circ u: (\D, \partial \D) \to (\R^2, \pi_2(L))$ is also a holomorphic map with corner at $0$. Here $\pi_2: \R^4 \to  \R^2$ is the projection $(x_1, y_1, x_2, y_2) \mapsto (x_2, y_2)$. Note that $\pi_2(L \cap U) = \pi_2(L^h) \cup \pi_2(L^l)$ where
\begin{align*}
\pi_2(L^h) = \{(x,y) \in \R^2 | x+y = 0, y \leq 0\},\\
\pi_2(L^l) = \{(x,y) \in \R^2 | x-y = 0, y \leq 0\}.
\end{align*}
As $u$ is holomorphic, the boundary of $u$ has to travel anti-clockwise on the complex plane. So, if $\pi_2 \circ u$ is non-constant, then its boundary traverses $\pi_2(L^h)$ first and then $\pi_2(L^l)$. If $\pi_2 \circ u$ is constant, then $u$ is a horizontal disk. Thus, for non-horizontal disks $\sign_q(u) = +1$.

The case for $q \in \p_- L$ is analogous.
\end{proof}

\section{Applications: Obstructions to the Existence of Undercut Relation}\label{sec_applications}
In this section, we give applications of Theorem \ref{thm_main application} in the form of obstructions to undercutting relation on diagrams. We also give proofs of the corollaries mentioned in Section \ref{sec_intro}. Many similar results can be obtained by applying Theorem \ref{thm_main application} and we only present select few.

Recall that topological data gives  a priori restrictions on Lagrangian cobordisms. Let $\pi_1: \R^4 \to \R^2$ be the projection $(x_1, y_1, x_2, y_2) \mapsto (x_1, y_1)$. If there is a cobordism $K_1 \prec_L K_2$, then $\pi_1(K_1)$ and $\pi_1(K_2)$ must bound the same signed area and have the same rotation number. Further, let us denote by $\mathrm{wr}(K_1)$ the writhe of the diagram $\pi(K_1) = \cD_{K_1}$ with respect to the framing $\ker \omega|_{\R^3_a}$ also known as black board framing. That is, it is the sum of the number of  positive crossings minus the number of negative crossings as shown in Figure \ref{fig:crossing sign}, when viewed from positive infinity on the $x_2$-axis. Then, for $K_1 \prec_L K_2$
\begin{align}\label{eqn_euler char from writhe}
\mathrm{wr}(K_1) - \mathrm{wr}(K_2) = \chi(L)
\end{align}
where $\chi$ denotes the Euler characteristic (see \cite{ST}). An interesting observation is that Gromov's theorem stating that there are no closed weakly exact Lagrangians in $\R^4$, implies that there is no knot that can be both filled and capped by an exact Lagrangian. This would also follow from the above Euler characteristic computations.
\begin{figure}[h!]
\begin{center}
  \includegraphics[width = 2in, height=1in]{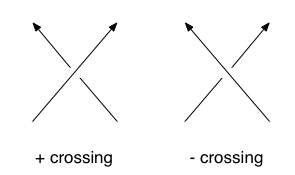}
\end{center}
\caption{Signs of crossings for calculating writhe.}
\label{fig:crossing sign}
\end{figure}

We want to study those situations where these topological constraints are inconclusive. For example, when $K_1$ and $K_2$ have the same writhes. We first prove Theorem \ref{thm_main application}, stated in the Introduction.
\begin{proof}
Suppose $[\cD_1,\sigma_1, \cA_1]\prec [ \cD_2, \sigma_2, \cA_2]$. This means that there exists a Lagrangian cobordism $L' \subset \R^4_{[a,b]}$ with $\p_+ L' = K_{\cD_2}$ and $\p_- L' = K_{\cD_1}$. Now, by using Lemma \ref{lem_tangle with same diagrams}, we get a Lagrangian tangle $L \subset \R^4_{[a,b]}$ with diagram of $\p_+ L = [\cD_2]$ and diagram of $\p_+ L = [\cD_1]$. This also means $\p_+ L = [\cD_2]$ and $\p_+ L = [\cD_1]$ as immersed curves. 

Suppose we have a big disk bound by the diagram pair. This means there exists a big disk bound by $\p_+ L$ or $\p_- L$ in the plane $\R^2_a$ or $\R^2_b$. In either case, we get $(u, {\bf z})$, a big, horizontal, holomorphic disk with corners, that has boundary on $L$. As $u$ is a big disk, $\ol\M(u)$ is a compact $1$-dimensional manifold (Theorem \ref{thm_automatic transversality}, Theorem \ref{thm_cptness}) with $[u] \in \ol \M(u)$ representing a boundary point of the moduli space (Part 1 of Theorem \ref{thm_properties of type 2 disks}). 

By Theorem \ref{thm_properties of type 2 disks}, Part 2, we must have even number of type 2 holomorphic disks as boundary points for $\ol \M(u)$. So there must exist a stable, holomorphic, Type 2 disk, $[{\bf u', z'}]$, distinct from $[u, {\bf z}]$.

If ${\bf u'}$ does not have any nodes, then it is a horizontal disk by Part 1 of Theorem \ref{thm_properties of type 2 disks}. In this case, from the definition of $\M(u)$, all the corners of $u'$ are the same as those of $u$, with the same signs.

If $({\bf u', z})$ is a stable disk with nodes, by Theorem \ref{thm_properties of type 2 disks}, the stable disk of type 2  would have at least one horizontal disk, say $v: (\D, \p \D) \to (\R^4, L)$. We claim that $v$ gives us a little disk to $u$ for the considered diagram pair. To prove the claim, take the case when $\im(u) \subset \R^2_b$. Then, $u$ has all positive corners. 
\begin{itemize}
\item If $\im (v) \subset \R^2_b$, then each corner $q$ of $v$ is either a corner point shared with $u$ or it is a node of the stable holomorphic disk that $v$ is a component of. If $q$ is a node, then it is a node of a stable holomorphic disk of type 2. So, it has to be a node between $v$ and a non-horizontal disk $v'$. As $v'$ is not horizontal, and $q \in \p_+ L$, by Lemma \ref{lem_signs of vert disks}, $\sign_q(v') = +1$. So, by Lemma \ref{lem_opposite sign}, $\sign_q (v) = - \sign_q (v') = -1$. If $q$ is shared by $v$ and $u$, then it is a corner point of the stable holomorphic disk and by virtue of the definition of $\ol \M(u)$, has the same sign as $\sign_q (u)$. So, $\sign_q (v) = +1$. 
\item In the case $\im(v)\subset \R^2_a$, all its corners are nodes of a stable holomorphic disk of type 2. Each node has an attached non-horizontal disk that must have sign equal to $-1$ at the node. So, all of $v$'s corners are positive by Part 3 of Theorem \ref{thm_properties of type 2 disks}. For the case when $
im (u) \subset \R^2_b$, the argument is analogous.
\end{itemize}

The area inequality comes from the fact that every holomorphic disk must have strictly positive area. Suppose, $v$ is the little disk of $u$ that appears as a horizontal disk of $({\bf u', z'})$, the stable holomorphic disk of Type 2  that represents a boundary point of $\ol \M(A)$. Then, 
\begin{align*}
\sum_{\alpha \in I} \omega([u'_\alpha]) = \omega \left(\sum_{\alpha \in I} [u'_\alpha]\right) = \omega (A) = \omega([u]).
\end{align*}
As each $u'_\alpha$ is a holomorphic disk, $\omega([u'_\alpha]) > 0$. So, we get that
\begin{align*}
\omega([v]) = \omega([u]) - \sum_{\alpha \in I, u'_\alpha \neq v} \omega([u'_\alpha]) < \omega([u]).
\end{align*}
The only exception to the above inequality happens when $[v] = [u] \in H_2(\R^4, L)$. 
\end{proof} 

Thoerem \ref{thm_main application} gives us many exciting corollaries. We remark here again that these are only example applications, not by any means an exhaustive list but instead a select few. 

Note that Corollaries \ref{cor_figure 8 knot}, \ref{cor_trefoil knot}, and \ref{cor_no relation exists}, is stated in terms of the undercut relation. This means we are proving that there cannot exist relatively exact Lagrangian cobordisms with the given boundary conditions, but in all of these examples the observation in Remark \ref{rem_automatic relative exactness} implies that any possible Lagrangian cobordism will be relatively exact. So, in fact, we are showing that no Lagrangian cobordisms can exist with these boundary enriched knot diagrams.
\begin{cor}\label{cor_figure 8 knot}
For $A, B > 0$, we observe the following growing and shrinking behaviour among enriched knot diagrams. Suppose $8_-$, $8_+$, $E_-$, and $E_+$ denote exact enriched knot diagrams as in Figures \ref{fig:8knots}, and \ref{fig:trefoils}.
\begin{enumerate}
\item If $8_-(A) \prec 8_-(B)$, then $A > B$;
\item If $8_+(A) \prec 8_+(B)$, then $A < B$;
\item If $E_-(A) \prec E_-(B)$, then $A > B$;
\item If $E_+(A) \prec E_+(B)$, then $A < B$
\end{enumerate}
Here, assume that the total areas of all the diagrams is $0$.
\end{cor}
\begin{figure}[h!]
\begin{center} 
  \includegraphics[width = 5in]{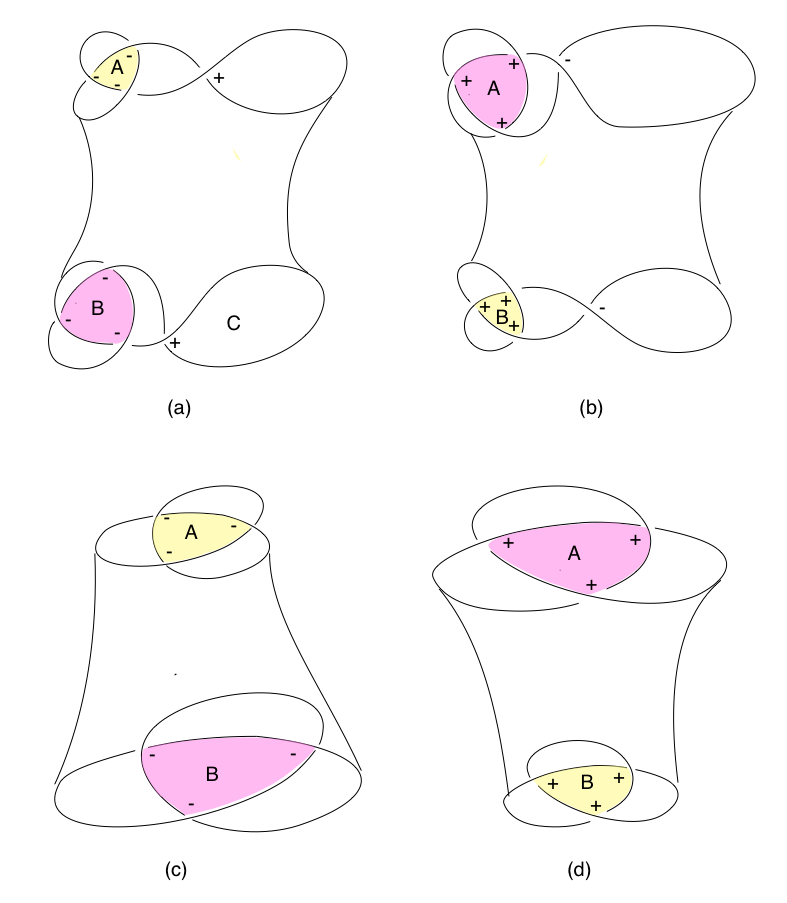}
   \end{center}
 \caption[caption]{\tabular[t]{@{}l@{}}In all the figures the pink disk has to be larger than the yellow disk:\\ (a) $E_-(B) \prec E_-(A)$ (b) $E_+(B) \prec E_+(A)$\\ (c) $T_-(B) \prec  T_-(A)$ (d) $T_+(B) \prec T_+(A)$\endtabular}
\label{fig:trefoils}
\end{figure}
\begin{proof}
\begin{enumerate}
\item[(1,2)]
Suppose $8_-(A) \prec 8_-(B)$. The diagram $8_-(A)$ cuts out two disks of same area $A$, and both are big disks. Let us pick one of them to be $A_1$ to apply Theorem \ref{thm_main application}. All the other disks cut out by the pair, $(8_-^1(A), 8_-^1(B))$, are little disks to $A_1$. Let the other disk cut out by the lower diagram be $A_2$. Let the two disks cut out by the upper diagram be $B_1$ and $B_2$. 

Suppose $L \subset \R^4_{[a,b]}$ is a Lagrangian tangle with $\p_+ L = \cD_{8_-(B)}$ and $\p_- L = \cD_{8_-(A)}$
Note that the writhes of both the diagrams are the same, which means the cobordism will be a cylinder. This means that the tangle $L$ is topologically a sphere minus four points.
So, even though $A_2$ is a little disk to $A_1$, $\area(A_1) = \area(A_2)$ but $[A_1] \neq [A_2] \in H_2(\R^4, L)$. So, we must have $\area(B_1) = \area(B_2) < \area(A_1)$, which implies $A > B$.

The proof for the $8_+$ case is analogous.
\item[(3,4)] If $E_-(B) \prec E_-(A)$, then $B$ is a big disk. The candidates for little disk for $B$ are $A$ and $C$ as labelled in Figure \ref{fig:trefoils} (a). But as the total bounded area of $E_-(B)$ is zero, $C > B$. So, $A < B$ by Theorem \ref{thm_main application}.

The proof for the $E_+$ statement is analogous.
\end{enumerate}
\end{proof}

We also get obstructions on diagrams that ae not exact.
\begin{cor}\label{cor_trefoil knot}
For $A, B > 0$,  the following growing and shrinking behaviour among trefoil knots hold.
\begin{enumerate}
\item  If $T_-(A) \prec T_-(B)$, then $A > B$;
\item If $T_+(A) \prec T_+(B)$, then $A < B$.
\end{enumerate}
\end{cor}
\begin{proof}
Note that $B$ is a big disk bound by $T_-(B)$. In particular, $B < a(T_-(B))$ as all the disks cut out by $T_-(B)$ contribute positively to $a(T_-(B))$.
$A$ is the only possible little disk to $B$. So, we must have $A<B$.
The proof for $T_+$ case is analogous.
\end{proof}

\begin{cor}\label{cor_area conditions}
Following are examples of obstructions to undercutting when the two enriched knot diagrams do not have the same underlying (topological) knot diagram. We assume that all the enriched knot diagrams in this corollary are exact. (See Figure \ref{fig:defn_diagrams}.)
\begin{enumerate}
\item If $\emptyset \prec C^{-++}(A_1,A_2,A_3)$ for $C^{-++}$ as in Figure \ref{fig:defn_diagrams} with total area equal to zero, then $$A_3 > A_2.$$
\item Let $A_1 > A_2 > A_3 > 0$. If $C^{-++}(B_1, B_2, B_3) \prec C^{-++}(A_1, A_2, A_3)$, then $$A_3 > \min(B_3, B_4),$$ where $B_4 = B_1 - B_2 + B_3$ is the area of the unmarked lobe. 
\item For $A_1 > A_2 > A_3>0$, and $B>0$, if $8_+(B) \prec C^{-++}(A_1, A_2, A_3)$, then 
\begin{align*}
A_3 > B.
\end{align*}
\end{enumerate}
\end{cor}
\begin{proof}
In all of these cases, the disk with area $A_3$ is a big disk. We make the above conclusions by tracking through possible little disks and obtain these area conditions.
\end{proof}

\begin{cor}\label{cor_no relation exists}[Lemma \ref{lem_partial order}, Part 3]
There exist diagrams that are unrelated. 
For example, for $0 < A < B$, $8_+(A) \nprec C^{+-+}(A, B, B)$ and $C^{+-+}(A, B, B) \nprec 8_+(A)$. Thus, $\prec$ is a partial order.
\end{cor}
\begin{proof}
Refer Figure \ref{fig:no relation}.
For the pair $(8_+(A), C^{+-+}(A, B, B))$, the upper lobe labelled $A_1$ is a big disk. The two possible little disks are labelled $B_1$, $A_3$, and $A_4$. But none of these have area strictly less than $A$. Therefore, this is not an undercutting pair.

For the pair $(C^{+-+}(A, B, B), 8_+(A))$, $A_3$ is a big disk. The little disks for $A_3$ are $A_1$, $A_2$, and $A_4$. None of these have area strictly less than $A$.
\begin{figure}[h!]
\begin{center}
  \includegraphics[width = 5in]{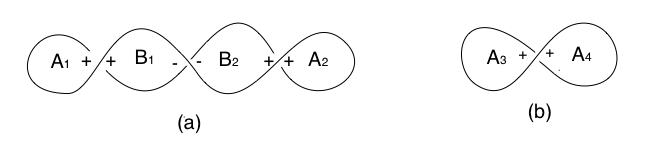}
   \end{center}
 \caption{For $A_1 = A_2 = A_3 = A_4$, $B_1 = B_2$ as areas: (a) $C^{+-+}(A,B,B)$, (b) $8_+(A)$.}
  \label{fig:no relation}
\end{figure}

\end{proof}

\medskip

\bibliographystyle{alpha}
\bibliography{references}

\end{document}